\documentclass[12pt,a4paper,tikz,border=3mm]{book}%
\usepackage[USenglish]{babel} %francais, polish, spanish, ...
\usepackage[T1]{fontenc}
\usepackage[ansinew]{inputenc}
\usepackage{a4wide}
 \usepackage{stmaryrd}
 \usepackage{subfiles}

\usepackage[toc,page]{appendix}

\usepackage{amsmath,amsthm,amsfonts,amssymb}%
\usepackage{bm}
\usepackage{graphicx}
\usepackage[all,cmtip]{xy}
\usepackage{tikz-cd}
\usetikzlibrary{arrows}
\usepackage[nottoc]{tocbibind}

\pgfdeclarelayer{background}
\pgfsetlayers{background,main}

\usepackage[margin=3cm]{geometry}
\usepackage{mathabx,epsfig}

\def\acts{\mathrel{\reflectbox{$\righttoleftarrow$}}}

\theoremstyle{plain}
\newtheorem{theorem}{Theorem}[section]
\newtheorem{lemma}[theorem]{Lemma}
\newtheorem{proposition}[theorem]{Proposition}
\newtheorem{corollary}[theorem]{Corollary}

\theoremstyle{definition}
\newtheorem{definition}[theorem]{Definition}
\newtheorem{example}[theorem]{Example}
\newtheorem{notation}[theorem]{Notation}
\newtheorem{remark}[theorem]{Remark}
\newtheorem{conventions}[theorem]{Conventions}

\DeclareMathOperator {\Hom}{Hom}

\DeclareMathOperator {\Map}{Map}

\DeclareMathOperator {\coker}{coker}
\DeclareMathOperator {\Aut}{Aut}

\DeclareMathOperator {\Tor}{Tor}

\def \G {\bbT \ltimes \widetilde{LG}}

\def \Z {\mathbb{Z}}
\def \C {\mathbb{C}}

\def \Q {\mathbb{Q}}
\def \D {\mathbb{D}}
\def \R {\mathbb{R}}
\def \bT {\mathbb{T}}

\def \bH {\mathbb{H}}
\def \bD {\mathcal{D}}

\def \< {\langle}
\def \> {\rangle}

\def \Ell {\mathcal{E}\ell\ell}

\def \S {\mathcal{S}}

\def \O {\mathcal{O}}

\def \N {\mathbb{N}}

\def \E {\mathcal{E}}

\def \F {\mathcal{F}}
\def \U {\mathcal{U}}

\def \L {\mathcal{L}}

\def \H {\mathcal{H}}

\def \G {\mathcal{G}}
\def \K {\mathcal{K}}

\def \X {\mathcal{X}}

\def \I  {\mathcal{I}}

\def \to {\rightarrow}

\def \t {\mathfrak{t}}

\def \H {\mathcal{H}}

\def \SL {\mathrm{SL}}
\def \GL {\mathrm{GL}}
\def \pt {\mathrm{pt}}

\def \id {\mathrm{id}}

\begin{document}
\frontmatter
%titlepage
\thispagestyle{empty}
\begin{center}
\begin{minipage}{0.75\linewidth}
    \centering
%University logo
    %\includegraphics[width=0.3\linewidth]{PRIMARY_A_Vertical_Housed_RGB.png}
   % \rule{0.4\linewidth}{0.15\linewidth}\par
   % \vspace{3cm}
%Thesis title
 \vspace{4cm}
    {\uppercase{\Large Comparison theorems for torus-equivariant elliptic cohomology theories\par}}
    \vspace{2cm}
%Author's name
    {\Large Matthew James Spong\par}
    {\small ORCID iD: https://orcid.org/0000-0001-6339-1211\par}
    \vspace{2cm}
%Degree
    {\Large A thesis submitted in total fulfilment of the requirements for the degree of Doctor of Philosophy\par}
    \vspace{2cm}
%Date
    {\Large February, 2019\par}
    {\Large School of Mathematics and Statistics\par}
    {\Large The University of Melbourne}
\end{minipage}
\end{center}

% especially in a document where chapters start at right-hand pages
%\phantomsection% for an anchor if you use hyperref
\chapter*{Abstract}
\chaptermark{}% for the actuall unnumbered heading
In 1994, Grojnowski gave a construction of an equivariant elliptic cohomology theory associated to an elliptic curve over the complex numbers. Grojnowski's construction has seen numerous applications in algebraic topology and geometric representation theory, however the construction is somewhat ad hoc and there has been significant interest in the question of its geometric interpretation.

We show that there are two global models for Grojnowski's theory, which shed light on its geometric meaning. The first model is constructed as the Borel-equivariant cohomology of a double free loop space, and is a holomorphic version of a construction of Rezk from 2016. The second model is constructed as the loop group-equivariant K-theory of a free loop space, and is a slight modification of a construction given in 2014 by Kitchloo, motivated by ideas in conformal field theory. We investigate the properties of each model and establish their precise relationship to Grojnowski's theory.

\chapter*{Declaration}
\chaptermark{}% for the actuall unnumbered heading
The author declares that: \par
(i) this thesis comprises only the original work of the author towards the degree of Doctor of Philosophy; \par
(ii) due acknowledgement has been made in the text to all other material used; and \par
(iii) this thesis is fewer than 100,000 words in length, exclusive of tables, maps, bibliographies and appendices. \par

Signature of the author: 

\clearpage
%\pagenumbering{roman}
%\clearpage

\begin{center}
    \vspace*{\fill}
    \textit{To Mum and Dad}
    \vspace*{\fill}
\end{center}
%\clearpage

%\addtocontents{toc}{chapter}{Acknowledgements}% if you wish to have a TOC entry
\chapter*{Acknowledgements}
\chaptermark{}% for the actuall unnumbered heading
%\thispagestyle{empty}% or plain etc.
%\markboth{Acknowledgements}{Acknowledgements}
I thank my supervisor, Nora Ganter, for all of her of guidance and encouragement, and for sharing with me her numerous insights. I thank the members of my supervisory panel, which included Marcy Robertson, Arun Ram and Yaping Yang, for their advice and support. In particular, I would like to thank Marcy for meeting with me often. I owe a very large debt to Nitu Kitchloo and Charles Rezk, as this thesis is based almost entirely on their ideas, and I thank them for explaining their work to me. I am especially grateful to Nitu for his generosity during my stay in San Diego. I received much inspiration from the work of Ioanid Rosu, who certainly deserves my thanks. I also thank Gufang Zhao, Zhen Huan, Christian Haesemeyer, Matthew Ando and Paul Zinn-Justin for useful conversations. Finally, I would like to express my deep gratitude to my partner, to my parents, and to the rest of my family for their patience and loving support. \par

This research was supported by an Australian Government Research Training Program (RTP) Scholarship.

\tableofcontents
\mainmatter
%\pagenumbering{arabic}

\chapter{Introduction} 

Ever since the desire to realise the Witten genus as a map of cohomology theories, there have been two major driving forces in elliptic cohomology. On one hand, there have been purely formal constructions, such as the first definition of elliptic cohomology in \cite{LRS}, and more recently \cite{AHS}, \cite{Hopkins}, and \cite{Gr}. On the other hand, there have been developments in the geometric interpretation of elliptic cohomology, from Segal's notion of an elliptic object in \cite{Segal1}, and continuing notably in \cite{ST}, and in \cite{KrizHu}. It is only now that, in the equivariant setting, a precise comparison between these two pictures is emerging (\cite{BET}, \cite{Huan}). The paper at hand is a contribution in this direction.

Equivariant elliptic cohomology has two traditions, one which has been focused on finite group actions, and the other on compact Lie group actions. In the finite case, Devoto \cite{Devoto} gave a definition using equivariant K-theory, which was further developed by Ganter \cite{Nora1} in the context of orbifold loop spaces. The finite case is related not only to physics (see Berwick-Evans \cite{BE}) but also moonshine phenomena (see Ganter \cite{Nora1}, \cite{Nora2} and Morava \cite{Morava}). In the Lie case, Grojnowski gave a purely formal, but ad hoc construction of a complex analytic equivariant elliptic cohomology for the purpose of constructing elliptic affine algebras \cite{Groj}. Although it is defined over the complex numbers, Grojnowski's construction has attracted a lot of interest and is continuing to find numerous applications (see for example \cite{AO}, \cite{Ando}, \cite{GKV}, \cite{Rezk}, and \cite{Rosu}). We are interested in a geometric model for Grojnowski's construction in the case that the Lie group is a torus.

%make clear to the reader that she does not need to understand all of the work cited in order to read this paper. Say that this paper is more accessible than most on elliptic cohomology theory.

There are striking similarities between the formal properties of elliptic cohomology and the representation theory of loop groups (see for example \cite{Ando}, \cite{Nora}, and \cite{KM}). Indeed, much of the work on elliptic cohomology was stimulated by the relationship (\cite{W1}, \cite{W2}) between elliptic genera and the index of the Dirac operator on a free loop space, where loop groups naturally arise. However, until now a precise connection has been difficult to state.

Motivated by the strong relationship between the representation theory of loop groups and conformal field theory (see section 4 of \cite{Segal2}), Kitchloo has constructed in \cite{Kitch1} a version of $G$-equivariant elliptic cohomology using loop group techniques, for a simple, simply-connected compact Lie group $G$. In this paper, we spell out Kitchloo's construction $\F_T$ for a torus $T$ and compare this to Grojnowski's definition $\G_T$, and find that the two are equivalent on compact, equivariantly formal\footnote{For the definition of equivariantly formal, see Definition \ref{eqfor}} $T$-spaces, for a particular elliptic curve. We thus establish a precise link between Kitchloo's physical definition and Grojnowski's formal definition. Along the way, we find that the formalism of Kitchloo's definition may be used to construct a holomorphic version $\Ell_T$ of a $T$-equivariant elliptic cohomology theory defined by Rezk in \cite{Rezk} (see Section 5). Following Rezk's ideas, we construct $\Ell_T$ over the moduli space of framed elliptic curves in such a way that it is equipped with a natural $\C^\times \times \SL_2(\Z)$-action. We then show that, over a particular elliptic curve, $\Ell_{T}$ is equivalent to $\G_T$ for all compact $T$-spaces.

There have been several recent contributions to the geometric interpretation of elliptic cohomology. Berwick-Evans and Tripathy give (\cite{BET}) a physical interpretation of complex analytic equivariant elliptic cohomology which unifies the discrete and Lie group points of view, inspired by the work of Stolz and Teichner \cite{ST}. Also, in \cite{Huan}, Huan constructs a quasi-elliptic cohomology theory, modeled on orbifold loop spaces. % Check this

We now state our main results, leaving until afterwards the task of explaining the notation. 

\begin{theorem}[Corollary \ref{glutes}]
There is an isomorphism of cohomology theories 
\[
\Ell^*_{T,t} \cong \G^*_{T,t},
\]
on finite $T$-CW complexes with values in coherent holomorphic sheaves of $\Z/2\Z$-graded $\O_{E_{T,t}}$-algebras.
\end{theorem}

\begin{theorem}[Corollary \ref{gluten}]
Let $q = e^{2\pi i\tau}$. There is an isomorphism of cohomology theories 
\[
\F^*_{T,q} \cong (\sigma_{T,\tau})_*\, \G^*_{T,\tau}
\]
on equivariantly formal, finite $T$-CW complexes with values in coherent holomorphic sheaves of $\Z/2\Z$-graded $\O_{C_{T,q}}$-algebras.
\end{theorem}

\begin{theorem}[Corollary \ref{wheat}]
Let $q = e^{2\pi i\tau}$. There is an isomorphism of cohomology theories 
\[
^k\F^*_{T,q} \cong  (\sigma_{T,\tau})_*\, (\G^*_{T,\tau}) \otimes_{\O_{C_{T,q}}} \L^k_q
\]
on equivariantly formal, finite $T$-CW complexes with values in coherent holomorphic sheaves of $\Z/2\Z$-graded $\O_{C_{T,q}}$-modules. 
\end{theorem}

The parameters $t,q$ and $\tau$ denote the restriction to particular complex elliptic curves $E_t$, $C_q$ and $E_\tau$. We write 
\[
E_{T,t} := \check{T} \otimes_\Z E_t \quad \text{and} \quad  C_{T,q} := \check{T} \otimes_\Z C_q
\]
for the underlying geometric objects of the respective theories. If $q = e^{2\pi i\tau}$, then there is an isomorphism $\sigma_{\tau}: E_\tau \cong C_q$ which induces 
\[
\sigma_{T,\tau}: \check{T} \otimes_\Z E_\tau \cong \check{T} \otimes_\Z C_q.
\]
It is the pushforward over this map which appears in the second and third results. The symbol $k$ represents a positive integer which arises as a parameter in the representation theory of loop groups, out of which $^k\F^*_{T,q}$ is constructed. Finally, the notation $\L_q^k$ denotes a certain line bundle over $C_{T,q}$, sometimes known as the \textit{Looijenga line bundle}, whose sections are degree $k$ theta functions. All sheaves are holomorphic.

We take the opportunity here to establish some conventions.

\begin{conventions}\label{convent}
Suppose that a group $G$ acts on a space $X$ from the left. We use $g \cdot x$ to denote the group action of $g\in G$ on $x \in X$, and we use $gg'$ to denote the group product of $g,g' \in G$. All group actions are assumed continuous, and all subgroups of Lie groups are assumed to be closed \par

If $G$ is a topological group, we denote the connected component of $G$ containing the identity by $G^0$. We will often just use the word \lq component' to mean a connected component. \par

All maps of topological spaces are assumed to be continuous. If $X$ and $Y$ are topological spaces, then the set of continuous maps $\Map(X,Y)$ is regarded as a space with the compact-open topology. \par

All vector bundles are complex vector bundles, and all sheaves are holomorphic.\par

Let $A$ be an abelian group, let $H$ be an arbitrary group, and let $A$ act on $H$. Our convention for the group law of the semidirect product $A \ltimes H$ is
\[
(a',h')(a,h) = (a'a, a^{-1}\cdot h' h).
\]
The tensor product $A \otimes B$ of two $\Z$-modules is over $\Z$, unless otherwise specified. All rings are assumed to have a multiplicative identity. By a $\Z$-graded commutative ring we mean a $\Z$-graded ring $R$ such that for two homogeneous elements $a \in R^i$ and $b \in R^j$, we have
\[
a b = (-1)^{ij}ba.	
\]
For (not necessarily square) matrices $A$, $m$, and $t$, we use expressions such as $Am$, $mt$, and $mA$ to mean matrix multiplication. So, for example, if $m = (m_1,m_2)$ and $t = (t_1,t_2)$ are vectors, then $mt$ means the dot product, where $t$ is understood to mean the transpose of $t$. The transpose of a vector should be understood whenever it is necessary to make sense of an expression.
\end{conventions}

\chapter{Background}\label{ChBack}

In this chapter, we set out the background material that will be used in the following chapters. In the first section we treat a variety of basic objects, the most important among which is perhaps the moduli stack of elliptic curves over $\C$, which we present as an equivariant space. In the second and third sections we introduce torus-equivariant ordinary cohomology and torus-equivariant K-theory, along with some of their properties and holomorphic analogues. In the fourth section we introduce Rosu's version of the equivariant Chern character (see \cite{Rosu03}), which will be immensely useful to us in comparing various cohomology theories. In the final section, we give the construction of Grojnowski's torus-equivariant elliptic cohomology, which is central to all of our results.

\section{Elliptic curves over $\C$}

Our account of the classification of elliptic curves over $\C$ is based on the short summary in Rezk's paper \cite{Rezk}, which is particularly convenient for our constructions in the next chapter. 

\begin{remark}
Consider the subspace 
\[
\X := \{ (t_1,t_2) \in \C^2 \, | \, \R t_1 + \R t_2 = \C \} \subset \C^2.
\]
An element $t = (t_1,t_2) \in \X$ defines a lattice
\[
\Lambda_t := \Z t_1 + \Z t_2 \subset \C.
\]
It is easily verified that $\X$ is preserved under left multiplication by $\GL_2(\Z)$, and that $\Lambda_t = \Lambda_{t'}$ if and only if there is a matrix $A \in \GL_2(\Z)$ such that $At = t'$. 
\end{remark}

\begin{definition}
An \textit{elliptic curve over $\C$} is a complex manifold
\[
E_t := \C/\Lambda_t,
\]
along with the quotient group structure induced by the additive group $\C$. 
\end{definition}

\begin{remark}
A map of elliptic curves $E_t \to E_{t'}$ is a map of complex manifolds which is also a homomorphism of groups. It is straightforward to show that any map of elliptic curves $E_t \to E_{t'}$ is induced by multiplication by a nonzero complex number $\lambda$ satisfying that $\lambda \Lambda_t \subset \Lambda_{t'}$. Furthermore, the map is an isomorphism if and only if $\lambda \Lambda_t = \Lambda_{t'}$. 
\end{remark}%reference

\begin{remark}\label{dominic}
By our previous remarks, for any two elliptic curves $E_t$ and $E_{t'}$,
\begin{itemize}
\item $E_t$ and $E_{t'}$ are identical if and only if there exists $A \in \GL_2(\Z)$ such that $At = t'$; and
\item isomorphisms $E_t \cong E_{t'}$ correspond bijectively to pairs $(\lambda,A) \in \GL_2(\Z)$ such that $\lambda A t = t'$. 
\end{itemize}
Therefore, elliptic curves over $\C$ are classified by the action of $\C^\times \times \GL_2(\Z)$ on $\X$ given by $(\lambda,A)\cdot t = \lambda At$. Alternatively, they are also classified by the action of the subgroup 
\[
\C^\times\times \SL_2(\Z) \subset \C^\times \times \GL_2(\Z)
\]
on the subspace
\[
\X^+ = \{(t_1,t_2) \in \X \, |\, \mathrm{Im}(t_1/t_2) > 0\} \subset \X,
\]
which is easily seen to inherit such an action. While we will use the $\C^\times \times \SL_2(\Z)$-action on $\X^+$ repeatedly in this thesis, we also mention a third classifying action, which is obtained by taking the quotient of $\X^+$ by the free action of $\C^\times$. Let $\bH = \{\tau \in \C \, | \, \mathrm{Im}(\tau) > 0\}$ denote the upper half plane and identify $\C^\times \backslash \X^+ \cong \bH$ via $[(t_1,t_2)] \mapsto t_1/t_2$. The action of $\SL_2(\Z)$ descends to an action on $\bH$ given by
\[
\left(\begin{array}{rcl}a&b\\c&d\end{array}\right)  \cdot \tau = \frac{a\tau + b}{c\tau + d}.
\]
Elliptic curves over $\C$ are also classified by this action.
\end{remark}

\begin{remark}
Let $\C^\times$ denote the multiplicative group of complex numbers. The image of $\X^+$ under the map $(t_1,t_2) \mapsto e^{2\pi i\tau}$, where $\tau = t_1/t_2$, is the punctured open unit disk 
\[
\D^\times := \{ z \in \C^\times \, |  \, |z| < 1\}.
\]
For $q \in \D^\times$, consider the complex analytic group $C_q = \C^\times/q^\Z$ obtained as the quotient of the multiplicative group $\C^\times$ by the subgroup generated by $q$. Let $\exp$ denote the complex exponential map $z \mapsto e^{2\pi i z}$. We have a commutative diagram of complex analytic groups
\begin{equation}\label{cut}
\begin{tikzcd}
\C \ar[r,"{\exp}"] \ar[d,two heads] & \C^\times \ar[d,two heads] \\
E_\tau := \C/\langle \tau,1 \rangle \ar[r,"{\cong}"] & \C^\times/q^\Z =: C_q
\end{tikzcd}
\end{equation}
whenever $q = e^{2\pi i \tau}$. 
\end{remark}

\begin{remark}\label{fad}
Let $T$ be a compact abelian Lie group and write $\hat{T} := \Hom(T,U(1))$ for the character group of $T$. We have four covariant functors  
\[
\begin{tikzcd}
\t_\C := \Hom(\hat{T},\C) & T_\C := \Hom(\hat{T},\C^\times) \\ 
E_{T,t} := \Hom(\hat{T},E_t) & C_{T,q}:= \Hom(\hat{T},C_q)
\end{tikzcd}
\]
from the category of compact abelian Lie groups into the category of complex manifolds. Note that an embedding $H \hookrightarrow T$ of compact abelian Lie groups induces an embedding of complex manifolds. We similarly write 
\[
E_{T,\tau} := \Hom(\hat{T},E_\tau).
\]
Applying $\Hom(\hat{T},-)$ to diagram \eqref{cut} yields a commutative diagram
\begin{equation}
\begin{tikzcd}
\t_\C \ar[r,"{\exp_T}"] \ar[d,two heads] & T_\C \ar[d,two heads] \\
E_{T,\tau} \ar[r,"{\cong}"] & C_{T,q}
\end{tikzcd}
\end{equation}
of natural transformations.
\end{remark}

\begin{remark}
Let $\bT$ be the circle group with fixed parametrisation. For example, in Chapter \ref{Rezk}, we will set $\bT = \R/\Z$, and in Chapter \ref{Kitch} we will set $\bT = S^1 \subset \C^\times$. We write 
\[
\check{T} := \Hom(\bT,T)
\]
for the lattice of homomorphisms $\bT \to T$, which does not depend on the parametrisation on $\bT$. If $T$ is a torus, then there is a canonical isomorphism $\check{T} = \Hom(\hat{T},U(1))$ given by the perfect pairing 
\[
\begin{array}{rcl}
\hat{T} \times \check{T} &\to& \Z \\
(\mu,m) &\mapsto & \mu \circ m =:\mu(m).
\end{array}
\]
We therefore have identifications
\[
\t_\C = \check{T} \otimes \C, \quad T_\C = \check{T} \otimes \C^\times \quad E_{T,t} = \check{T} \otimes E_t \quad \mathrm{and} \quad C_{T,q} = \check{T} \otimes C_q.
\]
\end{remark}

\begin{remark}
The complex exponential map 
\[
\exp: \C \longrightarrow \C^\times
\]
defined by $x \mapsto e^{2\pi i x}$ induces the map
\[
\t_\C := \check{T} \otimes \C \longrightarrow \check{T} \otimes \C^\times
\]
by tensoring with $\check{T}$, and the map
\[
\t := \check{T} \otimes \R \longrightarrow \check{T} \otimes U(1) 
\]
by restriction. We will also refer to each of these maps as $\exp$, relying on context to avoid confusion. The kernel of $\exp$ is $\check{T} \otimes \Z$, which is canonically isomorphic to $\check{T}$.
\end{remark}

\begin{notation}
If $T$ is a torus, we use the notation 
\[
u^m := m \otimes u \in T_\C  \quad \mathrm{and} \quad mx := m \otimes x \in \t_\C,
\]
so that $\exp(mx)$ is equal to
\[
e^{2\pi i mx} := (e^{2\pi i x})^m \in T_\C.
\]
Similarly, if $z = e^{2\pi imx} \in T_\C$ and $\mu \in \hat{T}$, then we write
\[
z^\mu := e^{2 \pi i x \mu(m)}.
\]
The group operations are written additively on the lattices $\hat{T}$ and $\check{T}$, and multiplicatively on $T_\C$ and $C_{T,q}$. If $T$ is a torus, if $q^{\check{T}}$ denotes the subgroup $\{q^m\}_{m\in \check{T}}$ then
\[
C_{T,q} = q^{\check{T}}\backslash T_\C
\]
so that $\psi_{T,q}$ is identified with the quotient map
\[
T_\C \twoheadrightarrow q^{\check{T}}\backslash T_\C.
\]
\end{notation}

\begin{definition}
A $T$-CW complex $X$ is a union 
\[
\bigcup_{n\in \N} X^n
\]
of $T$-subspaces $X^n$ such that
\begin{enumerate}
\item $X^0$ is a disjoint union of orbits $T/H$, where $H \subset T$ is a closed subgroup; and
\item $X^{n+1}$ is obtained from $X^n$ by attaching $T$-cells $T/H \times \D^{n+1}$ along $T$-equivariant attaching maps $T/H \times S^n \rightarrow X^n$, where $T$ acts trivially on $\D^{n+1}$.
\end{enumerate}
A \textit{finite $T$-CW complex} is a $T$-CW complex which is a union of finitely many $T$-cells. A \textit{pointed $T$-CW complex} is a $T$-CW complex along with a distinguished $T$-fixed basepoint in the $0$-skeleton of $X$. Given a $T$-CW complex $X$, we write $X_+$ for the pointed $T$-CW complex which is the disjoint union of $X$ and the basepoint $* = T/T$. A map $f: X \to Y$ of (pointed) $T$-CW complexes is a $T$-equivariant map such that $f(X^n) \subset Y^n$ for all $n$ (and preserving the basepoint). 
\end{definition}

\begin{example}\label{repsphere}
Let $T$ be a rank one torus and let $\lambda \in \hat{T}$ be an irreducible character of $T$. The \textit{representation sphere associated to $\lambda$} is the one-point compactification $S_\lambda$ of the one dimensional complex representation $\C_\lambda$ associated to $\lambda$. This may be equipped with the structure of a finite $T$-CW complex where
\[
X^0 = T/T \times \{\infty\} \: \amalg \: T/T \times \{0\},
\]
and $X^1 =  T \times \D^1$, with $T$-equivariant attaching map $T \times S^0 \rightarrow X^0$ given by sending one end of $\D^1$ to $\{0\}$ and the other end to $\{\infty\}$. An element $z \in T$ acts by multiplication by $\lambda(z)$ on the left factor of $X^1$ and trivially on the right factor.
\end{example}

%In fact, there is the following result.

%\begin{lemma}[Theorem 1.2, \cite{LMS}]
%Any smooth compact $T$-manifold is equivariantly homeomorphic to a finite $T$-CW complex.
%\end{lemma}

\section{Borel-equivariant ordinary cohomology}

In this section, we define Borel-equivariant ordinary cohomology and we state some of its properties. We endeavour to prove those properties which will be the most important for us, especially where the details of the proofs are needed in later chapters. We also introduce the notion of an equivariant cohomology theory taking values in holomorphic sheaves.

\begin{definition}\label{borel}
Let $ET$ be a contractible space with $T$ acting freely on the right. The Borel construction of $X$ is the quotient space $ET \times_T X$ obtained by identifying $(e\cdot g,x)$ with $(e,gx)$ for all $g \in T$, $e \in ET$, and $x \in X$. The Borel construction of $X$ is a fiber bundle over $BT$ with fiber $X$ (see \cite{AB}).
\end{definition}

\begin{definition}
Let $X$ be a finite $T$-CW complex. We write 
\[
H^*_T(X) := H^*(ET \times_T X;\C),
\]
for the Borel-equivariant cohomology of $X$ with coefficients in $\C$, which is a $\Z$-graded commutative ring with multiplication given by the cup product. Sometimes, we will drop the asterisk from the notation and just write $H_T(X)$, and we write
\[
H_T := H^*_T(\pt) = H^*(BT).
\]
for the Borel equivariant cohomology of a point. Since the unique map from $X$ to a point induces a map $H_T \to H_T^*(X)$ of graded rings, $H^*_T(X)$ is naturally equipped with the structure of a $\Z$-graded $H_T$-algebra. Our reference for Borel-equivariant ordinary cohomology is \cite{AB}.
\end{definition}

\begin{remark}
There is a canonical isomorphism of graded rings
\[
S(\t_\C^*) \cong H^*(BT;\C) 
\]
which enables us to identify an element in $H_T$ with a polynomial function on $\t_\C$. The map is produced as follows. Note that since $T$ is a torus, there is an identification 
\[
\hat{T} \otimes \C \cong \Hom(\check{T},\Z) \otimes \C \cong \t_\C^*.
\]
Let $\C_\lambda$ be the representation corresponding to an irreducible character $\lambda \in \hat{T}$. The map  
\[
\lambda \mapsto c_1(ET \times_T \C_\lambda)  
\]
induces an isomorphism $\hat{T} \cong H^2(BT;\Z)$, where $c_1$ denotes the first Chern class. Tensoring this map with $\C$ and extending by the symmetric product yields the isomorphism $S(\t_\C^*) \cong H^*(BT;\C)$. See Proposition 2.6 in \cite{Rosu03} for details.
\end{remark}

\subsection{Change of groups}

\begin{proposition}\label{change1}
There is an isomorphism of graded rings
\[
H_T(T/H) \cong H_H.
\]
\end{proposition}

\begin{proof}
Since $H$ acts freely on $ET$, the space $ET$ is a model for $EH$. Therefore, 
\[
ET \times_T T/H \cong ET/H
\]
is a model for $BH$. 
\end{proof}

\begin{remark}
The following two results are given as Proposition 2.3.4 in Chen's thesis \cite{Chen}. Let 
\[
1 \to H \to T \xrightarrow{f} K \to 1
\]
be a short exact sequence of compact abelian groups, where $T$ is a torus. Let $X$ be a finite $T$-CW complex such that $H$ acts trivially, and consider the diagram
\begin{equation}\label{pullback}
\begin{tikzcd}
ET \times_T X \ar[d,"h"] \ar[r,"j"] & BT \ar[d,"p"] \\
EK \times_K X \ar[r] & BK.
\end{tikzcd}
\end{equation}
where the left vertical map is induced by $f$.
\end{remark}

\begin{lemma}[Prop. 2.3.4., \cite{Chen}]\label{darryl}
Diagram \eqref{pullback} is a pullback diagram.
\end{lemma}

\begin{proposition}[Prop. 2.3.4., \cite{Chen}]\label{changeh}
Let $X$ be a finite $T$-CW complex such that $H \subset T$ acts trivially. There is an isomorphism of graded $H_T$-algebras
\[
H^*_{K}(X) \otimes_{H_{K}} H_T \cong H^*_T(X)
\]
natural in $X$, and induced by $h^* \cup j^*$. 
\end{proposition}

\begin{proof}
The diagram \eqref{pullback} is a pullback diagram by Lemma \ref{darryl}. We also have $\pi_1(BK) = \pi_0(K) = 0$, because $K$ is connected, since $T$ is a torus. We also know that $p$ is a fibration, since a short exact sequence
\[
1 \to H \to T \xrightarrow{f} K \to 1
\]
of compact abelian groups induces a fibration sequence $BH \to BT \xrightarrow{p} BK$. Therefore, we may apply the Eilenberg-Moore spectral sequence, which is a second quadrant spectral sequence converging to $H(ET \times_T X)$. The $E_2$-term is given by
\[
\begin{array}{rcl}
E_2^{p,q} &=& \Tor_{H^*(BK)}^{p,q} (H^*(EK\times_K X),H^*(BT)) \\
&=& \text{a subquotient of} \quad H^*(EK \times_K X) \otimes_{H^*(BK)} P^p
\end{array} 
\]
where $P$ is a projective resolution of $H^*(BT)$ in the category of differential-graded $H^*(BK)$-modules. Furthermore, by definition of the differential-graded version of $\Tor$, we have that
\[
\bigoplus_i \Tor_{H^*(BK)}^{0,i} = H^*(EK \times_K X) \otimes_{H^*(BK)} H^*(BT).
\]
Thus, if we can show that $H^*(BT)$ is a free $H^*(BK)$-module, it will follow that $E_2^{p,q}$ is trivial whenever $p \neq 0$, since in that case $P$ will be concentrated in degree $0$. Hence, all differentials on $E_r$ for $r \geq 2$ will be trivial, and the Eilenberg-Moore spectral sequence will collapse at $E_2$, yielding the result. It follows from the construction of the spectral sequence that the isomorphism thus produced is induced by $h^* \cup j^*$, and is therefore natural in $X$.

We will show that $H^*(BT)$ is a free $H^*(BK)$-module using the Serre spectral sequence for the fibration
\[
BH \to BT \xrightarrow{p} BK.
\]
Since we are working with complex coefficients, the universal coefficient theorem implies that the $E_2$-term is
\[
E_2^{p,q} = H^p(BK) \otimes_\C H^q(BH),
\]
and it converges to $H^*(BT)$. Since $K$ and $H$ are compact Lie groups, $H^*(BK)$ and $H^*(BH)$ are concentrated in even degrees, all of the differentials of the Serre spectral sequence are trivial and it collapses at $E_2$. Hence,
\[
H^{p+q}(BT) \cong H^p(BK) \otimes_\C H^q(BH),
\]
as $H^*(BK)$-modules, which means that $H^*(BT)$ is a free $H^*(BK)$-module.
\end{proof}

\begin{remark}
Let $T \twoheadrightarrow K \twoheadrightarrow G$ be a composition of surjective maps of compact abelian groups, where $T$ is a torus. If $X$ is a finite $G$-CW complex, then there is a commutative diagram 
\begin{equation}
\begin{tikzcd}
ET \times_T X \ar[d,"h"] \ar[r,"j"] & BT \ar[d,"p"] \\
EK \times_K X \ar[r,"g"] \ar[d,"f"] & BK \ar[d] \\
EG \times_G X \ar[r] & BG 
\end{tikzcd}
\end{equation}
where both squares are pullback diagrams. By Proposition \ref{changeh}, we have an induced diagram of isomorphisms of graded $H_T$-algebras  
\begin{equation}\label{pullback2}
\begin{tikzcd}
H_G(X) \otimes_{H_G} H_K \otimes_{H_K} H_T \ar[d] \ar[r,"{f^* \cup g^* \otimes \id}"] & H_K(X) \otimes_{H_K} H_T \ar[d,"{h^* \cup j^*}"] \\
H_G(X) \otimes_{H_G} H_T \ar[r,"{(f\circ h)^* \cup j^*}" ] & H_T(X)
\end{tikzcd}
\end{equation}
where the left vertical map is the canonical map induced by $p^*: H_K \to H_T$.
\end{remark}

\begin{lemma}\label{kurt}
The diagram \eqref{pullback2} commutes.
\end{lemma}

\begin{proof}
It suffices to show that diagram \eqref{pullback2} commutes for an element of the form $a \otimes b \otimes c$. We have
\[
\begin{tikzcd}
a \otimes b \otimes c \ar[r,mapsto] \ar[d,mapsto]& (f^*a \cup g^*b) \otimes c \ar[d,mapsto] \\
a \otimes (p^*b \cup c) \ar[r,mapsto] & (f\circ h)^*a \cup j^*(p^*b \cup c) = h^*(f^*a \cup g^*b) \cup j^*c,
\end{tikzcd}
\]
where equality holds since 
\[
\begin{array}{rcl}
(f\circ h)^*a \cup j^*(p^*b \cup c) &=& (f\circ h)^*a \cup (p\circ j)^*b \cup j^*c \\
&=& (f\circ h)^*a \cup (g\circ h)^*b \cup j^*c \\
&=& h^*f^*a \cup h^*g^*b \cup j^*c \\
&=& h^*(f^*a \cup g^*b) \cup j^*c.
\end{array}
\]
\end{proof}

\begin{remark}
It will be sometimes convenient for us to restrict our attention to finite $T$-CW complexes $X$ which satisfy the following equivalent conditions:
\begin{itemize}
\item the natural map $H_T(X) \cong H(X) \otimes_\C H_T$ induced by the inclusion $X \hookrightarrow ET \times_T X$ is an isomorphism of $H_T$-modules;
\item the Serre spectral sequence of the fibration $X \hookrightarrow ET \times_T X \twoheadrightarrow BT$ collapses on the $E_2$ page. 
\end{itemize}
The second condition indicates how this property might be checked. Note that, while we may have that $H_T(X) \cong H(X) \otimes_\C H_T$ as $H_T$-modules, information about the $T$-action on $X$ is still carried by the $H_T$-algebra structure on $H_T(X)$.
\end{remark}

\begin{definition}\label{eqfor}
A $T$-space $X$ is said to be \textit{equivariantly formal} if the Serre spectral sequence of the fibration $X \hookrightarrow ET \times_T X \twoheadrightarrow BT$ collapes on the $E_2$ page. 
\end{definition}

\begin{example}
There are many examples of interesting $T$-spaces which are equivariantly formal:
\begin{itemize}
\item Any $T$-CW complex $X$ with cohomology concentrated in even degrees is equivariantly formal, because then the differentials in the Serre spectral sequence of the assocaited Borel construction are all trivial. This includes the representation sphere $S_\lambda$ of Example \ref{repsphere};
\item Any symplectic manifold with a Hamiltonian $T$-action;
\item Any subvariety of a complex projective space with linear $T$-action.
\end{itemize} 
\end{example}

\subsection{Localisation}\label{loco}

\begin{proposition}\label{flat}
Let $T$ be a torus and let $U$ be an open set in the complex Lie algebra $\t_\C = \check{T} \otimes \C$. The inclusion of rings
\[
H_T \hookrightarrow \O_{\t_\C}(U)
\]
given by regarding a polynomial in $H_T$ as a holomorphic function on $U$ is faithfully flat.
\end{proposition}

\begin{proof}
See the proof of Proposition 2.8 in \cite{Rosu03}.
\end{proof}

\begin{definition}
Let $T$ be a torus and let $X$ be a finite $T$-CW complex. We have reserved the notation $\H^*_T(X)$ for the holomorphic sheaf associated to $H^*_T(X)$, whose value on an analytic open set $U \subset \t_\C$ is
\[
H^*_T(X)_U := H^*_T(X) \otimes_{H_T} \O_{\t_\C}(U)
\]
with restriction maps induced by those of $\O_{\t_\C}$. It follows from Proposition \ref{flat} that this is a sheaf, rather than just a presheaf (for the proof see Proposition 2.10 in \cite{Rosu03}). We write $\H^*_T(X)_V$ for the restriction of $\H^*_T(X)$ along the inclusion of any subset $V \subset \t_\C$. 
\end{definition}

\begin{remark}\label{eqii}
We have a decomposition of $\O_{\t_\C}$-modules
\[
\H^*_T(X) = \H_T^{even}(X) \oplus \H^{odd}_T(X)
\]
corresponding to even and odd cohomology
\[
H_T^{even}(X) := \bigoplus_{n\in \Z} H^{2n}_T(X) \quad \text{and} \quad H^{odd}_T(X) := \bigoplus_{n\in \Z} H^{2n+1}_T(X),
\]
since $\O_{\t_\C}$ is concentrated in even degrees. Therefore, $\H^*_T(X)$ is equipped with the structure of a $\Z/2\Z$-graded $\O_{\t_\C}$-algebra. 
\end{remark}

\begin{proposition}\label{Segall}
Let 
\[
p : T \longrightarrow T/H
\]
be a surjective map of compact abelian groups, with kernel $H$. The natural map
\[
p^*\, \H^*_{T/H}(X) \longrightarrow  \H^*_T(X)
\]
is an isomorphism of $\O_{\t_\C}$-algebras.
\end{proposition}

\begin{proof}
This follows immediately from Theorem \ref{changeh} by extending to holomorphic sheaves.
\end{proof}

\begin{definition}
Let $x \in \t_\C$. Recall that the inclusion of a closed subgroup $H \subset T$ induces an inclusion of complex Lie algebras 
\[
\mathrm{Lie}(H)_\C := \Hom(\hat{H},\C) \hookrightarrow \Hom(\hat{T},\C) =: \t_\C.
\]
Define the intersection
\[
T(x) = \bigcap_{x \in \mathrm{Lie}(H)_\C} H
\]
of closed subgroups $H \subset T$. For a finite $T$-CW complex $X$, denote by $X^{x}$ the subspace of points fixed by $T(x)$.
\end{definition}

The following theorem is called the Localisation Theorem for Borel-equivariant cohomology.

\begin{theorem}\label{localisationn}
Let $x \in \t_\C$ and $X$ be a finite $T$-CW complex. The restriction map
\[
\H^*_T(X)_x \longrightarrow \H^*_T(X^x)_x
\]
is an isomorphism of $\O_{\t_\C,x}$-algebras.
\end{theorem}

\begin{proof}
We show this by induction. For the base case, consider the $0$-skeleton $X^0$ of $X$, which is a coproduct
\[
X^0 = \coprod_\alpha T/H_\alpha
\]
indexed over finitely many $\alpha$. Then
\[
H^*_T(X^0) = \prod_\alpha H^*_T(T/H_\alpha) \cong \prod_\alpha H_{H_\alpha},
\]
where the last isomorphism follows from Proposition \ref{change1}. The stalk of $\H_T(X^0)$ at $x \in \t_\C$ is therefore 
\[
\H^*_T(X^0)_x = \prod_\alpha H_{H_\alpha} \otimes_{H_T} \O_{\t_\C,x}
\]
where the tensor products are defined via the respective restriction maps $H_T \to H_{H_\alpha}$. If $x \notin \mathrm{Lie}(H_{\alpha})_\C$, then there exists an element $p \in H_T$ which is nonzero at $x$ and whose image in $H_{H_\alpha}$ is zero. It follows that
\[
H_{H_\alpha} \otimes_{H_T} \O_{\t_\C,x} = 0
\]
whenever $x \notin \mathrm{Lie}(H_{\alpha})_\C$.

Note that $(X^0)^x$ is the disjoint union of all orbits $T/H_\alpha$ in $X^0$ such that $x \in \mathrm{Lie}(H_{\alpha})_\C$. The restriction map 
\[
\H^*_T(X^0)_x \longrightarrow \H^*_T((X^0)^x)_x
\]
is therefore an isomorphism, which proves the base case.

Let 
\[
Y^n = \coprod_\alpha T/H_\alpha \times \D_\alpha^n
\]
for $n > 0$. Since this is $T$-homotopy equivalent to
\[
\coprod_\alpha T/H_\alpha \times \D_\alpha^0,
\]
the same argument shows that 
\[
\H^*_T(Y^n)_x \longrightarrow \H^*_T((Y^n)^x)_x
\]
is an isomorphism. A standard Mayer-Vietoris argument then yields the induction step. Therefore, the isomorphism of the theorem holds for any finite $T$-CW complex $X$.
\end{proof}

\section{Equivariant K-theory}

In this section we define $T$-equivariant K-theory and set out the properties which will be important for us in the sequel. 

\begin{definition}
Let $T$ be a compact abelian group and $X$ be a finite $T$-CW complex. The set of isomorphism classes of $T$-vector bundles on $X$ forms a commutative monoid under direct sum, and the $T$-equivariant K-theory $K_T(X) = K^0_T(X)$ of $X$ is defined as the associated Grothendieck group. This means that an element of $K_T(X)$ is a formal difference $[V] - [W]$ of isomorphism classes of $T$-equivariant vector bundles on $X$, also called a \textit{virtual $T$-bundle}. The abelian group $K_T(X)$ has a ring structure which is induced by the tensor product of underlying vector bundles. Equivariant K-theory defines a contravariant functor on the category of finite $T$-CW complexes by assigning to a morphism $f:X \to Y$ the ring homomorphism $K_T(Y) \to K_T(X)$ given by pullback along $f$, and satisfies the axioms for a generalised cohomology theory. Our references for equivariant K-theory are \cite{Segal} and Chapter XIV of \cite{Alaska}.
\end{definition}

\begin{remark}
It is immediate from the definition that the equivariant K-theory $K_T := K_T(\pt)$ of a point is identical to the representation ring $R(T)$ of $T$. The canonical map $K_T \to K_T(X)$ makes $K_T(X)$ an algebra over $K_T$.
\end{remark}

\begin{definition}\label{susp}
For a pointed $T$-CW complex $X$ with basepoint $*$, the \textit{reduced equivariant K-theory} $\tilde{K}_T(X)$ of $X$ is defined to be the kernel of the map $K_T(X) \to K_T(\pt)$ induced by the inclusion of the basepoint. Thus, $\tilde{K}_T(X)$ is an ideal of the ring $K_T(X)$, and is therefore a module over $K_T$. For an (unpointed) $T$-CW complex $X$, one verifies that 
\[
K^0_T(X) = \tilde{K}_T(X_+),
\]
since we clearly have $K_T(X_+) = K_T(X) \oplus K_T(\pt)$. We can extend equivariant K-theory to a $\Z/2\Z$-graded theory as follows. Define the $K_T$-module
\[
K^{-n}_T(X) := \tilde{K}_T(S^n\wedge X_+)
\]
for all positive integers $n$. We write $K^0_T(X) = K_T(X)$ and $\tilde{K}^0_T(X) = \tilde{K}_T(X)$. We also write $\tilde{K}^n_T(X)$ for the kernel of the map $K^n_T(X) \to K^n_T(\pt)$ induced by the inclusion of the basepoint. 
\end{definition}

\begin{remark}
In the course of this paper, we will be applying the conventions and notation of Definition \ref{susp} to other cohomology theories also.
\end{remark}

\begin{remark}
We briefly mention the important equivariant Bott periodicity theorem (see Theorem 3.2 in Chapter XIV of \cite{Alaska}) in order to define the $\Z/2\Z$-graded theory. The theorem states, among other things, that there is an isomorphism
\[
K^0_T(X) \cong K^{-2}_T(X).
\]
Therefore, we can extend equivariant K-theory to a $\Z$-graded, 2-periodic functor by setting
\[
K^{n}_T(X) := K^{n-2}_T( X)
\]
for all positive integers $n$. Thus, $K^*_T(X)$ is a $\Z/2\Z$-graded $K_T$-module isomorphic to $K^0_T(X) \oplus K^1_T(X)$.
\end{remark}

\begin{remark}
There is a multiplicative product on $K^*_T(X)$ which makes it a $\Z/2\Z$-graded algebra over the (ungraded) ring $K_T$. The product is defined as follows. Let $V$ be a $T$-vector bundle on $S^i\wedge X_+$ and $W$ a $T$-vector bundle on $S^j\wedge X_+$, each of which restricts to the zero-dimensional bundle on its basepoint. Let $p_1$ and $p_2$ denote the projection of
\[
(S^i\wedge X_+) \wedge (S^j \wedge X_+)
\]
to the first and second factors, respectively. Then the tensor product $E = p_1^*V \otimes p_2^*W$ is a $T$-vector bundle on $S^{i+j}\wedge X_+ \wedge X_+$. The pullback of $E$ along the map
\[
S^{i+j}\wedge X_+ \xrightarrow{\id \wedge \Delta} S^{i+j}\wedge X_+ \wedge X_+, 
\]
where $\Delta$ is the diagonal map, restricts to zero on the basepoint, since $V$ and $W$ do. The product of $V$ and $W$ is defined to be the corresponding element of $K^{-i-j}_T(X)$. The multiplicative structure is therefore the composite map
\begin{equation}\label{multi}
\tilde{K}_T(S^1 \wedge X_+) \otimes \tilde{K}_T(S^j\wedge X_+) \xrightarrow{p_1^*\otimes p_2^*} \tilde{K}_T(S^{i+j} \wedge X_+ \wedge X_+) \xrightarrow{(\id \wedge \Delta)^*} \tilde{K}_T(S^{i+j}\wedge X_+).
\end{equation}
\end{remark}

\begin{remark}\label{harold}
Let $X \mapsto J^*_T(X)$ be a contravariant functor from finite $T$-CW complexes into $\Z/2\Z$-graded rings. Suppose that 
\[
J^1_T(X) = \ker(J^0_T(S^1 \wedge X_+) \to J^0_T(\pt)),
\]
and the multiplicative structure on $J^*_T(X)$ is defined as in \eqref{multi}, using the product on the degree zero ring $J^0_T(X)$. Then, a natural isomorphism $K^0_T(X) \cong J^0_T(X)$ of rings induces a natural isomorphism $K^*_T(X) \cong J^*_T(X)$ of $\Z/2\Z$-graded algebras over $K_T(\pt) \cong J_T(\pt)$. Indeed, the isomorphism in degree one is induced by naturality of the isomorphism in degree zero, since
\[
\begin{array}{rcl}
K^1_T(X) &:=& \tilde{K}^0_T(S^1 \wedge X_+) \\
&:=& \ker(K_T^0(S^1 \wedge X_+) \to K_T^0(\pt)) \\
&\cong& \ker(J_T^0(S^1 \wedge X_+) \to J_T^0(\pt)) \\
&=:& \tilde{J}^0_T(S^1 \wedge X_+) \\
&=& J^1_T(X).
\end{array}
\]
Since the multiplicative structures of $K^*_T(X)$ and $J^*_T(X)$ both defined as in \eqref{multi}, and the reduced functors are naturally isomorphic as rings, the isomorphism preserves the graded algebra structure.
\end{remark}

\subsection{Change of groups}

\begin{remark}
For any compact $T$-space $X$, a group homomorphism $f: H \to T$ induces a change of groups map
\[
f^*: K^*_T(X) \to K^*_H(X),
\]
given by pulling back the $T$-action along $f$, and which is natural in $X$.
\end{remark}

\begin{proposition}\label{change11}
Let $f: H \hookrightarrow T$ be a closed subgroup and let $X$ be a finite $H$-CW complex. Let $i: X \hookrightarrow T \times_H X$ be the $H$-equivariant inclusion map. There is a natural isomorphism of $\Z/2\Z$-graded rings given by the composite
\[
K^*_T(T \times_H X) \xrightarrow{f^*} K^*_H(T \times_H X) \xrightarrow{i^*} K^*_H(X).
\]
\end{proposition}

\begin{proof}
We clearly have an isomorphism in degree zero, with inverse given by $V \mapsto T \times_H V$. The isomorphism is natural in $X$, since a map $X \to Y$ of finite $H$-CW complexes induces a map $T \times_H X \to T \times_H Y$. By Lemma \ref{harold}, the isomorphism extends to an isomorphism of $\Z/2\Z$-graded rings since
\[
S^1 \wedge (T \times_H X_+) = T \times_H (S^1 \wedge X_+),
\]
since $S^1$ has trivial group action. 
\end{proof}

\begin{notation}
For a finite $T$-CW complex $X$, we denote by $\C_\lambda \to X$ the pullback along $X \to \pt$ of the representation corresponding to $\lambda \in \hat{T}$, which we also denote $\C_\lambda$ by abuse of notation. The dual line bundle $\C_\lambda^\vee$ is canonically isomorphic to $\C_{-\lambda}$. 
\end{notation}

\begin{remark}
Let $X$ be a compact $T/H$-space and let $E$ be a $T$-equivariant vector bundle over $X$. It is shown in the proof of Proposition 2.2 in \cite{Segal} that $E$ has a fiberwise decomposition
\[
E \cong \bigoplus_{\mu \in \hat{H}} E_\mu
\]
where $H$ acts on the fibers of $E_\mu$ via the character $\mu$. 
\end{remark}

\begin{remark}
Notice that the inclusion $H \subset T$ of a closed subgroup induces a surjective map of character groups
\[
\hat{T} \longrightarrow \hat{H},
\]
and therefore it is possible to extend any character of ${H}$ to a character of $T$. 
\end{remark}

\begin{proposition}\label{Segal1}
Let $H \subset T$ be a closed subgroup, and let $X$ be a compact $T$-space such that $H$ acts trivially on $X = X^H$. For $\mu \in \hat{H}$, choose an extension $\bar{\mu}$ of $\mu$ to $T$. Then 
\[
\begin{array}{rcl}
K_T(X) &\longrightarrow & K_{T/H}(X) \otimes_{K_{T/H}} K_T \\
{[E_\mu]} &\longmapsto & [E_\mu \otimes \C_{-\bar{\mu}}] \otimes [\C_{\bar{\mu}}]
\end{array}
\]
is a well defined isomorphism of $K_T$-algebras, natural in $X$.
\end{proposition}

\begin{proof}
The map is well defined since, for a different choice of extension $\bar{\mu}' \in \hat{T}$, the difference $[\C_{\bar{\mu} -\bar{\mu}'}]$ lies in $K_{T/H}$, and cancels out. The map preserves the $K_T$-algebra structure since
\[
[E_\mu \otimes \C_\lambda] \mapsto [E_\mu \otimes \C_\lambda \otimes \C_{-\bar{\mu}-\lambda}] \otimes [\C_{\bar{\mu}+\lambda}] = [E_\mu \otimes \C_{-\bar{\mu}}] \otimes [\C_{\bar{\mu}+\lambda}].
\]
Finally, it has an obvious inverse map given by pulling back the $T/H$-action on vector bundles to a $T$-action, and tensoring with an element of $K_T$. Naturality is clear.
\end{proof}

\subsection{Localisation}\label{loca}

\begin{proposition}[Prop. 2.22, \cite{Rosu03}]\label{flatt}
By regarding an element $f$ of $K_T$ as a character of $T$, we may view $f$ as a holomorphic function on the complexification $T_\C = \Hom(\hat{T},\C^\times)$ of $T$. The associated ring homomorphism
\[
K_T \longrightarrow \O_{T_\C}(U)
\]
is faithfully flat for any open $U \subset T_\C$.
\end{proposition}

\begin{definition}
Let $T$ be a torus and let $X$ be a finite $T$-CW complex. Denote by $\K^*_T(X)$ the holomorphic sheaf whose value on an analytic open set $U \subset T_\C$ is
\[
K^*_T(X)_U := K^*_T(X) \otimes_{K_T} \O_{T_\C}(U)
\]
with restriction maps induced by those of $\O_{T_\C}$. It follows from Proposition \ref{flatt} that this is a sheaf (as in the proof of Prop. 2.10 in \cite{Rosu03}. We write $\K^*_{T}(X)_{V}$ for the restriction of the sheaf to any subset $V \subset T_\C$. 
\end{definition}

\begin{proposition}\label{Segal}
Let 
\[
p : T \longrightarrow T/H
\]
be a surjective map of compact abelian groups, with kernel $H$. The natural map
\[
p_{\C}^*\, \K_{T/H}(X) \longrightarrow  \K_T(X)
\]
is an isomorphism of $\O_{T_\C}$-algebras.
\end{proposition}

\begin{proof}
This follows immediately from Theorem \ref{Segal1} by extending to holomorphic sheaves.
\end{proof}

\begin{definition}
Let $u \in T_\C$. Recall that the inclusion of a closed subgroup $H \subset T$ induces an inclusion of complex Lie groups
\[
H_\C := \Hom(\hat{H},\C^\times) \hookrightarrow \Hom(\hat{T},\C^\times) =: T_\C.
\]
Define the intersection 
\[
T(u) = \bigcap_{u \in H_\C} H
\]
of closed subgroups $H \subset T$. For a finite $T$-CW complex $X$, denote by $X^{u}$ the subspace of points fixed by $T(u)$.
\end{definition}

The following theorem is called the Localisation Theorem for equivariant K-theory.

\begin{theorem}\label{localisation}
Let $u \in T_\C$. The restriction map
\[
\K^*_T(X)_u \longrightarrow \K^*_T(X^u)_u
\]
is an isomorphism of $\O_{T_\C,u}$-algebras.
\end{theorem}

\begin{proof}
The proof is exactly analogous to the proof of Theorem \ref{localisationn}, using the additivity axiom, the change of groups isomorphism $K^*_T(T/H) \cong K_H$, and the Mayer-Vietoris sequence.
\end{proof}

%\begin{remark}\label{fol}
%Note $\psi(q,u) = a$ implies that $a \in C_{T(u),q}$, which is true if and only if $T(a) \subset T(u)$. Therefore, $X^u \subset X^a$.
%\end{remark}

\section{The equivariant Chern character}

Let $X$ be finite $T$-CW complex. The Chern character is a map of cohomology theories which is induced by the exponential map. Namely, it is a homomorphism of rings
\[
ch: K(X) \to \prod_{2n\in \Z} H^{2n}(X;\Q), 
\]
which is natural in $X$ and which commutes with coboundary maps. It is a classical result that $ch$ is an isomorphism after tensoring with $\Q$. In \cite{Rosu03}, Rosu defined an equivariant version of the Chern character, and it is the goal of this section to describe this map. 

\begin{definition}
By the splitting principle, it suffices to define the Chern character on a line bundle $L \in K(X)$. Thus
\[
ch(L) := e^{c_1(L)} = 1 + c_1(L) + c_1(L)^2/2! + ... \in \prod_{n\in \Z} H^{2n}(X;\Q).
\]
\end{definition}

\begin{remark}
The first step in formulating Rosu's result is to notice that, by applying the Chern character isomorphism to the $k$-skeleton $ET^{(k)} \times_T X$ of the Borel construction of $X$, we obtain an isomorphism 
\begin{equation}\label{lurb}
\varprojlim_k K(ET^{(k)} \times_T X) \cong \varprojlim_k \prod_{2n\in \Z} H^{2n}(ET^{(k)} \times_T X) \cong \prod_{n\in \Z} H^{2n}(ET \times_T X),
\end{equation}
where we are again using complex coefficients.
\end{remark}

\begin{notation}
Let $H_T(X)^\wedge_0$ denote the completion of the $H_T$-module $H_T(X)$ at $0 \in \t_\C$, and let $K_T(X)^\wedge_1$ denote the completion of the $K_T$-module $K_T(X)$ at $1 \in T_\C$. 
\end{notation}

\begin{theorem}[Completion theorem 1]
If $X$ is a finite $T$-CW complex, then there is a natural isomorphism
\[
H^{even}_T(X)^\wedge_0 \cong \prod_{n\in \Z} H^{2n}(ET \times_T X). 
\]
\end{theorem}

\begin{proof}
If $X$ is a finite $T$-CW complex, then $H_T(X)$ is finitely generated over $H_T$. The result now follows from Theorem 55 in \cite{Mats} since $H_T$ is Noetherian.
\end{proof}

\begin{theorem}[Completion theorem 2]
If $X$ is a finite $T$-CW complex, then there is a natural isomorphism
\[
K_T(X)^\wedge_1 \cong \varprojlim_k K(ET^{(k)} \times_T X) .
\]
\end{theorem}

\begin{proof}
This is the Atiyah-Segal completion theorem, see \cite{AS}.
\end{proof}

\begin{remark}
Via the completion theorems, we can interpret \eqref{lurb} as an isomorphism
\begin{equation}\label{lur}
K_T(X)^\wedge_1 \xrightarrow{\cong} H^{even}_T(X)^\wedge_0.
\end{equation}
To relate this to holomorphic sheaves, we need $X$ to be equivariantly formal, which means that $H_T(X) \cong H(X) \otimes_\C H_T$ as $H_T$-modules. Thus,
\[
H^{even}_T(X)^\wedge_0 \cong H^{even}(X) \otimes_{\C} H_T(\pt)^\wedge_0
\]
are free as $H_T(\pt)^\wedge_0$-modules, and so we have inclusions 
\[
\begin{tikzcd}
\H^{even}_T(X)_0 \ar[d,hook] \ar[r,"{\cong}"] &  H^{even}(X) \otimes_\C \O_{\t_\C,0} \ar[d,hook] \\
H^{even}_T(X)^\wedge_0 \ar[r,"{\cong}"] & H^{even}(X) \otimes_\C H^\wedge_{T,0}.
\end{tikzcd}
\]
\end{remark}

The following formulation of the equivariant Chern character is based on Ganter's formulation in \cite{Nora} (see Theorem 3.2). The isomorphism is first determined in degree zero in the manner outlined above, then extended formally to degree one by Lemma \ref{harold}.

\begin{theorem}[Theorem 3.6, \cite{Rosu03}]\label{chern}
Let $X$ be an equivariantly formal, finite $T$-CW complex, and let $U$ be a neighbourhood of zero in $\t_\C$ such that $\exp|_U$ is a bijection. There is an isomorphism 
\[
ch_T: \K^*_T(X)_{\exp(U)} \to (\exp|_U)_* (\H^*_T(X)_U)
\]
of sheaves of $\Z/2\Z$-graded $\O_{\exp(V)}$-algebras, natural in $X$, determined by the commutative diagram
\[
\begin{tikzcd}
\K_T(X)_1 \ar[r,"{ch_{T,1}}"] \ar[d] & \H^{even}_T(X)_0 \ar[d,hook] \\  
K_T(X)^\wedge_1 \ar[r,"{ch}"] & H^{even}_T(X)^\wedge_0
\end{tikzcd}
\]
\end{theorem}

\section{Grojnowski's equivariant elliptic cohomology}

There are already many thorough accounts (e.g. \cite{Rosu}, \cite{Ando}, \cite{Chen}, \cite{GKV}) of Grojnowski's equivariant elliptic cohomology in the literature. We give an account here because the reader should understand the construction first, before reading any further, in order to appreciate our main results. In this section, we fix an elliptic curve $E_t = \C/\Lambda_t$ corresponding to $t \in \X^+$ and we recall the complex analytic map $\zeta_{T,t}: \t_\C \to E_{T,t}$ induced by the quotient map $\C \twoheadrightarrow \C/\Lambda_t$.

\begin{definition}\label{teed}
Let $a \in E_{T,t}$. Define the intersection
\[
T(a) = \bigcap_{a \in E_{H,t}} H
\]
of closed subgroups $H \subset T$. For a $T$-CW complex $X$, denote by $X^a$ the subspace of points fixed by $T(a)$.
\end{definition}

\begin{remark}
If $\S$ is a finite set of closed subgroups of $T$, we can define an ordering on the points of $E_{T,t}$ by saying that $a \leq_\S b$ if $b \in E_{H,t}$ implies $a \in E_{H,t}$, for any $H \in \S$. If $\S$ is understood, then we just write $\leq$ for $\leq_\S$. 
\end{remark}

\begin{notation}
If $X$ is a finite $T$-CW complex, let $\S(X)$ be the finite set of isotropy subgroups of $X$. If $f: X \to Y$ is a map of finite $T$-CW complexes, let $\S(f)$ be the finite set of isotropy subgroups which occur in either $X$ or $Y$. An open set $U$ in $E_{T,t}$ is \textit{small} if $\zeta_{T,t}^{-1}(U)$ is a disjoint union of connected components $V$ such that $V \cong U$ via $\zeta_{T,t}$.
\end{notation}

\begin{definition}\label{ross}
Let $\S$ be a finite set of closed subgroups of $T$. An open cover $\U = \{U_a\}$ of $E_{T,t}$ indexed by the points of $E_{T,t}$ is said to be \textit{adapted to $\S$} if it has the following properties:
\begin{enumerate}
\item $a \in U_{a}$, and $U_{a}$ is small.
\item If $U_{a} \cap U_{b} \neq \emptyset$, then either $a \leq_\S b$ or $b \leq_\S a$.
\item If $a \leq_\S b$, and for some $H \in \S$, we have $a \in E_{H,t}$ but $b \notin E_{H,t}$, then $U_b \cap E_{H,t} = \emptyset$.
\item Let $a$ and $b$ lie in $E_{H,t}$ for some $H \in \S$. If $U_a \cap U_b \neq \emptyset$, then $a$ and $b$ belong to the same connected component of $E_{H,t}$.
\end{enumerate}
\end{definition}

\begin{lemma}[Theorem 2.2.8, \cite{Chen}]\label{constru}
For any finite set $\S$ of subgroups of $T$, there exists an open cover $\U$ of $E_{T,t}$ adapted to $\S$. Any refinement of $\U$ is also adapted to $\S$.
\end{lemma}

\begin{notation}
For $a \in E_{T,t}$ let 
\[
t_a: E_{T,t} \longrightarrow E_{T,t}
\]
denote translation by $a$. 
\end{notation}

\begin{remark}\label{canon}
Let $X$ be a finite $T$-CW complex and let $\U$ be a cover of $E_{T,t}$ which is adapted to $\S(X)$. Let $x \in \zeta_{T,t}^{-1}(a)$, and let $V_{x}$ be the component of $\zeta_{T,t}^{-1}(U_a)$ containing $x$. Let $V \subset V_x$ and $U \subset U_a$ be open subsets such that $V \cong U$ via $\zeta_{T,t}$. Since $U_a \in \U$ is small by the first property of an adapted cover, the map $\zeta_{T,t}$ induces an isomorphism of complex analytic spaces $V - x \cong U-a$. We may therefore consider the composite ring map
\begin{equation}\label{ring}
H_T \hookrightarrow \O_{\t_\C}(V-x) \cong \O_{E_{T,t}}(U - a).
\end{equation}
\end{remark}

\begin{definition}
Let $X$ be a finite $T$-CW complex. For each $U_a \in \U(X)$, define a sheaf $\G^*_{T,t}(X)_{U_a}$ of $\Z/2\Z$-graded $O_{U_a}$-algebras which takes the value
\[
H^*_T(X^a) \otimes_{H_T} \O_{E_{T,t}}(U-a),
\]
on $U \subset U_a$ open, with restriction maps given by restriction of holomorphic functions. The tensor product is defined over \eqref{ring}, and the $\O_{U_a}$-algebra structure is given by multiplication by $t_a^* f$ for $f \in \O_{U_a}(U)$. The grading is induced by the odd and even grading on the cohomology ring (see Remark \ref{eqii}).
\end{definition}

\begin{remark}
For a finite $T$-CW complex $X$, we have defined a sheaf on each patch $U_a$ of a cover $\U$ adapted to $\S(X)$. The next task is to glue the local sheaves together on nonempty intersections $U_a \cap U_b$ in a compatible way. To do this, we need to define gluing maps 
\[
\phi_{b,a}: \G_{T,t}(X)_{U_a}|_{U_a \cap U_b} \cong  \G_{T,t}(X)_{U_b}|_{U_a \cap U_b}
\]
for each such intersection, such that the cocycle condition $\phi_{c,b} \circ \phi_{b,a} = \phi_{c,a}$ is satisfied. 
\end{remark}

Note that we have either $X^b \subset X^a$ or $X^a \subset X^b$ whenever $U_{a} \cap U_{b} \neq \emptyset$, by the second property of an adapted cover.

\begin{theorem}
Let $X$ be a finite $T$-CW complex, and let $\U$ be a cover adapted to $\S(X)$. Let $a \leq b$ be points in $E_{T,t}$ and assume $U \subset U_{a} \cap U_{b}$ is an open subset. By the second property of an adapted cover, we may assume that $X^b \subset X^a$, with inclusion map $i_{b,a}$. The map 
\[
i_{b,a}^*\otimes \id: H_T(X^a) \otimes_{H_T} \O_{E_{T,t}}(U-a) \to H_T(X^b) \otimes_{H_T} \O_{E_{T,t}}(U-a)
\]
induced by restriction along $i_{b,a}$ is an isomorphism of $O_{E_{T,t}}(U)$-modules.
\end{theorem}

\begin{proof}
This is Theorem 2.3.3 in \cite{Chen}.
\end{proof}

\begin{remark}
Let $H = \langle T(a),T(b) \rangle$ and let $U \subset U_a \cap U_b$ be an open subset. There is a natural isomorphism of $\Z/2\Z$-graded $\O(U)$-algebras given on $U \subset U_a$ by the composite
\begin{equation}\label{perry}
\begin{array}{rcl}
H_T(X^a) \otimes_{H_T} \O_{E_{T,t}}(U-a)& \xrightarrow{i_{b,a}^*\otimes \id}& H_T(X^b) \otimes_{H_T} \O_{E_{T,t}}(U-a) \\
&\longrightarrow& H_{T/H}(X^b) \otimes_{H_{T/H}} \O_{E_{T,t}}(U-a) \\
&\xrightarrow{\id \otimes t_{b-a}^*}& H_{T/H}(X^b) \otimes_{H_{T/H}} \O_{E_{T,t}}(U-b) \\
&\longrightarrow& H_T(X^b) \otimes_{H_T} \O_{E_{T,t}}(U-b),
\end{array}
\end{equation}
where the second and final maps are the isomorphism of Proposition \ref{changeh}. Denote the composite by $\phi_{b,a}$.
\end{remark}

\begin{remark}
In the torus-equivariant version of Grojnowski's construction defined in \cite{Chen}, the gluing maps are defined using the map
\begin{equation}\label{sg}
H_{T/T(b)}(X^b) \otimes_{H_{T/T(b)}} \O(U-a) \xrightarrow{1\otimes t_{b-a}^*} H_{T/T(b)}(X^b) \otimes_{H_{T/T(b)}} \O(U-b).
\end{equation}
However, $t_{b-a}^*$ does not always preserve the $H_{T/T(b)}$-algebra structure, because $b-a$ is not always contained in $E_{T(b),t}$. To see this, take $X$ equal to a point, so that $\S(X) = \{T\}$, and set $a = [t_1/2]$ and $b = [0]$. Thus, $T(a) = \Z/2\Z$ and $T(b) = 1$, and in this case, $b - a$ is equal to $[-t_1/2]$ which is obviously not in $E_{T(b),t} = 0$. \par

This is easily fixed by setting $H$ equal to $\langle T(a),T(b) \rangle$ and using the change of groups map associated to $T \rightarrow T/H$. By definition of $T(a)$ and $T(b)$, we have that $b-a \in E_{H,t}$, and it follows that $t_{b-a}^*: \O(U-a) \to \O(U-b)$ is a map of $H_{T/H}$-algebras. Thus, we replace \eqref{sg} with the map
\[
H_{T/H}(X^b) \otimes_{H_{T/H}} \O(U-a) \xrightarrow{1\otimes t_{b-a}^*} H_{T/H}(X^b) \otimes_{H_{T/H}} \O(U-b)
\]
in our account of Grojnowski's construction. Note that $X^b$ is fixed by $H$.
\end{remark}

\begin{proposition}
The collection of maps $\{\phi_{b,a}\}$ satisfies the cocyle condition. 
\end{proposition}

\begin{definition}
We denote by $\G^*_{T,t}(X)$ the sheaf of $\Z/2\Z$-graded $\O_{E_{T,t}}$-algebras which is obtained by gluing together the sheaves $\G^*_{T,t}(X)_{U_a}$ via the maps $\phi_{b,a}$. 
\end{definition}

\begin{remark}
Up to isomorphism, the sheaf $\G^*_{T,t}(X)$ does not depend on the choice of $\U$ since any refinement of $\U$ is also adapted to $\S(X)$. More explicitly, given two covers $\U$ and $\U'$ adapted to $\S(X)$, one may take the common refinement $\U''$ and consider the theory defined using $\U''$. The resulting theory is then naturally isomorphic to those theories corresponding to $\U$ and $\U'$, since the maps induced by the refinement are isomorphisms on stalks.
\end{remark}

The following is Theorem 2.3.8 in \cite{Chen}. We reproduce the proof here as it is important for our main results.

\begin{proposition}\label{kerry}
The functor $X \mapsto \G^*_{T,t}(X)$ is a $T$-equivariant cohomology theory defined on finite $T$-CW complexes with values in coherent sheaves of $\Z/2$-graded $\O_{E_{T,t}}$-algebras.
\end{proposition}

\begin{proof}
That $\G^*_{T,t}(X)$ is a coherent sheaf simply follows from the fact that $X$ is a finite $T$-CW complex, and that $\G^*_{T,t}(X)$ may be computed locally using ordinary cellular cohomology. We show that the construction of $\G_{T,t}(X)$ is functorial in $X$. Let $f: X \to Y$ be a map of finite $T$-CW complexes and let $\U$ be a cover of $E_{T,t}$ which is adapted to $\S(f)$. For $a \in E_{T,t}$, the map $f$ induces a map $f_a: X^a \to Y^a$ by restriction. This induces a map 
\[
f_a^* \otimes \id: H_T(Y^a) \otimes_{H_T} \O(U - a) \to H_T(X^a) \otimes_{H_T} \O(U-a)
\]
for each $U \subset U_a$, which clearly commutes with the restriction maps of the sheaf. It is evident that the collection of such maps for all $a \in E_{T,t}$ glue well, and that identity maps and composition of maps are preserved, by the functoriality of Borel-equivariant cohomology and naturality of the isomorphism of Proposition \ref{changeh}.

Define the reduced theory on pointed finite $T$-CW complexes by setting
\[
\tilde{\G}^*_{T,t}(X,A) := \ker(i^*: \G^*_{T,t}(X/A,*) \to \G^*_{T,t}(\pt)) 
\]
where $i: \pt \hookrightarrow X$ is the inclusion of the basepoint. By naturality of $\G^*_T$, a map $f: X \to Y$ of pointed complexes gives rise to a unique map $f^*: \tilde{\G}^*_{T,t}(X,A) \to \tilde{\G}^*_{T,t}(Y,B)$ on the corresponding kernels. This is functorial for the reasons already set out for single complexes.

Define a suspension isomorphism $\tilde{\G}^{*+1}_{T,t}(S^1 \wedge X) \to \tilde{\G}^{*}_{T,t}(X)$ by gluing the maps 
\[
\sigma_a \otimes \id: \tilde{H}^{*+1}_T(S^1 \wedge X^a) \otimes_{H^*_T} \O(U - a) \to \tilde{H}^{*}_T(X^a) \otimes_{H^*_T} \O(U-a),
\]
where $\sigma_a$ is the suspension isomorphism of Borel-equivariant cohomology. The maps $\sigma_a \otimes \id$ glue well since $\sigma_a$ is natural.

Finally, the properties of exactness and additivity may be checked on stalks
\[
\tilde{\G}_{T,t}(X)_a = \tilde{H}_T(X^a) \otimes_{H_T} \O_{E_{T,t},0}.
\]
This is clear, since Borel-equivariant cohomology satisfies these properties, and tensoring with $\O_{E_{T,t},0} \cong \O_{\t_\C,0}$ is exact.
\end{proof}

\chapter{Elliptic cohomology and double loops}\label{Rezk}

In Section 5 of \cite{Rezk}, Rezk proposed a construction $E_T(X)$ of an equivariant elliptic cohomology theory modeled on the equivariant ordinary cohomology of the double free loop space $L^2X$. It was noted in that paper that a holomorphic version of $E_T(X)$ should serve as a model for Grojnowski's equivariant elliptic cohomology. However, tensoring $E_T(X)$ with holomorphic functions does not behave well, because $E_T(X)$ is often non-Noetherian, even when $X$ is a $T$-orbit. \par

In this chapter, we construct a holomorphic, equivariant sheaf $\Ell_T(X)$ from the double loop space of $X$ by applying an idea of Kitchloo (\cite{Kitch1}) to Rezk's construction. Namely, we replace the equivariant ordinary cohomology of $L^2X$ with the inverse limit over finite subcomplexes of $L^2X$, tensoring with holomorphic functions before applying the limit. If $X$ is a finite $T$-CW complex, then tensoring in this fashion behaves well, because the cohomology ring of a finite CW-complex is finitely generated. In addition, this construction enables us to apply the localisation theorem of equivariant cohomology, which was proved inductively on finite subskeleta in Theorem \ref{localisationn}. In fact, we will establish a certain strengthening of the localisation theorem which allows us to show that $\Ell_T(X)$ is a coherent sheaf. Finally, we give a local description of $\Ell_{T}(X)$ over a chosen elliptic curve $E_t$, which turns out to be exactly Grojnowski's construction. It follows that $\Ell_T(X)$ is a cohomology theory in $X$.

\begin{notation}
We set some notation which will be used in this chapter only. We write $\bT$ for the additive circle $\R/\Z$, so that the canonical isomorphism $\check{T} \otimes \bT  \cong T$ provides us with additive coordinates for $T$. The Lie algebra of $T$ is identified with 
\[
\t := \check{T} \otimes \R,
\]
and the Lie algebra map is identified with
\[
\pi_T: \check{T} \otimes \R \longrightarrow \check{T} \otimes \bT
\] 
induced by the quotient $\pi: \R \twoheadrightarrow \R/\Z$. The kernel of $\pi_T$ is thus $\check{T} \otimes \Z = \check{T}$.
\end{notation}

\section{The $\C^\times \times \SL_2(\Z)$-equivariant complex manifold $E_T$}

Recall the $\C^\times \times \SL_2(\Z)$-equivariant space $\X^+$ that was introduced in Remark \ref{dominic}, and which classifies elliptic curves over $\C$. In this section, we construct a $\C^\times \times \SL_2(\Z)$-equivariant complex manifold $E_T$ as a fiber bundle over $\X^+$, such that the fiber over $t$ is equal to $E_{T,t} = \check{T} \otimes E_t$. The manifold $E_T$ will serve as the underlying object of the $\C^\times \times \SL_2(\Z)$-equivariant sheaf $\Ell_T(X)$ later in the chapter. The idea of the construction is due to Rezk (see Section 2.12 in \cite{Rezk}, also Etingof and Frenkel's paper \cite{EF}). 

\begin{remark}
Consider the semidirect product group 
\[
\SL_2(\Z) \ltimes \bT^2
\]
where $\SL_2(\Z)$ acts on $\bT^2 = \R^2/\Z^2$ by left multiplication. The group operation is given by 
\[
(A',t')(A,t) = (A'A, A^{-1}t' + t),
\]
so that the inverse of $(A,t)$ is $(A^{-1},-At)$. We may think of $\SL_2(\Z) \ltimes \bT^2$ as the group of orientation-preserving diffeomorphisms 
\[
\begin{array}{rcl}
(A,t): \bT^2 &\longrightarrow & \bT^2 \\
s &\longmapsto & As + t.
\end{array}
\]
Let $L^2T$ be the topological group of smooth maps $\bT^2 \to T$, with group multiplication defined pointwise. A diffeomorphism $(A,s)$ acts on a loop $\gamma \in L^2T$ from the left by
\[
(A,t)\cdot \gamma(s) = \gamma(A^{-1}s - At).
\]
\end{remark}

\begin{definition}
Following \cite{Rezk}, define the \textit{extended double loop group of $T$} as the semidirect product 
\[
\G :=  (\SL_2(\Z) \ltimes \bT^2) \ltimes L^2T
\]
with group operation 
\begin{equation}\label{cute}
(A',t',\gamma'(s))(A,t,\gamma(s)) = (A'A, A^{-1}t' + t, \gamma'(As+ t)+ \gamma(s)).
\end{equation}
One may think of an element $(A,t,\gamma) \in \G$ as the automorphism 
\[
\begin{tikzcd}
\bT^2 \times T \ar[r,"{\phi}"] \ar[d,two heads] & \bT^2 \times T \ar[d,two heads] \\
\bT^2 \ar[r,"{(A,t)}"] & \bT^2
\end{tikzcd}
\]
covering the diffeomorphism $(A,t)$ of $\bT^2$, where $\phi(r,s)$ is equal to $\gamma(r) + s$. It is easily verified that the inverse of $(A,t,\gamma(s))$ is equal to $(A^{-1},-At,- \gamma(A^{-1}s-At))$.
\end{definition}

\begin{remark}
For a finite $T$-CW complex $X$, the extended double loop group $\G$ acts on the double loop space 
\[
L^2X := \Map(\bT^2,X)
\] 
by
\begin{equation}
(A,t,\gamma(s))\cdot \gamma'(s) = \gamma(A^{-1}s- At)+ \gamma'(A^{-1}s-At).
\end{equation}
\end{remark}

\begin{definition}[Maximal torus and Weyl group]
Consider the subgroup 
\[
\bT^2 \times T \subset \bT^2 \ltimes L^2T
\]
where the translations $\bT^2$ act trivially on the subgroup of constant loops $T \subset L^2T$. One sees that this is a maximal torus in $\G$ by noting that the intersection of $\SL_2(\Z)$ with $\G^0$ is trivial, and that a nonconstant loop in $L^2T$ does not commute with $\bT^2$. Let $N_\G(\bT^2 \times T)$ be the normaliser of $\bT^2 \times T$ in $\G$. The Weyl group associated to $\bT^2 \times T \subset \G$ is defined to be 
\[
W_\G = W_\G(\bT^2 \times T) := N_\G(\bT^2 \times T)/(\bT^2 \times T).
\]
\end{definition}

\begin{remark}
In the following proposition, we will consider the subgroup
\[
\SL_2(\Z) \ltimes \check{T}^2 \subset \G
\]
where $m \in \check{T}^2$ is identified with the loop $\gamma(s) = ms \in L^2T$ via
\[
\check{T}^2 := \Hom(\bT,T)^2 \cong \Hom(\bT^2,T) \subset L^2T.
\]
The group operation, induced by that of $\G$, is given by
\[
(A',m')(A,m) = (A'A, m'A +m) \in \SL_2(\Z) \ltimes \check{T}^2,
 \]
and the inverse of $(A,m)$ is given by $(A^{-1},-mA^{-1})$. 
\end{remark}

\begin{proposition}
The subgroup $\SL_2(\Z) \ltimes \check{T}^2 \subset \G$ is contained in $N_\G(\bT^2 \times T)$, and the composite map
\[
\SL_2(\Z) \ltimes \check{T}^2 \hookrightarrow N_\G(\bT^2\times T) \twoheadrightarrow W_\G(\bT^2 \times T)
\]
is an isomorphism.
\end{proposition}

\begin{proof}
Let $(A,m) \in \SL_2(\Z) \ltimes \check{T}^2$. A straightforward calculation using \eqref{cute} shows that 
\[
(A,0,m) (1,r,t) (A^{-1},0,-mA^{-1}) = (1,Ar,t + mr) \in \bT^2 \times T,
\]
which proves the first assertion. For the second assertion, it suffices to define an inverse to the composite map of the proposition. Let $g$ be an arbitrary element in $N_\G(\bT^2\times T)$ and let $[g]$ be its image in $W_\G$. By definition of $W_\G$, we may translate $g$ by elements of $\bT^2$, and act on $g$ by constant loops, without changing $[g]$. Therefore, there exists $\gamma \in L^2T$ with $\gamma(0,0) = 1$ such that 
\[
[g] = [(A,0,\gamma)] \in W_\G,
\]
for some $A \in \SL_2(\Z)$. We will now show that $\gamma \in \check{T}^2$, and finally that $[g] \mapsto (A,\gamma)$ is a well defined inverse to the composite map. Using \eqref{cute} again, for any $(r,t) \in \bT^2 \times T$, we have
\begin{equation}\label{afff}
(A,0,\gamma(s)) (r,t)(A,0,- \gamma(s)) = (1,Ar,\gamma(r+A^{-1}s)+ t - \gamma(A^{-1}s))  \in \bT^2 \times T.
\end{equation}
It follows that $\gamma(r+A^{-1}s) - \gamma(A^{-1}s)$ does not depend on $s$. Thus,
\[
\gamma(r+A^{-1}s)- \gamma(A^{-1}s) = \gamma(r)
\]
for all $r,s \in \bT^2$, and setting $s = As'$ shows that $\gamma(r)+ \gamma(s') = \gamma(r+s')$ for all $r,s' \in \bT^2$. Therefore, $\gamma$ is a group homomorphism, which means that it lies in $\check{T}^2$. The map $[g] \mapsto (A,\gamma)$ is well defined, since $g \in \bT^2 \times T$ allows us to choose $A = 1$ and $\gamma = 1$, and is evidently a group homomorphism which is inverse to the composite map of the proposition. The completes the proof.
\end{proof}

\begin{remark}[Weyl action]
It follows directly from equation \eqref{afff} that the action of $\SL_2(\Z) \ltimes \check{T}^2$ on $\bT^2 \times T$ is given by
\[
(A,m) \cdot (r,t) = (Ar,t+mr).
\]
The induced action of $\SL_2(\Z) \ltimes \check{T}^2 \subset \G$ on the complex Lie algebra $\C^2 \times \t_\C$ is given by the same formula, in which case we write it as
\[
(A,m) \cdot (t,x) = (At,x+mt).
\]
\end{remark}

\begin{remark}[The space $E_T$]\label{zeta}
Since $\SL_2(\Z)$ preserves the subspace $\X^+ \subset \C^2$, the action of $\SL_2(\Z) \ltimes \check{T}^2$ on $\C^2 \times \t_\C$ preserves $\X^+ \times \t_\C$. Since the action of $\check{T}^2$ is free and properly discontinuous, the quotient map
\[
\zeta_T: \X^+ \times \t_\C \twoheadrightarrow \check{T}^2 \backslash (\X^+ \times \t_\C) =: E_T
\]
is a complex analytic map and the quotient space is a complex manfiold. The residual action of $\SL_2(\Z)$ descends to $E_T$, since $\check{T}^2$ is a normal subgroup. The projection $p_1: \X^+ \times \t_\C \twoheadrightarrow \X^+$ induces a fiber bundle $E_T \twoheadrightarrow \X^+$, and one sees easily that the fiber of $E_T$ over $t \in \X^+$ is given by 
\[
E_{T,t} = (\check{T}t_1 + \check{T}t_2)\backslash \t_\C = \check{T} \otimes \C/\Lambda_t.
\]
We denote by $\zeta_{T,t}$ the restriction of the map $\zeta_T$ to the fiber over $t \in \X^+$. 
\end{remark}

\begin{remark}
We can view $\X^+$ as a parameter space of complex structures on $\t \times \t$ and $T \times T$. Denote by $\xi_T$ the map
\[
\begin{array}{rcl}
\X^+ \times \t \times \t &\longrightarrow& \X^+ \times \t_\C \\
(t_1,t_2,x_1,x_2) &\mapsto & (t_1,t_2,x_1t_1 + x_2t_2). 
\end{array}
\]
The restriction of $\xi_T$ to the fiber over $t \in \X^+$ is the isomorphism
\[
\xi_{T,t}: \check{T} \otimes \R^2 \cong \check{T} \otimes \C
\]
of the underlying real vector spaces induced by the complex linear structure $ \R t_1 + \R t_2 = \C$ on $\R^2$. Taking the quotient by $\check{T}^2$ induces a commutative diagram
\begin{equation}\label{deer}
\begin{tikzcd}
\X^+ \times \t \times \t \ar[r,"\xi_T"] \ar[d,two heads,"\id \times \pi"]& \X^+ \times \t_\C \ar[d, two heads, "{\zeta_{T}}" ]\\
\X^+ \times T \times T \ar[r,"\chi_T"] & E_{T}
\end{tikzcd}
\end{equation} 
of the underlying real manifolds. The restriction of $\chi_T$ to the fiber over $t \in \X^+$ is the isomorphism
\[
\chi_{T,t}: \check{T} \otimes \bT^2 \cong \check{T} \otimes E_t
\]
of real Lie groups induced by the complex manifold structure
\[
\bT^2 \cong (\R t_1 + \R t_2)/(\Z t_1 + \Z t_2) = \C/\Lambda_t =: E_t.
\]
%By the Riemann-Hilbert correspondence, for each $t \in \X^+$ there is a bijection between $T \times T \cong \Hom(\pi_1(E_t),T)$ and the set of gauge equivalence classes of flat connections on principal $T$-bundles over $E_t$. 
\end{remark}

%\begin{definition}
%An \textit{open covering} of an action groupoid $G \backslash \backslash M$ is a open covering of the coarse quotient $G \backslash M$. This induces a decomposition of $G \backslash \backslash M$ into open subgroupoids.
%\end{definition}

\section{An open cover of $E_T$ adapted to $X$}

In this section, we begin by defining, for a finite $T$-CW complex $X$, an open cover of the compact Lie group $T \times T$ which is adapted to $X$. We show that such a cover exists, and that it induces an open cover of the total space $E_T$ via the isomorphism $\X^+ \times T \times T \cong E_T$, such that the restriction of the cover to $E_{T,t}$ is adapted to $X$ in the sense of Definition \ref{ross}. Finally, we establish some properties of the open cover which will be useful in later sections, and we give an example of a cover adapted to the representation sphere $S_\lambda$. 

\begin{remark}[Ordering on $T\times T$]
If $\S$ is a finite set of closed subgroups of $T$, we can define a relation on the points of $T\times T$ by saying that $(a_1,a_2) \leq_\S (b_1,b_2)$ if $(b_1,b_2) \in H \times H$ implies $(a_1,a_2) \in H\times H$, for any $H \in \S$. This relation is obviously reflexive and transitive, but not symmetric.
\end{remark}

\begin{notation}
If $X$ is a finite $T$-CW complex, let $\S(X)$ be the finite set of isotropy subgroups of $X$. If $f: X \to Y$ is a map of finite $T$-CW complexes, let $\S(f)$ be the finite set of isotropy subgroups which occur in either $X$ or $Y$. An open set $U$ in $T \times T$ is \textit{small} if $\pi^{-1}(U)$ is a disjoint union of connected components $V$ such that $\pi|_V: V \to U$ is a bijection for each $V$.
\end{notation}

\begin{definition}[Open cover of $T \times T$]\label{rachel}
An open cover $\U = \{U_{a_1,a_2}\}$ of $T \times T$ indexed by the points of $T\times T$ is said to be \textit{adapted to $\S$} if it has the following properties:
\begin{enumerate}
\item $(a_1,a_2) \in U_{a_1,a_2}$, and $U_{a_1,a_2}$ is small.
\item If $U_{a_1,a_2} \cap U_{b_1,b_2} \neq \emptyset$, then either $(a_1,a_2) \leq_\S (b_1,b_2)$ or $(b_1,b_2) \leq_\S (a_1,a_2)$.
\item If $(a_1,a_2) \leq_\S (b_1,b_2)$, and for some $H \in \S$, we have $(a_1,a_2) \in H \times H$ but $(b_1,b_2) \notin H\times H$, then $U_{b_1,b_2} \cap H \times H = \emptyset$.
\item Let $(a_1,a_2)$ and $(b_1,b_2)$ lie in $H \times H$ for some $H \in \S$. If $U_{a_1,a_2} \cap U_{b_1,b_2} \neq \emptyset$, then $(a_1,a_2)$ and $(b_1,b_2)$ belong to the same connected component of $H\times H$.
\end{enumerate}
If $\S = \S(X)$ and $\U$ is adapted to $\S$ then we say that $\U$ is adapted to $X$. If $\S = \S(f)$ and $\U$ is adapted to $\S$ then we say that $\U$ is adapted to $f$. If $\S$ is understood, then we just write $\leq$ for $\leq_\S$. 
\end{definition}

Our proof of the following result is based on the proof of Proposition 2.5 in \cite{Rosu03}.

\begin{lemma}\label{construu}
For any finite set $\S$ of subgroups of $T$, there exists an open cover $\U$ of $\E_T$ adapted to $\S$. Any refinement of $\U$ is also adapted to $\S$.
\end{lemma}

\begin{proof}
Since a compact abelian group has finitely many components, the set 
\[
\S^0 = \{D \subset T \times T \, | \, \text{$D$ is a component of $H \times H$ for some $H \in \S$}\}
\]
is finite. Let $d$ be the metric on $T \times T$ defined by
\[
d((a_1,a_2),(b_1,b_2)) := \min \{d_{\t\times \t}((a_1,a_2) + m,(b_1,b_2) + m') \, | \, m,m'\in \check{T} \times \check{T}\}
\]
where $d_{\t\times \t}$ denotes the Euclidean metric on $\t \times \t$. If $(a_1,a_2) \in D$ for all $D \in \S^0$, then define $U_{a_1,a_2}$ to be an open ball centered at $(a_1,a_2)$ with radius $r = \frac{1}{2}$. Otherwise, define $U_{a_1,a_2}$ to be an open ball centered at $(a_1,a_2)$, with radius  
\[
r = \frac{1}{2} \min \{D\, | \, (a_1,a_2) \notin D \},
\]
where 
\[
d((a_1,a_2),D) = \min \{d((a_1,a_2),(b_1,b_2))\, | \, (b_1,b_2)\in D\}.
\]
The open cover $\U$ of $T \times T$ thus constructed clearly satisfies the first condition of an adapted cover.  

Furthermore, if there exist distinct components $D,D' \in \S^0$ such that $(a_1,a_2)$ is in $D$ but not in $D'$, and $(b_1,b_2)$ is in $D'$ but not in $D$, then $U_{a_1,a_2} \cap U_{b_1,b_2}$ is empty by construction. If $D$ and $D'$ correspond to distinct elements of $\S$, then $(a_1,a_2)$ and $(b_1,b_2)$ do not relate under the ordering, and the previous statement implies the contrapositive of the second condition. If $D$ and $D'$ correspond to the same element of $\S$, then it implies the contrapositive of the fourth condition. 

The third condition holds simply because $(b_1,b_2) \notin H\times H$ always implies that $U_{b_1,b_2} \cap H\times H = \emptyset$, by construction. It is clear that any refinement of $\U$ will also satisfy all four conditions.
\end{proof}

\begin{remark}
Let $\U$ be a cover adapted to $\S(X)$. We will sometimes refer to $\U$ as being adapted to $X$, since $\S(X)$ is completely determined by $X$.
\end{remark}

\begin{lemma}\label{hard2}
Let $\U$ be an open cover of $T\times T$ adapted to $\S$. If $(b_1,b_2) \in U_{a_1,a_2}$, then $(a_1,a_2) \leq (b_1,b_2)$.
\end{lemma}

\begin{proof}
Suppose that $(b_1,b_2) \in U_{a_1,a_2}$ and $(a_1,a_2) \leq (b_1,b_2)$ does not hold. This implies two things. Firstly, by the second property of an adapted cover, we have $(b_1,b_2) \leq (a_1,a_2)$. Secondly, by definition of the relation, there must exist some $H \in \S$ such that $(b_1,b_2) \in H \times H$ and $(a_1,a_2) \notin H \times H$. Together, the two statements imply that $U_{a_1,a_2} \cap H \times H = \emptyset$, by the third property of an adapted cover. This contradicts the assumption that $(b_1,b_2) \in U_{a_1,a_2}$, since $(b_1,b_2) \in H \times H$.
\end{proof}

\begin{definition}\label{siss}
Let $\U = \{U_{a_1,a_2}\}$ be an open cover adapted to $\S$. We denote by $V_{x_1,x_2} \subset \t \times \t$ the open subset which is the component of $\pi^{-1}(U_{\pi(x_1,x_2)})$ containing $(x_1,x_2)$.
\end{definition}

The following lemma is a strengthening of the fourth property of an adapted cover, which we will need in the next section. 

\begin{lemma}\label{hard3}
Let $\U$ be a cover of $T\times T$ adapted to $\S$. Let $(a_1,a_2),(b_1,b_2) \in T\times T$ with open neighbourhoods $U_{a_1,a_2},U_{b_1,b_2} \in \U$, and let $(x_1,x_2),(y_1,y_2) \in \t \times \t$ such that $\pi(x_i) = a_i$ and $\pi(y_i) = b_i$. Let $H \in \S$ and suppose $(a_1,a_2),(b_1,b_2) \in H\times H$. If
\[
V_{x_1,x_2} \cap V_{y_1,y_2} \neq \emptyset, 
\]
then $(x_1,x_2)$ and $(y_1,y_2)$ lie in the same component of $\pi^{-1}(H\times H)$.
\end{lemma}

\begin{proof}
Since $V_{x_1,x_2} \cap V_{y_1,y_2} \neq \emptyset$, we have $U_{a_1,a_2} \cap U_{b_1,b_2} \neq \emptyset$. Therefore, $(a_1,a_2)$ and $(b_1,b_2)$ lie in the same component $D$ of $H\times H$, by the fourth property of an adapted cover, so that $(x_1,x_2),(y_1,y_2) \in \pi^{-1}(D)$. We have
\begin{equation}
\begin{array}{rcl}
\pi^{-1}(D) &=& \{(x_1,x_2)\} + \check{T} \times \check{T} + \mathrm{Lie}(H) \times \mathrm{Lie}(H)\\
&=& \{(x_1,x_2)\} + \check{T}/\check{H} \times \check{T}/\check{H} + \mathrm{Lie}(H) \times \mathrm{Lie}(H). \\
\end{array}
\end{equation}
Suppose that $(x_1,x_2)$ and $(y_1,y_2)$ lie in different components of $\pi^{-1}(D)$. Then $\check{T}/\check{H}$ is nontrivial, and we have
\[
(x_1,x_2) = (x_1 + m_1 + h_1, x_2 + m_2 + h_2) \quad \mathrm{and} \quad (y_1,y_2) = (x_1 + m_1' + h_1', x_2 +m_2' + h_2')
\]
for some $h_1,h_2,h_1',h_2' \in$Lie$(H)$ and distinct $(m_1,m_2),(m_1',m_2') \in \check{T}/\check{H}\times \check{T}/\check{H}$.

Let $d = d_{\t\times \t}$ be the metric induced by the Euclidean inner product on $\t \times \t$. We have
\[
\begin{array}{rcl}
d((x_1,x_2),(y_1,y_2)) &=& |(x_1,x_2)-(y_1,y_2)|  \\
&=& |(m_1-m_1' + h_1 - h_1',m_2-m_2' + h_2 - h_2')|\\
&\geq &|m_1-m_1',m_2-m_2' | \\
\end{array}
\]
where the inequality holds since $\check{T}/\check{H}$ is orthogonal to Lie$(H)$. Since $(m_1,m_2) \neq (m_1',m_2')$, we have
\[
d((x_1,x_2),(y_1,y_2)) \geq 1.
\]

By the first property of an adapted cover, $U_{a_1,a_2}$ is small, which means that $V_{x_1,x_2}$ is contained in the interior of a ball at $(x_1,x_2)$ with radius $\frac{1}{2}$. The same is true for $V_{y_1,y_2}$. But this means that $V_{x_1,x_2}$ and $V_{y_1,y_2}$ cannot intersect, since $d((x_1,x_2),(y_1,y_2)) \geq 1$, so we have a contradiction. Therefore, $(x_1,x_2)$ and $(y_1,y_2)$ must lie in the same connected component of $\pi^{-1}(D)$, and hence the same connected component of $\pi^{-1}(H\times H)$.
\end{proof}

\begin{definition}\label{defo}
Let $\U$ be an open cover of $T \times T$ adapted to $\S$. For each $(t,x)  \in \X^+ \times \t_\C$, writing $x = x_1t_1 + x_2t_2$, we define an open subset
\[
V_{t,x} := \xi_T(\X^+ \times V_{x_1,x_2}) \subset \X^+ \times \t_\C,
\]
so that $(t,x) \in V_{t,x}$. Note that $V_{t,x} = V_{t',x'}$ whenever we have $x = x_1t_1 +x_2t_2$ and $x' = x_1t_1'+x_2t_2'$. The set 
\[
\{V_{t,x}\}_{(t,x) \in \X^+ \times \t_\C}
\]
is an open cover of $\X^+ \times \t_\C$ (with some redundant elements). The set
\[
\{\zeta_T(V_{t,x})\}_{(t,x) \in \X^+ \times \t_\C}
\]
is an open cover of $E_T$.
\end{definition}

\begin{definition}\label{sag}
Let $\U$ be an open cover of $T \times T$ adapted to $\S$. Given $a \in E_{T,t}$, define the open subset
\[
U_a := \zeta_{T}(V_{t,x}) \cap E_{T,t}
\]
where $x \in \t_\C$ is any element such that $a= \zeta_{T,t}(x)$, so that $a \in U_a$. The set $\{U_a\}_{a\in E_{T,t}}$ is an open cover of $E_{T,t}$.
\end{definition}

\begin{lemma}
Let $\U$ be an open cover of $T \times T$ which is adapted to $\S$. The open cover of $E_{T,t}$ in Definition \ref{sag} is adapted to $\S$ in the sense of Definition \ref{ross}.
\end{lemma}

\begin{proof}
We have
\[
\begin{array}{rcl}
U_a &:=& \zeta_T (V_{t,x}) \cap E_{T,t}  \\
&=& (\zeta_T \circ \xi_T)(\X^+ \times V_{x_1,x_2}) \cap E_{T,t}  \\
&=& (\chi_T \circ (\id_{\X^+} \times \pi_{T\times T}))(\X^+ \times V_{x_1,x_2}) \cap E_{T,t}  \\
&=& \chi_T(\X^+ \times U_{a_1,a_2}) \cap E_{T,t} \\
&=& \chi_{T,t}(U_{a_1,a_2})
\end{array}
\]
where $a_i = \pi_T(x_i)$. Therefore, the open cover $\{U_a\}_{a\in E_{T,t}}$ of $E_{T,t}$ corresponds exactly to the open cover $\U$ of $T \times T$ via the isomorphism
\[
\chi_{T,t}: T \times T \cong E_{T,t}
\]
of real Lie groups. We have identifications
\[
H \times H \cong \Hom(\hat{H},\bT)^2 \cong \Hom(\hat{H},\bT^2) \cong \Hom(\hat{H},E_t) = E_{H,t}
\]
and a commutative diagram
\begin{equation}\label{deery}
\begin{tikzcd}
H \times H \ar[r,"{\cong}"] \ar[d,hook] & E_{H,t} \ar[d,hook] \\
T \times T \ar[r,"\chi_{T,t}"] & E_{T,t}. \\
\end{tikzcd}
\end{equation}
From the diagram, it is clear that the properties of an adapted cover in the sense of Definition \ref{ross} are equivalent to the properties in Definition \ref{rachel}. Since $\U$ is adapted to $\S$ in the sense of Definition \ref{rachel}, the result now follows.
\end{proof}

\begin{example}\label{coverex}
Let $T = \R/\Z$, and for $\lambda \in \hat{T}$ set $X$ equal to the representation sphere $S_\lambda$ associated to $\lambda$ (see Example \ref{repsphere}). Recall that $X$ has a $T$-CW complex structure with a $0$-cell $T/T \times \D_N^0$ at the north pole, a $0$-cell $T/T \times \D_S^0$ at the south pole, and a 1-cell $T \times \D^1$ of free orbits connecting the two poles. Thus,
\[
\S(X) = \{T, 1\}.
\]
We can now describe all of the relations $\leq_{\S(X)}$ between the points of $T \times T$ as follows.
\begin{itemize}
\item $(0,0) \leq (a_1,a_2)$ for all $(a_1,a_2) \in T \times T$
\item $(a_1,a_2) \leq (b_1,b_2)$ for all $(a_1,a_2),(b_1,b_2) \in T \times T - \{(0,0)\}$
\end{itemize}
As in the proof of Lemma \ref{construu}, we can easily construct an open cover of $T \times T$ which is adapted to $\S(X)$. Note that $\S^0 = \S$ in this case. Let $\U$ denote the open cover consisting of open balls $U_{a_1,a_2}$ centered at $(a_1,a_2)$ with radius 
\[
r_{a_1,a_2} = \begin{cases} \frac{1}{2} \sqrt{a_1^2 + a_2^2}  &\mbox{if } a_1 \neq 0 \mbox{ or } a_2 \neq 0 \\ 
\frac{1}{2} & \mbox{if } (a_1,a_2) = (0,0)\end{cases},
\]
where we have identified $a_1,a_2$ with their unique representatives in $[ 0,1)$. It is easily verified that $\U$ is a cover adapted to $\S(X)$. Now, for $a = \chi_{T,t}(a_1,a_2)$ and $x = x_1t_1 +x_2t_2$ such that $\zeta_{T,t}(x) = a$, the open set
\[
V_{x_1,x_2} \subset \R^2
\]
is an open ball of radius $\frac{1}{2}$ if $(x_1,x_2) \in \Z^2$, and an open neighbourhood not intersecting $\Z^2$ otherwise. 
\end{example}

\section{The equivariant cohomology of double loop spaces}\label{low}

The double loop space $L^2X$ of a $T$-CW complex $X$ is equipped, via the action of $\G$, with an action of $\bT^2 \times T$. As we noted at the beginning of the chapter, we would like to consider the $\bT^2 \times T$-equivariant cohomology of $L^2X$, and to somehow construct a holomorphic sheaf from this. 

\begin{remark}
Let $X$ be a finite $T$-CW complex. By Theorem 1.1 in \cite{LMS}, the double loop space $L^2X = \Map(\bT^2,X)$ of $X$ is weakly $\bT^2 \times T$-homotopy equivalent to a $\bT^2 \times T$-CW complex $Z$.\footnote{To apply this theorem to our situation, set $G = \bT^2 \times T$. Now let $G$ act on $\bT^2$ via the projection to the first factor, and on $X$ via the projection to the second factor.} From now on, when we speak of a $\bT^2 \times T$-CW structure on $L^2X$, we mean the replacement complex $Z$. In this situation, we will abuse notation and write $L^2X$ for $Z$.
\end{remark}

\begin{definition}\label{varlimm}
Define the sheaf of $\O_{\X^+ \times \t_\C}$-algebras
\[
\H^*_{L^2T}(L^2X) := \varprojlim_{Y \subset L^2X} \H^*_{\bT^2 \times T}(Y)_{\X^+ \times \t_\C}
\]
where the inverse limit runs over all finite $\bT^2 \times T$-CW subcomplexes $Y$ of $L^2X$. 
\end{definition}

\begin{remark}
The inverse limit in the category of sheaves exists, and is computed in the category of presheaves. Therefore, the value of the inverse limit sheaf $\H^*_{L^2T}(L^2X)$ on an open subset $U \subset \X^+ \times \t_\C$ is 
\[
\varprojlim_{Y \subset L^2X} H^*_{\bT^2 \times T}(Y) \otimes_{H_{\bT^2 \times T}} \O_{\X^+ \times \t_\C}(U).
\]
However, it is not true in general that the stalk of the inverse limit is the inverse limit of the stalks.
\end{remark}

\begin{remark}
We have that $\H^*_{L^2T}(L^2*) = \O_{\X^+ \times \t_\C}$, by construction.
\end{remark}

Our aim now is to obtain localisation-type results for the ring $\H^*_{L^2T}(L^2X)$, in the spirit of Theorem \ref{localisation}. In fact, we apply that theorem directly and show that it holds not only on stalks, but also on the open sets in a cover adapted to $X$. Since $\H^*_{L^2T}(L^2X)$ is defined over $\X^+ \times \t_\C$, we use the open cover $\{V_{t,x}\}$ of that space which was defined in the previous section. Before we can prove the localisation-type results, we will need several technical results about the groups $T(t,x)$ associated to the points $(t,x) \in \X^+ \times \t_\C$, and about the fixed points subspaces $L^2X^{t,x} \subset L^2X$ corresponding to those groups and how they interact with the open sets $V_{t,x}$.

\begin{definition}
Let $(t,x) \in \X^+ \times \t_\C$. Define the intersection
\[
T(t,x) = \bigcap_{(t,x) \in \mathrm{Lie}(H)_\C} H
\]
of closed subgroups $H \subset \bT^2 \times T$. For a $\bT^2 \times T$-space $Y$, denote by $Y^{t,x}$ the subspace of points fixed by $T(t,x)$.
\end{definition}

\begin{remark}[Explicit description of $T(t,x)$]\label{push}
Let $(t,x) \in \X^+ \times \t_\C$ and write $x = x_1t_1 + x_2t_2$. For a closed subgroup $H \subset \bT^2 \times T$, the definition of $\X^+$ implies that
\begin{gather*}
(t,x) = (t_1,t_2,x_1t_1+x_2t_2) = (1,0,x_1) t_1 + (0,1,x_2) t_2 \\ \in \, \mathrm{Lie}(H)\otimes_\R \C = \mathrm{Lie}(H) t_1 + \mathrm{Lie}(H) t_2
\end{gather*}
if and only if 
\[
(1,0,x_1),(0,1,x_2) \in \mathrm{Lie}(H).
\] 
By definition of $T(t,x)$, we therefore have
\[
T(t,x) = \langle \pi((r_1,0,x_1 r_1)),\pi((0,r_2,x_2 r_2)) \rangle_{r_1,r_2 \in \R}.
\]
Since the intersection of $T(t,x)$ with $T \subset \bT^2 \times T$ consists exactly of those elements of $T(t,x)$ for which $r_1,r_2 \in \Z$, we have 
\[
T(t,x) \cap T = \langle \pi(x_1),\pi(x_2) \rangle.
\]
\end{remark}

\begin{remark}[Explicit description of $T(a)$]\label{pusha}
Recall the definition of $T(a)$ in Definition \ref{teed}. Let $a = \zeta_T(t,x)$ and write $x = x_1t_1 + x_2t_2$. By diagrams \eqref{deer} and \eqref{deery}, it is clear that $E_{H,t}$ contains $a$ if and only if 
\[
\pi(x_1),\pi(x_2) \in H.
\]
Therefore, $T(a)$ is the closed subgroup 
\[
\langle \pi(x_1), \pi(x_2) \rangle \subset T,
\] 
and by Remark \ref{push}, we have that $T(a) = T \cap T(t,x)$.
\end{remark}

\begin{lemma}\label{surj}
Let $\zeta_T(t,x) = a$. There is a short exact sequence of compact abelian groups 
\[
0 \to T(a) \to T(t,x) \to \bT^2 \to 0
\]
where $T(t,x) \to \bT^2$ is the map induced by the projection $\bT^2 \times T \twoheadrightarrow \bT^2$.
\end{lemma}

\begin{proof}
The projection of $T(t,x) \subset \bT^2 \times T$ onto $\bT^2$ is surjective by the description in Remark \ref{push}, and has kernel $T(t,x) \cap T = T(a)$ by Remark \ref{pusha}.
\end{proof}

\begin{notation}
Write
\[
p_a: T \twoheadrightarrow T/T(a) \qquad \mathrm{and}  \qquad p_{t,x}: \bT^2 \times T \twoheadrightarrow (\bT^2 \times T)/T(t,x)
\]
for the quotient maps, and let 
\[
\iota_{(0,0}: T \hookrightarrow \bT^2 \times T
\]
denote the inclusion of the fiber over $(s_1,s_2) \in \bT^2 = (\R/\Z)^2$. %In particular, the inclusion map $\iota_1$ is the inclusion of the fixed points under the action $(\tau,x) \mapsto (\tau,x +m\tau +n)$ of $(m,n) \in \check{T}\times \check{T}$.
\end{notation}

\begin{remark}
It follows from Lemma \ref{surj} that there is a commutative diagram 
\begin{equation}\label{rect}
\begin{tikzcd}
T(a) \ar[r,hook] \ar[d,hook] & T \ar[r, two heads,"p_a"] \ar[d, "\iota_{(0,0)}",hook] & T/T(a) \ar[d, "\nu", dashed] \\
T(t,x) \ar[r,hook] \ar[d,two heads] & \bT^2 \times T \ar[r, two heads,"{p_{\tau,x}}"] \ar[d,two heads] & (\bT^2 \times T)/T(t,x) \ar[d] \\
\bT^2 \ar[r,equal] & \bT^2 \ar[r] & 0
\end{tikzcd}
\end{equation}
where $\nu$ is induced by $\iota_{(0,0)}$.
\end{remark}

\begin{lemma}\label{groupss} 
The map $\nu$ of diagram \eqref{rect} is an isomorphism.
\end{lemma}

\begin{proof}
By Lemma \ref{surj}, the left hand column is a short exact sequence. It is clear that the middle column is a short exact sequence, as are all rows. It now follows from a standard diagram chase that $\nu$ is an isomorphism.
\end{proof}

\begin{remark}[Explicit description of $L^2X^{t,x}$]\label{pushy}
The description in Remark \ref{push} allows us to explicitly describe $L^2X^{t,x}$, using $x = x_1t_1 + x_2t_2$. By this description, a loop $\gamma \in L^2X$ is fixed by $T(t,x)$ if and only if
\[
\gamma(s_1,s_2) = \pi((r_1,0,x_1 r_1)) \cdot \gamma(s_1,s_2) = \pi( x_1 r_1) \cdot \gamma(s_1 - \pi(r_1),s_2)
\]
and
\[
\gamma(s_1,s_2) = \pi((0,r_2,x_2 r_2)) \cdot \gamma(s_1,s_2) = \pi(x_2 r_2) \cdot \gamma(s_1,s_2- \pi(r_2))
\]
for all $s=(s_1,s_2) \in \bT^2$ and all $r_1,r_2 \in \R$. By setting $s_1 = s_2 = 0$, one sees that this is true if and only if 
\[
\gamma(s_1,s_2) = \pi(x_1 s_1 + x_2s_2) \cdot \gamma(0,0),
\]
where $\gamma(0,0) \in X^a$ for $a = \zeta_T(t,x)$. Therefore, we have
\[
L^2X^{t,x} = \{ \gamma \in L^2X \, | \,  \gamma(s_1,s_2) = \pi(x_1 s_1 +  x_2 s_2)\cdot z, \, z \in X^a\}.
\]
Although $s_1,s_2$ are elements of $\bT = \R/\Z$, the loop $\pi(x_1 s_1 +  x_2 s_2)\cdot z$ is well defined since $\pi$ is a homomorphism and $z$ is fixed by both $\pi(x_1)$ and $\pi(x_2)$. 
\end{remark}

\begin{remark}
The map 
\[
\begin{tikzcd}
ev: L^2X \ar[r]& X
\end{tikzcd}
\]
given by evaluating a loop at $(0,0)$ is evidently $T$-equivariant and continuous.
\end{remark}

\begin{lemma}\label{equal}
The map $ev$ induces a homeomorphism
\[
\begin{tikzcd}
ev_{t,x}: L^2X^{t,x} \ar[r,"\cong"]& X^{a}
\end{tikzcd}
\]
which is natural in $X$, and equivariant with respect to $\nu$. 
\end{lemma}

\begin{proof}
This is evident by Remark \ref{pushy}.
\end{proof}

\begin{example}\label{loopex}
We calculate $LX^{t,x}$ in the example of the representation sphere $X = S_\lambda$ associated to $\lambda \in \hat{T} = \Z$. For $a \in E_{T,t}$, let $a_1,a_2$ denote the preimage of $a$ under $\chi_{T,t}$. Then 
\[
T(a) = \langle a_1,a_2\rangle \subset T,
\]
and we have
\[
X^a =  \begin{cases} \{N,S\}  &\mbox{if } (a_1,a_2) \neq (0,0) \\ 
S_\lambda & \mbox{if } (a_1,a_2) = (0,0)\end{cases}. 
\]
Let $x \in \zeta_{T,t}^{-1}(a)$ and write $x = x_1t_1 + x_2t_2$. We have
\[
L^2X^{t,x} = \begin{cases} \{N,S\}  &\mbox{if } (x_1,x_2) \notin \Z^2 \\ 
\{ \pi(x_1s_1+x_2s_2)\cdot z \, | \,z \in S_\lambda \} & \mbox{if } (x_1,x_2) \in \Z^2 \end{cases}
\]
where $N,S$ denote constant loops. A double loop $\pi(x_1s_1+x_2s_2)\cdot_\lambda z$, for $z \neq N,S$, wraps the first loop around the sphere $\lambda x_1$ times, and the second loop around $\lambda x_2$ times, where the loops run parallel to the equator. The direction of the loop corresponds to the sign of the $\lambda x_i$. 
\end{example}

\begin{lemma}\label{hard}
Let $X$ be a finite $T$-CW complex and let $\U$ be an open cover of $T \times T$ adapted to $\S(X)$. Let $(t,x),(t',y) \in \X^+ \times \t_\C$, with $a = \zeta_T(t,x)$ and $b = \zeta_T(t',y)$. We have the following.
\begin{enumerate}
\item If $V_{t,x} \cap V_{t',y} \neq \emptyset$, then either $X^b \subset X^a$ or $X^a \subset X^b$. 
\item If $V_{t,x} \cap V_{t',y} \neq \emptyset$ and $X^b \subset X^a$, then $L^2X^{t',y} \subset L^2X^{t,x}$.
\end{enumerate}
\end{lemma}

\begin{proof}
Write $x = x_1t_1 + x_2t_2$, $y = y_1t_1' + y_2t_2'$ and $a_i = \pi(x_i)$, $b_i = \pi(y_i)$. Since $V_{t,x} \cap V_{t',y} \neq \emptyset$, we have
\[
V_{x_1,x_2} \cap V_{y_1,y_2} \neq \emptyset
\]
by definition. This implies that 
\[
U_{a_1,a_2} \cap U_{b_1,b_2} \neq \emptyset,
\]
so by the second property of an adapted cover, either $(a_1,a_2) \leq (b_1,b_2)$ or $(b_1,b_2) \leq (a_1,a_2)$. This implies that either $X^b \subset X^a$ or $X^a \subset X^b$, which yields the first part of the result.

For the second part, assume that $X^b \subset X^a$, and let $\gamma \in L^2X^{t',y}$. Let $z = \gamma(0,0)$. By the description in Remark \ref{pushy}, we have
\[
\gamma(s_1,s_2) = \pi(y_1s_1 + y_2 s_2)\cdot z.
\]
Let $H \subset T$ be the isotropy group of $z \in X^b \subset X^a$, so that $H \in \S(X)$ and $a_i,b_i \in H$ for $i = 1,2$. The condition 
\[
V_{x_1,x_2} \cap V_{y_1,y_2} \neq \emptyset
\]
implies by Lemma \ref{hard3} that $(x_1,x_2)$ and $(y_1,y_2)$ lie in the same component of $\pi^{-1}(H\times H)$, since $\U$ is adapted to $\S(X)$. Therefore,
\[
(x_1 - y_1, x_2 - y_2) \in \pi^{-1}(H\times H)^0 = \mathrm{Lie}(H) \times \mathrm{Lie}(H).
\]
which implies that $z$ is fixed by $\pi((x_1-y_1)r_1)$ and $\pi((x_2-y_2)r_2)$ for all $r_1,r_2 \in \R$.

We can now write
\[
\begin{array}{rcl}
\gamma(s_1,s_2) &=& \pi( y_1 s_1 +  y_2 s_2)\cdot z \\
&=& \pi(y_1 s_1+  y_2 s_2) \cdot (\pi((x_1 - y_1)s_1 + (x_2 - y_2)s_2)\cdot z) \\
&=& \pi(x_1 s_1 + x_2 s_2)\cdot z,
\end{array}
\]
which is a loop in $L^2X^{t,x}$ since $z \in X^a$. This yields the second part of the result.
\end{proof}

\begin{lemma}\label{hay}
Let $X$ be a finite $T$-CW complex and let $\U$ be a cover adapted to $\S(X)$. If $(t',y) \in V_{t,x}$, then $L^2X^{t',y} \subset L^2X^{t,x}$.
\end{lemma}

\begin{proof}
Write $a = \zeta_T(t,x)$, $b = \zeta_T(t',y)$, $x = x_1t_1 + x_2t_2$, $y = y_1t_1' + y_2t_2'$, $b_i = \pi(y_i)$, and $a_i = \pi(x_i)$. Since $(t',y) \in V_{t,x}$, we have 
\[
(y_1,y_2) \in V_{x_1,x_2} \subset \t \times \t
\] 
and so 
\[
(b_1,b_2) \in U_{a_1,a_2} \subset T\times T.
\]
Therefore $(a_1,a_2) \leq (b_1,b_2)$ by Lemma \ref{hard2}, from which it follows that $X^b \subset X^a$, since $\U$ is adapted to $\S(X)$. Lemma \ref{hard} yields the result.
\end{proof}

\begin{example}\label{incex}
We examine the inclusions of Lemmas \ref{hard} and \ref{hay} in the example of the representation sphere $X = S_\lambda$. Let $(t,x), (t',y) \in \X^+ \times \t_\C$, write $x = x_1t_1+x_2t_2$, $y=y_1t_1' + y_2t_2'$, $a = \zeta_T(t,x)$ and $b = \zeta_T(t',y)$. If $V_{t,x} \cap V_{t,y} = \emptyset$, then by Definition \ref{defo} we have
\[
V_{x_1,x_2} \cap V_{y_1,y_2} \neq \emptyset.
\]
Therefore, by Example \ref{coverex}, at least one of $(x_1,x_2)$ and $(y_1,y_2)$ lies outside the lattice $\Z^2 \subset \R^2$. By Example \ref{loopex}, this means that either $L^2X^{t',y} = X^b = \{N,S\}$ or $L^2X^{t,x} = X^a = \{N,S\}$, and we clearly have either $X^b \subset X^a$ or $X^a \subset X^b$. If we assume that $X^b \subset X^a$, then since at least one of the spaces is equal to $\{N,S\}$, we must have $X^b = \{N,S\}$. Therefore, $L^2X^{t',y} \subset L^2X^{t,x}$. Note that if we had the additional hypothesis that $(t',y) \in V_{t,x}$, then this would imply that $(y_1,y_2) \in V_{x_1,x_2}$, and by the description in Example \ref{coverex} we would have $(y_1,y_2) \notin \Z^2$, which means that $X^b \subset X^a$. So with this hypothesis, no assumption would be necessary.
\end{example}

\begin{proposition}\label{extend}
Let $Y \subset L^2X$ be a finite $\bT^2 \times T$-CW subcomplex. The inclusion $Y^{t,x} \subset Y$ induces an isomorphism of $\O_{V_{t,x}}$-algebras
\[
\H^*_{\bT^2\times T}(Y)_{V_{t,x}} \cong \H^*_{\bT^2\times T}(Y^{t,x})_{V_{t,x}}.
\]
\end{proposition}

\begin{proof}
Let $(t',y) \in V_{t,x}$. Then $L^2X^{t',y} \subset L^2X^{t,x}$ by Lemma \ref{hay}, which implies that $Y^{t',y} \subset Y^{t,x}$. Consider the commutative diagram

\begin{equation}
\begin{tikzcd}
\H^*_{\bT^2\times T}(Y)_{V_{t,x}} \ar[dr] \ar[rr] && \H^*_{\bT^2\times T}(Y^{t,x})_{V_{t,x}} \ar[dl] \\
&\H^*_{\bT^2\times T}(Y^{t',y})_{V_{t,x}}, &
\end{tikzcd}
\end{equation}

which is induced by the evident inclusions. Taking stalks at $(t',y)$, Theorem \ref{localisationn} implies that the two diagonal maps are isomorphisms, and so the horizontal map is also an isomorphism. The isomorphism of the proposition follows. 
\end{proof}

\begin{remark} 
The subspace $L^2X^{t,x} \subset L^2X$ is a $\bT^2 \times T$-equivariant CW subcomplex, consisting of those equivariant cells in $L^2X$ whose isotropy group contains $T(t,x)$. In fact, it follows easily from Lemma \ref{equal} that $L^2X^{t,x}$ is a finite $\bT^2 \times T$-CW complex, since $X$ is finite. 
\end{remark}

\begin{corollary}\label{stalkss}
The inclusion $L^2X^{t,x} \subset L^2X$ induces an isomorphism of $\O_{V_{t,x}}$-algebras
\[
\H^*_{L^2T}(L^2X)_{V_{t,x}} \cong \H^*_{\bT^2 \times T}(L^2X^{t,x})_{V_{t,x}},
\]
natural in $X$. In particular, we have an isomorphism of stalks
\[
\H^*_{L^2T}(L^2X)_{(t,x)} \cong \H^*_{\bT^2 \times T}(L^2X^{t,x})_{(t,x)}.
\]
\end{corollary}

\begin{proof}
It follows from Definition \ref{varlimm} and Proposition \ref{extend} that 
\[
\begin{array}{rcl}
\H^*_{L^2T}(L^2X)_{V_{t,x}} &=& \varprojlim_{Y \subset L^2X} \H^*_{\bT^2 \times T}(Y)_{V_{t,x}} \\
&\cong& \varprojlim_{Y \subset L^2X} \H^*_{\bT^2 \times T}(Y^{t,x})_{V_{t,x}} \\
&=& \H^*_{\bT^2 \times T}(L^2X^{t,x})_{V_{t,x}}.
\end{array}
\]
The final equality holds by definition of the inverse limit, since each $Y^{t,x}$ is contained in the finite $\bT^2 \times T$-CW subcomplex $L^2X^{t,x} \subset L^2X$. One can easily show naturality with respect to a $T$-equivariant map $f: X \to Y$ by refining the cover $\U$ so that it is adapted to $\S(f)$, and by using the functoriality of the loop space functor and of Borel-equivariant ordinary cohomology. The second statement follows immediately by definition of the stalk.
\end{proof}

\begin{remark}\label{coherent} 
It follows from Corollary \ref{stalkss} that $\H^*_{L^2T}(L^2X)$ is a coherent sheaf, since $L^2X^{t,x}$ is a finite $\bT^2 \times T$-CW complex.
\end{remark}

\section{Construction of the equivariant sheaf $\Ell^*_T(X)$ over $E_T$}
\sectionmark{Construction of the sheaf $\Ell_T(X)$ over $E_T$}

In this section, we begin with a result which says that $\H_{L^2T}(L^2X)$ depends only on loops which are contained in a subspace of the form $L^2X^{t,x}$, even if we are computing global sections. This is an improvement upon the localisation theorems of the previous chapter, which allowed us to restrict to loops in $L^2X^{t,x}$ after localising around the open set $V_{t,x}$, because we can now perform computations while discarding both the extra loops and the open cover. We will see an example of this when we compute $\Ell^*_T(X)$ on an orbit $X = T/H$. \par

We then recall the $\G$-action on $L^2X$ and explain how this induces an action of the Weyl group $\SL_2(\Z) \ltimes \check{T}^2$ on $\H_{L^2T}(L^2X)$, which also carries a $\C^\times$-action. Finally, we define the $\C^\times \times \SL_2(\Z)$-equivariant sheaf $\Ell^*_T(X)$ over $E_T$ as the $\check{T}^2$-invariants of the pushforward of $\H_{L^2T}(L^2X)$ along $\zeta_T$.

\begin{definition}
Let $\mathcal{D}^2(X)$ denote the set of finite $\bT^2 \times T$-CW complexes which is generated by the subspaces
\[
\{L^2X^{t,x}\}_{(t,x) \in \X^+ \times \t_\C}
\]
of $L^2X$ under finite unions and intersections. The set $\mathcal{D}^2(X)$ is partially ordered by inclusion.
\end{definition}

\begin{theorem}\label{yess}
There is an isomorphism of $\O_{\X^+ \times \t_\C}$-algebras
\[
\H^*_{L^2T}(L^2X) \cong \varprojlim_{Y \in \mathcal{D}^2(X)} \H^*_{\bT^2 \times T}(Y)_{\X^+ \times \t_\C}
\]
natural in $X$.
\end{theorem}

\begin{proof}
Consider the union
\[
S := \bigcup_{(t,x) \in \X^+ \times \t_\C} L^2X^{t,x}.
\]
Notice that $S$ is a $\bT^2 \times T$-CW subcomplex of $L^2X$. For each finite $\bT^2 \times T$-CW subcomplex $Y \subset L^2X$, we have an inclusion $Y \cap S \hookrightarrow Y$. The induced map of $\O_{\X^+ \times \t_\C}$-algebras
\begin{equation}\label{maps1}
\varprojlim_{Y \subset L^2X} \H^*_{\bT^2 \times T}(Y)_{\X^+ \times \t_\C} \to \varprojlim_{Y \subset L^2X} \H^*_{\bT^2 \times T}(Y \cap S)_{\X^+ \times \t_\C}
\end{equation}
is natural in $X$, by the functoriality of Borel-equivariant ordinary cohomology. Let $(t,x)$ be an arbitrary point. By Corollary \ref{stalkss}, the map \eqref{maps1} induces an isomorphism of stalks at $(t,x)$ because $S$ contains $L^2X^{t,x}$. Therefore, the map \eqref{maps1} is an isomorphism of sheaves.

It is clear that the set $\{S \cap Y \, | \, Y \subset L^2X \, \text{finite}\}$ is equal to the set of all finite equivariant subcomplexes of $S$. Therefore, the target of map \eqref{maps1} is equal to the inverse limit
\begin{equation}\label{mapsss1}
\varprojlim_{Y \subset S} \H^*_{\bT^2 \times T}(Y)_{\X^+ \times \t_\C}
\end{equation}
over all finite equivariant subcomplexes $Y \subset S$. Any such $Y$, since it is finite, is contained in the union of finitely many spaces of the form $L^2X^{t,x}$, which is also a finite $\bT^2 \times T$-CW complex. Therefore, by definition of the inverse limit, \eqref{mapsss1} is equal to 
\[
\varprojlim_{Y \in \mathcal{D}^2(X)} \H^*_{\bT^2 \times T}(Y)_{\X^+ \times \t_\C},
\]
which completes the proof.
\end{proof}

\begin{remark}
Theorem \ref{yess} allows us to define $\H^*_{L^2T}(L^2X)$ without a $\bT^2 \times T$-CW complex structure for $L^2X$.
\end{remark}

\begin{remark}\label{earl}
Let $X$ be a finite $T$-CW complex. Recall that the extended double loop group $\G$ acts on $L^2X$ via
\[
(A,t,\gamma(s)) \cdot \gamma'(s) = \gamma(A^{-1}s-At)\cdot \gamma'(A^{-1}s-At).
\]
By applying the Milnor construction to the topological group $\G$, we obtain a contractible space $E\G$ equipped with a free right action of $\G$. Define a left action of $\G$ on $E\G \times L^2X$ by setting 
\[
g\cdot (e,\gamma) = (e\cdot g^{-1},g\cdot \gamma)
\]
for $g \in \G$. The topological quotient of $E\G \times L^2X$ by $\bT^2 \times T \subset \G$ is a model for the Borel construction
\[
E(\bT^2 \times T) \times_{\bT^2 \times T} L^2X,
\]
since $\bT^2 \times T$ acts freely on $E\G$ via the action of $\G$. Recall that $\SL_2(\Z) \ltimes \check{T}^2$ is the Weyl group associated to $\bT^2 \times T \subset \G$. The action of $\G$ on $E\G \times L^2X$ induces an action
\[
\SL_2(\Z) \ltimes \check{T}^2  \acts E(\bT^2 \times T) \times_{\bT^2 \times T} L^2X.
\]
The action of $(A,m) \in \SL_2(\Z) \ltimes \check{T}^2$ induces a homeomorphism  
\[
E(\bT^2 \times T) \times_{\bT^2 \times T} L^2X^{t,x} \longrightarrow E(\bT^2 \times T) \times_{\bT^2 \times T} L^2X^{At,x+mt},
\]
for each $(t,x) \in \X^+ \times \t_\C$. Writing $x = x_1t_1 + x_2t_2$, the homeomorphism sends
\begin{equation}\label{earl1}
(e,\pi(x_1s_1 + x_2s_2)\cdot z) \mapsto (e\cdot (A,m)^{-1},\pi((x_1+m_1,x_2+m_2)A^{-1}(s_1,s_2))\cdot z).
\end{equation}
Note that  
\[
\pi((x_1+m_1,x_2+m_2)A^{-1}(s_1,s_2))\cdot z 
\]
does in fact lie in $L^2X^{At,x+mt}$, since
\[
x + mt = ((x_1+m_1)t_1, (x_2+m_2)t_2) = (x_1+m_1,x_2+m_2)A^{-1} A(t_1,t_2).
\]
There is an induced isomorphism on cohomology rings
\[
(A,m)^*: H^*_{\bT^2 \times T}(L^2X^{At,x+mt}) \longrightarrow H^*_{\bT^2 \times T}(L^2X^{t,x}).
\]
Let $\omega_{A,m}: \X^+ \times \t_\C \rightarrow \X^+ \times \t_\C$ be the action map of $(A,m)$. Then $(A,m)^*$ induces an isomorphism of sheaves
\[
\omega_{A,m}^* (\H^*_{\bT^2 \times T}(L^2X^{At,x+mt})_{\X^+ \times \t_\C}) \longrightarrow \H^*_{\bT^2\times T}(L^2X^{t,x})_{\X^+ \times \t_\C} 
\]
In fact, we can get such a map for each $Y \in \bD^2(X)$ by replacing $L^2X^{q,u}$ with $Y$ and $L^2X^{At,x+mt}$ with $(A,m)\cdot Y$. It is easily verified that $(A,m)\cdot Y \in \bD^2(X)$ whenever $Y \in \bD^2(X)$. As they are induced by the action of $\G$ on $L^2X$, the family of isomorphisms
\[
\{ \omega_{A,m}^* (\H^*_{\bT^2 \times T}((A,m)\cdot Y)_{\X^+ \times \t_\C}) \longrightarrow \H^*_{\bT^2\times T}(Y)_{\X^+ \times \t_\C} \}_{Y \in \bD^2(X)}
\]
is compatible with inclusions in $\bD^2(X)$. We therefore have an induced isomorphism
\[
\omega_{A,m}^*\H^*_{L^2T}(L^2X) \longrightarrow \H^*_{L^2T}(L^2X)
\]
of inverse limit sheaves, which induces an isomorphism 
\begin{equation}
(A,m)^*: \H^*_{\bT^2 \times T}(L^2X^{At,x+mt})_{(At,x+mt)} \longrightarrow \H^*_{\bT^2 \times T}(L^2X^{t,x})_{(t,x)} \\
\end{equation}
of stalks, also denoted $(A,m)^*$. Since it is induced by a group action, the collection of isomorphisms
\[
\{\omega_{A,m}^*\H^*_{L^2T}(L^2X) \longrightarrow \H^*_{L^2T}(L^2X)\}_{(A,m) \in \SL_2(\Z) \ltimes \check{T}^2}
\] 
defines an action of $\SL_2(\Z) \ltimes \check{T}^2$ on $\H^*_{L^2T}(L^2X)$. 
\end{remark}

\begin{remark}\label{jeeves}
We describe a $\C^\times$-action on $\H^*_{L^2T}(L^2X)$ which clearly commutes with the action of $\SL_2(\Z) \ltimes \check{T}^2$. Let $\C^\times$ act on $\X^+ \times \t_\C$ by\footnote{The action on the left factor $\X^+$ should be compared to the scalar action in Remark \ref{dominic}. We are forced to introduce the square because the holomorphic coordinate functions have degree 2 when regarded as classes in equivariant cohomology.}
\[
\lambda \cdot (t,x) = (\lambda^2 t,\lambda^2 x).
\]
The action on the sheaf is given for each $\lambda \in \C^\times$ by isomorphisms
\[
\begin{array}{rcl}
\omega_{\lambda}^* (\H^*_{\bT^2 \times T}(L^2X^{\lambda^2t,\lambda^2x})_{\X^+ \times \t_\C}) &\longrightarrow& \H^*_{\bT^2\times T}(L^2X^{t,x})_{\X^+ \times \t_\C}  \\
\alpha \otimes f &\longmapsto & \lambda^n \alpha \otimes f
\end{array}
\]
where $\alpha$ is a degree $n$ cohomology class and $\omega_{\lambda}$ denotes the action of $\lambda$ on $\X^+ \times \t_\C$. On an open subset $U \subset \X^+ \times \t_\C$, the isomorphisms are 
\[
\begin{array}{rcl}
H^*_{\bT^2 \times T}(L^2X^{\lambda^2t,\lambda^2x})_{\lambda^2 U} &\longrightarrow& H^*_{\bT^2\times T}(L^2X^{t,x})_{U}  \\
\alpha \otimes f &\longmapsto & \lambda \cdot \alpha \otimes f \circ \omega_\lambda.
\end{array}
\]
For example, if $\alpha$ is a class of degree $i$ and $f$ satisfies $f(\lambda^2t,\lambda^2x) = \lambda^j f(t,x)$, then $\alpha \otimes f$ is sent to $\lambda^{i+j} (\alpha \otimes f)$.
\end{remark}

\begin{theorem}\label{actyon1}
The sheaf $\H^*_{L^2T}(L^2X)$ is a $\C^\times \times (\SL_2(\Z) \ltimes \check{T}^2)$-equivariant sheaf of $\O_{\X^+ \times \t_\C}$-algebras.
\end{theorem}

\begin{proof}
This holds by Remarks \ref{earl} and \ref{jeeves}.
\end{proof}

We denote by $\iota_{t}$ the inclusion of the fiber $\t_\C \hookrightarrow \X^+ \times \t_\C$ over $t \in \X^+$. We have a commutative diagram of complex manifolds
\[
\begin{tikzcd}
\t_\C \ar[r,hook,"{\iota_t}"] \ar[d,two heads, "{\zeta_{T,t}}"] & \X^+ \times \t_\C \ar[d,two heads, "{\zeta_T}"] \\
E_{T,t} \ar[r,hook] & E_T. 
\end{tikzcd}
\]
The $\C^\times \times \SL_2(\Z)$-action on $E_T$ does not preserve the fiber over $t$. With Remark \ref{coherent} in mind, we make the following definition.

\begin{definition}
Define the $\C^\times \times \SL_2(\Z)$-equivariant, coherent sheaf of $\O_{E_T}$-algebras 
\[
\Ell^*_T(X) := ((\zeta_{T})_*\, \H^*_{L^2T}(L^2X))^{\check{T}^2},
\]
which has a $\Z$-grading corresponding to the action of $\C^\times$. Define the coherent sheaf of $\O_{E_{T,t}}$-algebras 
\[
\Ell^*_{T,t}(X) :=  ((\zeta_{T,t})_*\, \iota_t^*\, \H^*_{L^2T}(L^2X))^{\check{T}^2},
\]
which has a $\Z/2\Z$-grading given by even and odd cohomology, as in Remark \ref{eqii}.
\end{definition}

\begin{remark}
By construction, $\Ell_{T}^*(\pt)$ is equal to the $\C^\times \times \SL_2(\Z)$-equivariant structure sheaf
\[
\O_{E_T} = ((\zeta_T)_*\O_{\X^+ \times \t_\C})^{\check{T}^2} 
\]
of $E_T$. Similarly, $\Ell_{T,t}^*(\pt)$ is equal to the structure sheaf
\[
\O_{E_{T,t}} = ((\zeta_{T,t})_* \O_{\t_\C})^{\check{T}^2}
\]
of the fiber of $E_T$ over $t \in \X^+$.
\end{remark}

\begin{example}
It is straightforward to compute $\Ell_T^*(X)$ when $T = 1$, for this implies that $\bD^2(X) = \{X\}$. We obtain the $\C^\times \times \SL_2(\Z)$-equivariant sheaf whose value on an open subset $U \subset \X^+$ is 
\[
\Ell_1^*(X)(U) = H^*(X) \otimes_\C \O_{\X^+}(U).
\]
The value of $\Ell_1^n(X)$ on $U$ is the vector space generated by elements $\alpha \otimes f$ where $\alpha \in H^i(X)$ and $f$ is a holomorphic on $U$ such that $f(\lambda^2 t) = \lambda^j f(t)$, for all $(i,j) \in \Z^2$ satisfying $i + j = n$.
\end{example}

\begin{example}
In the case that $T = 1$ and $X = \pt$, we obtain the $\C^\times \times \SL_2(\Z)$-equivariant holomorphic structure sheaf of $\X^+$. The value of $\Ell_1^n(\pt)$ on $U$ is the vector space generated by holomorphic functions $f: U \to \C$ satisfying $f(\lambda^2 t) = \lambda^n f(t)$, whenever this is defined. The $\SL_2(\Z)$-invariant global sections of $\Ell_1^n(\pt)$ are exactly the weak modular forms of weight $-n/2$. In particular, the $\SL_2(\Z)$-invariant global sections of $\Ell_1^*(\pt)$ are concentrated in negative even degrees.
\end{example}

\section{The suspension isomorphism}

Our aim now is to show that the reduced theory $\widetilde{\Ell}^*_T(X)$ has a suspension isomorphism which is inherited from ordinary cohomology. We first prove a lemma to show that the suspension functor commutes with the double loop space functor, but only once we have taken fixed points with respect to some $T(t,x)$. Thus, we will implicitly be using the localisation theorem to establish the suspension isomorphism.

\begin{remark}
Let $X$ be a finite $T$-CW complex. We regard the loop space $L^2(S^1 \wedge X_+)$ as a pointed $\bT^2 \times T$-CW complex with basepoint given by the loop $\gamma_*: \bT^2 \to \pt \hookrightarrow S^1 \wedge X_+$. Since the basepoint of $S^1 \wedge X_+$ is fixed by $T$, the loop $\gamma_*$ is fixed by $\bT^2 \times T$, and so $\gamma_*$ is contained in $L^2(S^1 \wedge X_+)^{t,x}$ for all $(t,x)$. Therefore, each $Y \in \mathcal{D}^2(S^1 \wedge X_+)$ is a based subcomplex of $L^2(S^1 \wedge X_+)$.
\end{remark}

\begin{lemma}\label{wedger}
For each $(t,x) \in \X^+ \times \t_\C$, we have an equality
\[
L^2(S^1\wedge X_+)^{t,x} = S^1 \wedge L^2X^{t,x}_+ 
\]
as subsets of $L^2(S^1 \wedge X_+)$.
\end{lemma}

\begin{proof}
Suppose that $\gamma$ is a loop in $L^2(S^1\times X_+)^{t,x}$ sending $s \mapsto (\gamma_1(s),\gamma_2(s))$. Write $\gamma(0) = (z_1,z_2)$ and $x = x_1t_1 + x_2t_2$. The loop $\gamma$ is fixed by $T(t,x)$ if and only if 
\[
(\gamma_1(s),\gamma_2(s)) = \pi(x_1r_1+x_2r_2)\cdot (\gamma_1(s-r),\gamma_2(s-r)) =  (\gamma_1(s-r),\pi(x_1r_1+x_2r_2)\cdot \gamma_2(s-r))
\]
for all $r,s \in \bT^2$, where $S^1$ is fixed by $T$. Setting $r = s$, one sees that this is true if and only if
\[
(\gamma_1(r),\gamma_2(r)) =  (\gamma_1(0),\pi(x_1r_1+x_2r_2)\cdot \gamma_2(0)), 
\]
which holds if and only $\gamma_1$ is constant in $S^1$ and $\gamma_2$ is in $L^2X^{t,x}_+$. Therefore, we have an equality
\[
L^2(S^1 \times X_+)^{t,x} = S^1 \times L^2X^{t,x}_+.
\]
To prove the equality of the lemma, consider that the image of $\gamma$ is contained in $0 \times X_+ \cup S^1 \times \pt$ if and only if we have either $\gamma_1(s) = 0$ or $\gamma_2(s) = \pt$ for each $s$. But this means that either $\gamma_1(s) = 0$ for all $s$, or $\gamma_2(s) = \pt$ for all $s$, which means that $\gamma \in 0 \times L^2X^{t,x} \cup S^1 \times \pt$. The equality of the lemma follows.
\end{proof}

\begin{definition}
For a pointed finite $T$-CW complex $X$, define the reduced theory
\[
\widetilde{\Ell}_T^*(X) := \ker(\Ell_T^*(X) \to \Ell_T^*(\pt)).
\]
\end{definition}

\begin{proposition}\label{jezza}
Let $X$ be a finite $T$-CW complex. There is an $\SL_2(\Z)$-equivariant isomorphism of sheaves of $\O_{E_T}$-modules
\[
\Ell^{*-1}_T(X) \cong \widetilde{\Ell}^*_T(S^1 \wedge X_+)
\]
natural in $X$.
\end{proposition}

\begin{proof}
We have
\[
\begin{array}{rcl}
\Ell^{*-1}_T(X) &:=& ((\zeta_T)_* \varprojlim_{Y \in \bD^2(X)} \H^{*-1}_{\bT^2 \times T}(Y)_{\X^+ \times \t_\C})^{\check{T}^2}\\
&\cong& ((\zeta_T)_* \varprojlim_{Y \in \bD^2(X)} \widetilde{\H}^*_{\bT^2 \times T}(S^1 \wedge Y_+)_{\X^+ \times \t_\C})^{\check{T}^2 }\\
&=& ((\zeta_T)_* \varprojlim_{Y \in \bD^2(S^1 \wedge X_+)} \widetilde{\H}^*_{\bT^2 \times T}(Y)_{\X^+ \times \t_\C})^{\check{T}^2} \\
&=:& ((\zeta_T)_* \varprojlim_{Y \in \bD^2(S^1 \wedge X_+)} \ker(\H^*_{\bT^2 \times T}(Y) \to \H^*_{\bT^2 \times T}(\pt))_{\X^+ \times \t_\C})^{\check{T}^2} \\
&=&  \ker(\Ell^*_T(S^1\wedge X_+) \to \Ell^*_T(\pt)) \\
&=:&  \widetilde{\Ell}^*_T(S^1 \wedge X_+).
\end{array}
\]
Indeed, the second line is induced by the natural suspension isomorphism of ordinary cohomology. The third holds because the equality $L^2(S^1\wedge X_+)^{t,x} = S^1 \wedge L^2X^{t,x}_+$ of Lemma \ref{wedger} is an equality of subsets of $L^2(S^1 \wedge X_+)$ for each $(t,x)$ which are compatible with inclusion, so that they extend to an equality of partially ordered sets
\[
\bD^2(S^1 \wedge X_+) = \{S^1 \wedge Y_+\}_{Y\in \mathcal{D}^2(X)}.
\]
The fifth holds since the inverse limit is a right adjoint functor, and therefore respects all limits, including kernels. The remaining equalities hold by definition. The isomorphism of the second line is $\SL_2(Z)$-equivariant by the functoriality of Borel-equivariant cohomology, and each equality is evidently $\SL_2(\Z)$-equivariant. Therefore, the composite is $\SL_2(\Z)$-equivariant.
\end{proof}

\begin{remark}\label{hewson}
Once it is equipped with the suspension isomorphisms of Proposition \ref{jezza}, the map sending $X \mapsto \widetilde{\Ell}_T^*(X)$ defines a reduced cohomology theory on finite $T$-CW complexes. It is clearly functorial and homotopy invariant, since both the double loop space functor and Borel-equivariant ordinary cohomology are functorial and homotopy invariant. It is also exact and additive, since these properties may be checked on stalks, and Corollary \ref{stalkss} implies that these properties are inherited from Borel-equivariant ordinary cohomology (since $L^2X^{t,x} \cong X^a$ by Lemma \ref{equal}; compare the proof of Proposition \ref{kerry}).
\end{remark}

\section{The value of $\Ell_T^*(X)$ on an orbit}

The description of $\H^*_{L^2T}(L^2X)$ given in Theorem \ref{yess} allows us to compute $\Ell^*_T(X)$ on an orbit $X = T/H$, where the $T$ action is induced by group multiplication. Having made this computation, one can then compute $\Ell^*_T(X)$ on any other finite $T$-CW complex $X$ using the Mayer-Vietoris sequence. \par

We will have shown by the end of this section that
\[
\Ell^*_T(T/H) = \O_{E_H}.
\]
where $E_H$ is the image of $\X^+ \times H \times H$ under the isomorphism $\chi_T: \X^+ \times T \times T \longrightarrow E_T$. \par

We begin by calculating $\bD^2(T/H)$. Let $a \in E_{T,t}$ and write $(a_1,a_2)$ for the preimage of $a$ under $\chi_{T,t}$. We have
\[
(T/H)^a = (T/H)^{\langle a_1,a_2 \rangle} = \begin{cases} T/H &\mbox{if } (a_1,a_2) \in H \times H \\ \emptyset & \mbox{otherwise.} \end{cases}.
\]
Let $x \in \zeta_{T,t}^{-1}(a)$ and write $x = x_1t_1 + x_2t_2$. We have that
\begin{equation}\label{henry}
L^2X^{t,x} = \{\gamma(s_1,s_2) = \pi(x_1s_1 + x_2s_2) \cdot z \, | \, z\in T/H\}
\end{equation}
if $(x_1,x_2) \in \pi^{-1}(H \times H)$, and $L^2X^{t,x}$ is empty otherwise. To calculate $\bD^2(T/H)$, we need to calculate the intersections of these subspaces. Suppose 
\[
\gamma(s_1,s_2) \in L^2X^{t,x}  \cap L^2X^{t',x'} 
\]
with $x = x_1t_1 + x_2t_2$ and $x' = x_1't_1'+ x_2't_2'$. Then
\[
\gamma(s_1,s_2) = \pi(x_1s_1+x_2s_2)\cdot z = \pi(x_1's_1+x_2's_2)\cdot z'
\]
for $z,z' \in T/H$. A straightforward calculation shows that this holds if and only if
\[
z - z' = \pi((x_1-x_1')s_1 + (x_2-x_2')s_2) \text{ mod } H
\]
for all $s_1,s_2$, which holds if and only if
\[
z = z' \quad \text{and} \quad (x_1-x_1',x_2-x_2') \in \mathrm{Lie}(H) \times \mathrm{Lie}(H).
\]
Thus, for all $(t,x),(t',x') \in \X^+ \times \t_\C$, the intersection
\[
L^2X^{t,x} \cap L^2X^{t',x'} = L^2X^{t,x} = L^2X^{t',x'}, 
\]
if and only if $(x_1,x_2),(x_1',x_2')$ lie in the same component of $\pi^{-1}(H \times H)$, and is empty otherwise. In particular, the spaces $L^2X^{t,x}$ are indexed by the components of $\pi^{-1}(H \times H)$, where the component containing $(x_1,x_2)$ has the form
\[
\{(x_1,x_2)\} + \mathrm{Lie}(H) \times \mathrm{Lie}(H) \subset \t \times \t.
\]
We will call $(x_1,x_2)$ a \textit{representative} of this component. Moreover, we have seen that there are no nonempty intersections between two spaces of the form $L^2X^{t,x}$, unless they are equal. Therefore $\bD^2(X)$ is the set of finite disjoint unions of such spaces. \par

The next step is to calculate the cohomology ring $H_{\bT^2 \times T}(L^2X^{t,x})$, but first we need to know a bit more about the space $L^2X^{t,x}$. Since the $\bT^2 \times T$-action on $L^2X^{t,x}$ is clearly transitive, we can apply the change of groups property of Proposition \ref{change1} if we know the $\bT^2 \times T$-isotropy. An element $(r,u) \in \bT^2 \times T$ fixes a nonempty subset $L^2X^{t,x}$ if and only if
\[
(r,u) \cdot (\pi(x_1s_1 + x_2s_2)\cdot z) = (\pi(x_1(s_1-r_1) + x_2(s_2-r_2)) + u ) \cdot z
= \pi(x_1s_1 + x_2s_2)\cdot z,
\]
which holds if and only if $\pi(-x_1r_1 - x_2r_2) + u$ fixes $z$. Therefore, we must have
\[
u - \pi(x_1r_1 + x_2r_2) \in H,
\]
which means that the isotropy group of the $\bT^2 \times T$-orbit $L^2X^{t,x}$ is equal to 
\begin{equation}\label{ernie}
\{(r,u) \in \bT^2 \times T \, | \, u - \pi(x_1r_1 + x_2r_2) \in H \} = \langle T(t,x),H \rangle.
\end{equation}
Furthermore, since two spaces of the form $L^2X^{t,x}$ are equal if they correspond to the same component of $\pi^{-1}(H \times H)$, we must have that 
\[
\langle T(t,x),H \rangle = \langle T(t',x'),H \rangle
\]
whenever $(x_1,x_2),(x_1',x_2')$ lie in the same component $\pi^{-1}(H \times H)$. Using Proposition \ref{change1}, we calculate 
\[
H_{\bT^2 \times T}(L^2X^{t,x}) \cong H_{\bT^2 \times T}((\bT^2 \times T)/\langle T(t,x),H \rangle) \cong H_{\langle T(t,x),H \rangle}.
\]
The value of $\H^*_{L^2T}(L^2X)$ on an open subset $U \subset \X^+ \times \t_\C$ is therefore
\begin{gather*}
\varprojlim_{Y \in \mathcal{D}^2(X)} H^*_{\bT^2 \times T}(Y) \otimes_{H_{\bT^2 \times T}} \O_{\X^+ \times \t_\C}(U) \\ = \prod_{(x_1,x_2) \in J(H)} H_{\langle T(x_1,x_2),H \rangle}  \otimes_{H_{\bT^2 \times T}} \O_{\X^+ \times \t_\C}(U)
\end{gather*}
where the product is indexed over a set $J(H) = \{(x_1,x_2)\}$ of representatives of the components of $\pi^{-1}(H \times H)$, and $T(x_1,x_2) := T(t,x)$ for any $(t,x)$ such that $x = x_1t_1+x_2t_2$. \par

From the description in \eqref{ernie}, we see that
\begin{equation}\label{berrol}
\mathrm{Lie}(\langle T(x_1,x_2),H \rangle)_\C = \{ (t,y) \in \X^+ \times \t_\C \, | \, y \in (x_1 + \mathrm{Lie}(H))t_1 + (x_2 + \mathrm{Lie}(H))t_2 \}.
\end{equation}
Let $I(x_1,x_2,H) \subset \C[t_1,t_2,y]$ be the ideal associated to $\mathrm{Lie}(\langle T(x_1,x_2),H \rangle)_\C \subset \C^2 \times \t_\C$, so that
\[
H_{\langle T(x_1,x_2),H \rangle} = \C[t,y]/I(x_1,x_2,H).
\]
By Proposition \ref{flat}, tensoring over $H_{\bT^2 \times T}$ with the ring of holomorphic functions is an exact functor. Therefore, if $\I(x_1,x_2,H) \subset \O_{\X^+ \times \t_\C}$ denotes the analytic ideal associated to $I(x_1,x_2,H)$, we can write
\begin{equation}\label{beer}
\H^*_{L^2T}(L^2(T/H)) = \prod_{(x_1,x_2)\in J(H)}  \O_{\X^+ \times \t_\C}/\I(x_1,x_2,H).
\end{equation}
For example, if $T$ has rank one and $H = \Z/n\Z$, then $\mathrm{Lie}(H)$ is trivial and we have
\[
\H^*_{L^2T}(L^2(T/H))(U) = \prod_{(x_1,x_2)} \C[t_1,t_2,y]/(y - x_1t_1 - x_2t_2) \otimes_{\C[t,y]} \O_{\X^+ \times \t_\C}(U)
\]
where $(x_1,x_2)$ ranges over $J(\Z/n\Z) = \{(a_1/n, a_2/n) \, | \, a_1,a_2 \in \Z\}$. This is a holomorphic version of the calculation made by Rezk in Example 5.2 of \cite{Rezk}.\par

The $\check{T}^2$-action induces isomorphisms 
\[
\O(\X^+ \times \t_\C)/\I(x_1+m_1,x_2+m_2,H) \longrightarrow \O(\X^+ \times \t_\C)/\I(x_1,x_2,H),
\]
for each $m \in \check{T}^2$ and each $(x_1,x_2) \in J(H)$, given by pullback along $(t,y) \mapsto (t,y+mt)$. We calculate the invariants by first noting that
\[
\pi^{-1}(H \times H) = \coprod_{(x_1,x_2)\in J(H)} \{(x_1,x_2)\} + \mathrm{Lie}(H) \times \mathrm{Lie}(H).
\]
Thus, the diagram
\[
\begin{tikzcd}
\X^+ \times \pi^{-1}(H \times H) \ar[d] \ar[r,"{\cong}"] & \coprod_{(x_1,x_2)\in J(H)} \mathrm{Lie}(\langle T(x_1,x_2),H \rangle)_\C \ar[d] \\
\X^+ \times H \times H \ar[r,"{\cong}"]& E_H
\end{tikzcd}
\]
commutes, where the arrows are the evident restrictions of the maps appearing in the diagram \eqref{deer}, and $E_H \subset E_T$ is defined to be the image of $\X^+ \times H \times H$ under $\chi_T$. Since the right vertical map is exactly the preimage under $\zeta_T$ of $E_H$, we see that $E_H$ is the complex analytic quotient of the free $\check{T}^2$ action on the complex vector space \eqref{berrol}, which has holomorphic structure sheaf \eqref{beer}. Therefore, the $\check{T}^2$-invariants of \eqref{beer} is equal to the sheaf of holomorphic functions on $E_H$, and so we have 
\[
\Ell^*_T(T/H) = \O_{E_H}.
\]

\section{A local description}
	%\sectionmark{A local description}}
%\sectionmark{A local description}

In this section, given a finite $T$-CW complex $X$, we give a local description of $\Ell_{T,t}(X)$ over a cover of $E_{T,t}$ adapted to $X$, which turns out to be identical to Grojnowski's construction of $T$-equivariant elliptic cohomology corresponding to the elliptic curve $E_t$. We show that this is actually an isomorphism of $T$-equivariant cohomology theories on finite complexes. \par

\begin{notation}
Let $K(t,x) := (\bT^2 \times T)/T(t,x)$. Let $t_{t,x}$ denote translation by $(t,x)$ in $\C^2 \times \t_\C$, and $t_x$ translation by $x$ in $\t_\C$. If $f$ is a map of compact Lie groups, we abuse notation and also write $f$ for the induced map of complex Lie algebras. 
\end{notation}

\begin{remark}
It will be essential to the proof of the following Lemma to show that the diagram
\begin{equation}\label{dg}
\begin{tikzcd}
\t_\C \ar[r,"p_{a} \circ t_{-x}",two heads] \ar[d,hook,"\iota_{t}"] & \mathrm{Lie}(T/T(a))_\C \ar[d,"\nu","\cong"'] \\
\C^2 \times \t_\C \ar[r,"p_{t,x}", two heads] & \mathrm{Lie}(K(t,x))_\C
\end{tikzcd}
\end{equation}
commutes. 
\end{remark}

\begin{lemma}\label{equal1}
There is an isomorphism of sheaves of $\O_{\t_\C}$-algebras
\[
t_{-x}^* \, p_{a}^*\, \H^*_{T/T(a)}(X^{a}) \cong \iota_{t}^* \, p_{t,x}^*\, \H^*_{K(t,x)}(L^2X^{t,x}) 
\]
natural in $X$.
\end{lemma}

\begin{proof}
Recall the isomorphism $\nu: K(t,x) \cong T/T(a)$ of diagram \eqref{rect}. By Lemma \ref{equal}, the evaluation map
\[
ev_{t,x}: L^2X^{t,x} \cong X^a
\]
is natural in $X$ and equivariant with respect to $\nu$. The induced homeomorphism
\[
E(K(t,x)) \times_{K(t,x)} L^2X^{t,x} \cong E(T/T(a)) \times_{T/T(a)} X^a
\]
induces, in turn, an isomorphism of $H_{T/T(a)}$-algebras
\[
H^*_{T/T(a)}(X^{a}) \cong H^*_{K(t,x)}(L^2X^{t,x}),
\]
natural in $X$. Here the ring $H_{T/T(a)}$ acts on the target via the isomorphism 
\[
H_{K(t,x)} \cong H_{T/T(a)}
\]
induced by $\nu$. Thus, we have an isomorphism of $\O_{\mathrm{Lie}(T/T(a))_\C}$-algebras   
\[
\H^*_{T/T(a)}(X^{a}) \cong \nu^* \, \H^*_{K(t,x)}(L^2X^{t,x}),
\]
and hence an isomorphism of $\O_{\t_\C}$-algebras,
\[
t_{-x}^*\, p_a^*\, \H^*_{T/T(a)}(X^{a}) \cong t^*_{-x}\, p_{a}^*\,\nu^* \, \H^*_{K(t,x)}(L^2X^{t,x})
\]
natural in $X$.

We now show that diagram \eqref{dg} commutes, from which the result follows immediately. The commutative diagram \eqref{rect} induces a commutative diagram of complex Lie algebras, so that
\[
\nu \circ p_a = p_{t,x} \circ \iota_{(0,0)}
\]
on complex Lie algebras. Consider the commutative diagram
\[
\begin{tikzcd}
T(t,x) \ar[r] \ar[d] & \bT^2 \times T \ar[d] \\
0 \ar[r] & (\bT^2 \times T)/T(t,x) 
\end{tikzcd}
\]
of compact abelian groups. By applying the Lie algebra functor and then tensoring with $\C$, we see that $(t,x)$ lies in the kernel of 
\[
p_{t,x}: \C^2 \times \t_\C \twoheadrightarrow \mathrm{Lie}((\bT^2 \times T)/T(t,x))_\C.
\]
Therefore,
\[
p_{t,x} = p_{t,x} \circ t_{t,x}
\]
and so
\[
\begin{array}{rcl}
\nu \circ p_a &=& p_{t,x} \circ t_{t,x} \circ \iota_{(0,0)} \\
&=& p_{t,x} \circ \iota_t \circ t_{x}. \\
\end{array}
\]
Composing on the right by $t_{-x}$ now yields the commutativity of diagram \eqref{dg}. 
\end{proof}

\begin{theorem}\label{iso}
Let $X$ be a finite $T$-CW complex and let $\U$ be a cover adapted to $\S(X)$. Let $a \in E_{T,t}$, and let $U_a$ be the corresponding open neighbourhood of $a$ in $E_{T,t}$. There is an isomorphism of sheaves of $\Z/2\Z$-graded $\O_{U_a}$-algebras
\[
\Ell^*_{T,t}(X)_{U_a} \cong \G^*_{T,t}(X)_{U_a},
\]
natural in $X$.
\end{theorem}

\begin{proof}
For $x \in \zeta_{T,t}^{-1}(a)$, write $V_x$ for the component of $\zeta_{T,t}^{-1}(U_a)$ containing $x$. We have a sequence of isomorphisms
\begin{equation}\label{compos}
\begin{array}{rcl}
((\zeta_{T,t})_* \,  \iota_{t}^* \,\H^*_{L^2T}(L^2X))_{U_a} &\cong& \prod_{x \, \in \, \zeta_{T,t}^{-1}(a)} \:  (\iota_{t}^* \, \H^*_{\bT^2\times T}(L^2X^{t,x}))_{V_{x}} \\
&\cong& \prod_{x \, \in \, \zeta_{T,t}^{-1}(a)} \: (\iota_{t}^* \, p_{t,x}^*\, \H^*_{K(t,x)}(L^2X^{t,x}))_{V_x} \\
&\cong& \prod_{x \, \in \, \zeta_{T,t}^{-1}(a)} \: (t_{-x}^*\, p_a^*\, \H^*_{T/T(a)}(X^{a}))_{V_x} \\
&\cong& \prod_{x \, \in \, \zeta_{T,t}^{-1}(a)} \: (t_{-x}^*\, \H^*_{T}(X^{a}))_{V_x} \\
\end{array}
\end{equation}
The first map is the isomorphism of Corollary \ref{extend}. The second map and fourth maps are the isomorphism of Proposition \ref{changeh}. The third map is the isomorphism of Lemma \ref{equal1}. The first to the fourth maps, in each factor, have been shown to preserve the $\O_{V_{x}}$-algebra structure. The entire composite is therefore an isomorphism of $\O_{U_a}$-algebras, where $f\in \O_{U_a}(U)$ acts on the right hand side via multiplication by $\zeta_{T,t}^*f$. Each map has been shown to be natural in $X$, replacing, if necessary, the covering $\U$ with a refinement $\U(f)$ adapted to a $T$-equivariant map $f:X \to Y$. All maps preserve the $\Z$-grading on the cohomology, except the map of Proposition \ref{changeh}. The latter does, however, preserve the $\Z/2\Z$-grading by odd and even elements, since it is defined by taking the cup product with elements of $H_T$, which have even degree.

It remains to find the image of the $\check{T}^2$-invariants under the composite map above. Let $U \subset U_a$ be an open subset and let $V \subset V_x$ be an open set such that $V \cong U$ via $\zeta_{T,t}$. Over $U$, the composite is equal to  
\[
\begin{array}{rcl}
\prod \: H_{L^2T}(L^2X)_{\{t\} \times V} &\cong& \prod \: H_{\bT^2\times T}(L^2X^{t,x})_{\{t\} \times V}  \\
&\cong& \prod \: H_{K(t,x)}(L^2X^{t,x})_{\{t\} \times V}  \\
&\cong& \prod \: H_{T/T(a)}(X^{a})_{V - x} \\
&\cong& \prod \: H_{T}(X^{a})_{V-x} \\
\end{array}
\]
where each product runs over all $x \in \zeta_{T,t}^{-1}(a)$. Note that 
\[
V_x - x = V_{x+mt} - x-mt
\]
for all $m$. We will show  
\[
\left\{ \prod_{x \, \in \, \zeta_{T,t}^{-1}(a)} H_T(X^a)_{V - x} \right\}^{\check{T}^2} =  H_T(X^a)_{V - x}
\]
by showing that $\check{T}^2$ merely permutes the indexing set of the product. This yields our result via the isomorphism 
\[
\G^*_{T,t}(X)_{U_a}(U) = H^*_T(X^a) \otimes_{H_T} \O_{E_{T,t}}(U - a) \cong H^*_T(X^a) \otimes_{H_T} \O_{\t_\C}(V - x) =: H_{T}(X^{a})_{V - x}
\]
induced by the canonical isomorphism $V - x \cong U - a$, as in Remark \ref{canon}.\par

In other words, we must show that the action of $m$ induces the identity map 
\[
H_{T/T(a)}(X^{a})_{V_{x+mt} - x-mt}  = H_{T/T(a)}(X^{a})_{V_x - x}.
\]
To do this, it suffices to check the commutativity of two diagrams. The first diagram is
\begin{equation}\label{pal}
\begin{tikzcd}
H_{\bT^2 \times T}(L^2X^{t,x+mt})_{\{t\} \times V_{x+mt}} \ar[r] \ar[d,"{m^*}"] & H_{K(t,x+mt)}(L^2X^{t,x+mt})_{\{t\} \times V_{x+mt}} \ar[d,"{m^*}"] \\
H_{\bT^2 \times T}(L^2X^{t,x})_{\{t\} \times V_x}  \ar[r] & H_{K(t,x)}(L^2X^{t,x})_{\{t\} \times V_x}  \\
\end{tikzcd}
\end{equation}
where the vertical arrows are induced by the action of $m$ on the spaces
\[
E(\bT^2 \times T) \times_{\bT^2 \times T} L^2X^{t,x} \longrightarrow E(\bT^2 \times T) \times_{\bT^2 \times T} L^2X^{t,x+mt}
\]
and
\[
E(K(t,x)) \times_{K(t,x)} L^2X^{t,x}  \longrightarrow  E(K(t,x+mt)) \times_{K(t,x+mt)} L^2X^{t,x+mt}.
\]
The horizontal maps are defined, as in Proposition \ref{changeh}, using the Eilenberg-Moore spectral sequences associated to the pullback diagrams
\begin{equation}\label{gary}
\begin{tikzcd}
E(\bT^2 \times T) \times_{\bT^2\times T} L^2X^{t,x} \ar[d] \ar[r] & B(\bT^2 \times T) \ar[d] \\
E(K(t,x)) \times_{K(t,x)} L^2X^{t,x} \ar[r] &  B(K(t,x)) \\
\end{tikzcd}
\end{equation}
and
\begin{equation}\label{steve}
\begin{tikzcd}
E(\bT^2 \times T) \times_{\bT^2\times T} L^2X^{t,x+mt} \ar[r] \ar[d] & B(\bT^2 \times T) \ar[d] \\ 
E(K(t,x+mt)) \times_{K(t,x+mt)} L^2X^{t,x+mt} \ar[r] & B(K(t,x+mt)). \\
\end{tikzcd}
\end{equation}
It is easily verified that the action of $m$ induces an isomorphism from diagram \eqref{gary} to diagram \eqref{steve}, from which it follows that \eqref{pal} commutes. \par

The second diagram to check is
\begin{equation}\label{pall}
\begin{tikzcd}
H_{K(t,x+mt)}(L^2X^{t,x+mt})_{\{t\} \times V_{x+mt}} \ar[d,"{m^*}"] \ar[r] &H_{T/T(a)}(X^{a})_{V_{x+mt} - x-mt} \ar[d,equal] \\
H_{K(t,x)}(L^2X^{t,x})_{\{t\} \times V_x} \ar[r] & H_{T/T(a)}(X^{a})_{V_x - x} \\
\end{tikzcd}
\end{equation}
with vertical maps induced by the action of $m$, ane horizontal maps as in the proof of Lemma \ref{equal1}. By the proof of Lemma \ref{equal1}, diagram \eqref{pall} commutes if 
\begin{equation}\label{pal1}
\begin{tikzcd}[row sep=large, column sep=3cm]
E(K(t,x)) \times_{K(t,x)} L^2X^{t,x} \ar[d,"{m}"] \ar[r,"{E\nu^{-1} \times ev_{t,x}}"] \ar[d] & E(T/T(a)) \times_{T/T(a)} X^a \ar[d,equal] \\
E(K(t,x+mt)) \times_{K(t,x+mt)} L^2X^{t,x+mt} \ar[r,"{E\nu^{-1} \times ev_{t,x+mt}}"] & E(T/T(a)) \times_{T/T(a)} X^{a}\\
\end{tikzcd}
\end{equation}
commutes, where $ev_{t,x}: L^2X^{t,x}\cong X^a$ is equivariant with respect to $\nu^{-1}: K(t,x) \cong T/T(a)$. To see that this commutes, note firstly that $\nu$ is induced by the inclusion $T \hookrightarrow \bT^2 \times T$ of the fixed points of the Weyl action $(r,t) \mapsto (r,t+mr)$, which implies that $\nu^{-1} \circ m = \nu^{-1}$. Secondly, note that the action of $m$ on a loop $\gamma$ fixes $\gamma(0,0)$, which implies that $ev_{t,x+mt}\circ m = ev_{t,x}$. These two observations imply the commutativity of diagram \eqref{pal1}, which completes the proof.
\end{proof}

\begin{remark}\label{gluh}
Our aim is now to determine the gluing maps associated to the local description in Theorem \ref{iso}. Let $X$ be a finite $T$-CW complex, let $\U$ be a cover adapted to $\S(X)$, and let $a,b \in E_{T,t}$ be such that $U_a \cap U_b \neq \emptyset$. Choose $x\in \zeta_{T,t}^{-1}(a)$ and $y \in \zeta_{T,t}^{-1}(b)$ such that $V_x \cap V_y \neq \emptyset$, which implies that $V_{t,x} \cap V_{t,y} \neq \emptyset$. It follows from Lemma \ref{hard} that either $X^b \subset X^a$ or $X^a \subset X^b$. We may assume that $X^b \subset X^a$. Let $U$ be an open subset in $U_a \cap U_b$ and let $V \subset V_{x} \cap V_{y}$ be such that $V \cong U$ via $\zeta_{T,t}$. Let $H$ be the subgroup of $T$ generated by $T(a)$ and $T(b)$. Consider the composite of isomorphisms
\[
\begin{array}{rcl}
H_T(X^a) \otimes_{H_T} \O_{\t_\C}(V - x) &\cong & H_T(X^b) \otimes_{H_T}  \O_{\t_\C}(V - x) \\
&\cong& H_{T/H}(X^b) \otimes_{H_T}  \O_{\t_\C}(V - x) \\ 
&\cong& H_{T/H}(X^b) \otimes_{H_T}  \O_{\t_\C}(V - y) \\
&\cong& H_T(X^b) \otimes_{H_T}  \O_{\t_\C}(V - y)
\end{array}
\]
The first map is induced by the inclusion $X^b \subset X^a$, the second map and fourth maps are induced by the change of groups map of Proposition \ref{changeh}, and the third map is  
\[
\id \otimes t^*_{y-x}: H_{T/H}(X^b) \otimes_{H_{T/H}} \O_{\t_\C}(V - x) \longrightarrow H_{T/H}(X^b) \otimes_{H_{T/H}} \O_{\t_\C}(V - y),
\]
which is a well defined map of $\O_{\t_\C}$-algebras, by Lemma \ref{gart}. All maps thus preserve the $\Z/2\Z$-graded $\O_{\t_\C}$-algebra structure. We have a commutative diagram
\[
\begin{tikzcd}
V - x \ar[r,"{t_{y-x}}"] \ar[d,"{\zeta_{T,t}}"] & V - y \ar[d,"{\zeta_{T,t}}"] \\
U - a \ar[r,"{t_{b-a}}"] & U - b 
\end{tikzcd}
\]
of complex analytic isomorphisms. Via this diagram, the composite above is canonically identified with 
\[
\begin{array}{rcl}
H_T(X^a) \otimes_{H_T} \O_{E_{T,t}}(V - x) &\cong & H_T(X^b) \otimes_{H_T}  \O_{E_{T,t}}(U - a) \\
&\cong& H_{T/H}(X^b) \otimes_{H_T}  \O_{E_{T,t}}(U - a) \\ 
&\cong& H_{T/H}(X^b) \otimes_{H_T} \O_{E_{T,t}}(U - b) \\
&\cong& H_T(X^b) \otimes_{H_T}  \O_{E_{T,t}}(U - b),
\end{array}
\]
where the fourth map is $\id \otimes t^*_{b-a}$. This is the gluing map $\phi_{b,a}$ of Grojnowski's construction (see Remark \ref{perry}), and we will show in Theorem \ref{duhh} that this is the gluing map associated to Theorem \ref{iso}.
\end{remark}

\begin{lemma}\label{gart}
With the hypotheses of Remark \ref{gluh}, the translation map
\[
t^*_{y -x}: \O_{\t_\C}(V - x) \longrightarrow \O_{\t_\C}(V - y)
\]
is $H_{T/H}$-linear.
\end{lemma} 

\begin{proof}
Let $x = x_1t_1 + x_2t_2$ and $y = y_1t_1 + y_2t_2$. By the same argument as in the proof of Lemma \ref{hard}, we have that 
\[
(y_1-x_1, y_2-x_2) \in \mathrm{Lie}(H) \times \mathrm{Lie}(H),
\]
since $T(a),T(b) \subset H$. Therefore
\[
y - x = (y_1-x_1)t_1 + (y_2-x_2)t_2 \in \mathrm{Lie}(H) \otimes_\R (\R t_1 + \R t_2) = \mathrm{Lie}(H) \otimes_\R \C.
\]
This implies the result.
\end{proof}

\begin{theorem}\label{duhh}
With the hypotheses of Remark \ref{gluh}, the gluing map associated to the local description in Theorem \ref{iso} on an open subset $U \subset U_{a} \cap U_{b}$ is equal to the composite map in Remark \eqref{gluh}.
\end{theorem}

\begin{proof}
We must first make a few observations before we can make sense of the diagram which we will use to prove the theorem. Let $K$ be the subgroup of $\bT^2 \times T$ generated by $T(t,x)$ and $T(t,y)$. We have $H \subset T \cap K$, since
\begin{gather*}
H = \langle T(a),T(b) \rangle = \langle T \cap T(t,x), T \cap T(t,y) \rangle \\ \subset T \cap \langle T(t,x),T(t,y) \rangle = T \cap K.
\end{gather*}
Furthermore, the inclusion $T \subset \bT^2 \times T$ induces an isomorphism $T/(T\cap K) \cong (\bT^2 \times T)/K$, as may be verified by chasing a diagram analogous to \eqref{rect}. We therefore have identifications 
\begin{gather*}
X^{T \cap K} = \Map_T(T/(T\cap K),X) \cong \Map_T((\bT^2 \times T)/K,X) = (L^2X)^K = L^2X^{t,y} \cong X^b, 
\end{gather*}
where the mapping spaces $\Map_T(-,-)$ of $T$-equivariant maps are identified with fixed-point spaces by evaluating at $0 \in T$. The composite is $T$-equivariant if we let $T$ act on mapping spaces via the target space. Thus, $X^b$ is fixed by $T\cap K$, and the homeomorphism $X^b \cong LX^{t,y}$ is equivariant with respect to $T/(T\cap K) \cong (\bT^2 \times T)/K$. Finally, note that Lemma \ref{hard} implies that 
\begin{equation}\label{dia2}
\begin{tikzcd}
L^2X^{t,y} \ar[r,hook,"{i_{y,x}}"] \ar[d,"{\cong}"',"{ev_{t,y}}"] & L^2X^{t,x} \ar[d,"{\cong}"',"{ev_{t,x}}"]  \\
X^b \ar[r,hook,"{i_{b,a}}"] & X^a
\end{tikzcd}
\end{equation}
commutes. Now, the diagram is as follows.

\begin{equation}\label{dirt}
\begin{tikzcd}[row sep=large, column sep=0.5cm]
H_T(X^a)_{V-x} \ar[r] \ar[d] \drar[phantom, "(1)"] & H_{T/T(a)}(X^a)_{V-x} \ar[r] \ar[d] \drar[phantom, "(2)"] & H_{K(t,x)}(L^2X^{t,x})_{\{t\} \times V} \ar[r] \ar[d] \drar[phantom, "(1)"] & H_{\bT^2\times T}(L^2X^{t,x})_{\{t\} \times V} \ar[d] \\
H_T(X^b)_{V-x} \ar[r] \ar[d]\drar[phantom, "(3)"] & H_{T/T(a)}(X^b)_{V-x} \ar[r] \ar[d] \drar[phantom, "(4)"]& H_{K(t,x)}(L^2X^{t,y})_{\{t\} \times V} \ar[r] \ar[d]  \drar[phantom, "(3)"]& H_{\bT^2\times T}(L^2X^{t,y})_{\{t\} \times V} \ar[d] \\
H_{T/H}(X^b)_{V-x} \ar[r] \ar[d] \drar[phantom, "(5)"]& H_{T/(T\cap K)}(X^b)_{V-x} \ar[r] \ar[d] \drar[phantom, "(6)"]& H_{(\bT^2 \times T)/K}(L^2X^{t,y})_{\{t\} \times V} \ar[r] \ar[dl, bend left=15] \dar[phantom, "(4)"] & H_{K(t,y)}(L^2X^{t,y})_{\{t\} \times V} \ar[dl] \\
H_{T/H}(X^b)_{V-y} \ar[r] \ar[d] & H_{T/(T\cap K)}(X^b)_{V-y} \dlar[phantom, "(3)"'] & H_{T/T(b)}(X^b)_{V-y} \ar[dll] \ar[l] \\
H_T(X^b)_{V-y} \\
\end{tikzcd}
\end{equation}

Each map is an isomorphism of $\O_{\t_\C}(V)$-algebras, and is exactly one of the following four types:
\begin{itemize}
\item the change of groups map of Proposition \ref{changeh} (if the target and source only differ by equivariance group);
\item the map induced by an inclusion of spaces (if the target and source only differ by the topological space); 
\item the translation map $\id \otimes t^*_{y-x}$ of Remark \ref{gluh} (these are the vertical maps of region (5) - note that translation by $y-x$ is $H_{T/(T\cap K)}$-linear, since $H \subset T \cap K$); or
\item the map of Lemma \ref{equal1}, or a map obtained in a way exactly analogous to the proof of Lemma \ref{equal1} (see below).
\end{itemize}

We will show that diagram \eqref{dirt} commutes, which implies that the two outermost paths from $H_T(X^a)_{V-x}$ to $H_T(X^b)_{V-y}$ are equal. This gives us the statement of the theorem.

Each of the regions numbered in the diagram above commutes for the corresponding reason stated below.
\begin{enumerate}
\item By naturality of the isomorphism of Proposition \ref{changeh}.
\item By naturality of the isomorphism of Lemma \ref{equal1}, along with diagram \eqref{dia2}.
\item By Lemma \ref{kurt}.
\item This holds essentially because the homeomorphism $X^b \cong LX^{t,y}$ is equivariant with respect to the inclusion $T \subset \bT^2 \times T$. (See below for a more detailed proof.)
\item This is obvious.
\item Note that both $(t,x)$ and $(t,y)$ are contained in the complexified Lie algebra of $K$, and are therefore also in the kernel of 
\[
\C^2 \times \t_\C \twoheadrightarrow \mathrm{Lie}((\bT^2 \times T)/K)_\C.
\]
Thus, in the proof of Lemma \ref{equal1}, we may translate by either $(t,x)$ or $(t,y)$, leading to the horizontal and diagonal maps, respectively. These two maps evidently commute with the vertical map, which is $\id \otimes t^*_{y-x}$.
\end{enumerate}

In more detail, the claim of item 4 holds by a proof similar to that of Lemma \ref{equal1}. Indeed, the identification $X^b \cong LX^{t,y}$ is equivariant with respect to the isomorphisms $T/(T\cap K) \cong (\bT^2 \times T)/K$, $T/T(a) \cong K(t,x)$ and $T/T(b) \cong K(t,y)$, which are all induced by the inclusion $T \subset \bT^2 \times T$. This implies that, in the case of the middle square, we have an isomorphism of the diagram 
\begin{equation}
\begin{tikzcd}
E(T/T(a))  \times_{T/T(a)} X^{b} \ar[d] \ar[r] & B(T/T(a)) \ar[d] \\
E(T/(T\cap K)) \times_{T/(T\cap K)} X^b \ar[r] &  B(T/(T\cap K)), \\
\end{tikzcd}
\end{equation}
which induces the left vertical map, and the diagram  
\begin{equation}
\begin{tikzcd}
E(K(t,x)) \times_{K(t,x)} L^2X^{t,y} \ar[d] \ar[r] & B(K(t,x)) \ar[d] \\
E((\bT^2 \times T)/K) \times_{(\bT^2 \times T)/K} L^2X^{t,y} \ar[r] &  B((\bT^2 \times T)/K), \\
\end{tikzcd}
\end{equation}
which induces the right vertical map. We therefore have a commutative square 
\[
\begin{tikzcd}
H_{T/T(a)}(X^b) \ar[r] \ar[d] & H_{K(t,x)}(L^2X^{t,y}) \ar[d] \\
H_{T/(T\cap K)}(X^b) \otimes_{H_{T/(T\cap K)}} H_{T/T(a)} \ar[r] & H_{(\bT^2 \times T)/K}(L^2X^{t,y}) \otimes_{H_{(\bT^2 \times T)/K}} H_{K(t,x)}
\end{tikzcd}
\]
of isomorphisms of $H_{T/T(a)} \cong H_{K(t,x)}$-algebras. Then, to get isomorphisms of $\O_{\t_\C}(V)$-algebras, we tensor the left hand side over $p_a \circ t_{-x}$ and the right hand side over $p_{t,x}\circ \iota_t$, as in diagram \eqref{dg}. One shows the commutativity of the other square labelled (4) in exactly the same way, replacing $T(a)$ with $T(b)$ and $K(t,x)$ with $K(t,y)$.
\end{proof}

\begin{corollary}\label{glutes}
There is an isomorphism of cohomology theories 
\[
\Ell^*_{T,t} \cong \G^*_{T,t},
\]
defined on the category of finite $T$-CW complexes and taking values in the category of sheaves of $\Z/2\Z$-graded $\O_{E_{T,t}}$-algebras.
\end{corollary}

\begin{proof}
The local description of $\Ell_{T,t}(X)$ given in Theorems \ref{iso} and \ref{duhh} amounts to the natural isomorphism that we require. It is clear that this is compatible with suspension isomorphisms since, in each theory, these are induced by the suspension isomorphism in ordinary cohomology, by the proofs of Propositions \ref{kerry} and \ref{jezza}.
\end{proof}

\chapter{Elliptic cohomology and single loops}\label{Kitch}

In the 2014 paper \cite{Kitch1}, Kitchloo constructed, for a simple, simply-connected compact Lie group $G$, a $G$-equivariant elliptic cohomology theory modeled on the loop group-equivariant K-theory of the free loop space $LX$. Kitchloo noted in that paper that he expects his theory to be closely related to Grojnowski's theory, and we aim to show in the next chapter that they are indeed isomorphic, for certain spaces $X$ and up to a twist by a line bundle.   \par

In this chapter, we construct a cohomology theory that is in a sense intermediate between Kitchloo's theory and Grojnowski's theory. We first construct a holomorphic sheaf $\K_{LT}(LX)$ as an inverse limit of rings $K_{\bT \times T}(Y)$ over the finite subcomplexes $Y$ of $LX$, where the torus $\bT \times T$ arises as the maximal torus of an "extended loop group" $\bT \ltimes LT$. The inverse limit $\K_{LT}(LX)$ is equipped with a natural action of the associated Weyl group $\check{T}$, as is the underlying space of $\K_{LT}(LX)$, which is the complexification $\C^\times \times T_\C$ of the maximal torus. One obtains an elliptic cohomology theory from this by pushing forward over the quotient by the action of the Weyl group, and then taking invariants to produce a coherent sheaf $\F_T(X)$. It turns out that the underlying elliptic curves of this construction are the multiplicative curves $C_q$ indexed over $q \in \D^\times$, whereas those of the previous chapter were additive. Finally, we use the equivariant Chern character to show that the theory $\F_T(X)$, over an elliptic curve $C_q$, is isomorphic to Grojnowski's theory over $E_\tau$, where $q = e^{2\pi i\tau}$. It follows that $\F_T(X)$ is a cohomology theory in $X$. \par

We have tried to emphasise the parallelism between this construction and that of the previous chapter. Indeed, the constructions are almost identical, after substituting equivariant K-theory for ordinary equivariant cohomology, and single loops for double loops. There are however some important differences between the constructions, such as the fact that one is a Borel-equivariant theory and the other is genuine equivariant theory, and this has made it necessary to treat them separately. Unfortunately, this approach has resulted in many proofs being formally similar to those of the previous chapter, but we have included them nevertheless, to convince the reader of all the details. 

\begin{notation}
We set some notation for use in this chapter and the next. We now write $\bT$ for the multiplicative circle $S^1 \subset \C$. The canonical isomorphism $\check{T} \otimes S^1 \cong T$ provides us with multiplicative coordinates for $T$. Furthermore, we have a canonical inclusion of $T$ into $T_\C$ which is given by the map
\[
\check{T} \otimes \bT \hookrightarrow \check{T} \otimes \C^\times, 
\]
induced by the inclusion $S^1 \hookrightarrow \C^\times$.
\end{notation}

\section{The complex manifold $C_T$}

In this section we construct the complex manifold $C_T$, which plays a role analogous to that of $E_T$ in the previous chapter. 

%change this to just $\bT$, but write a footnote saying that the group of differmorphisms analogous to the previous chapter actually includes $\{\pm 1\}$. But we leave this out as it doesn't seem to add any interesting structure.
\begin{definition}
Let $LT$ be the topological group of smooth maps $\bT \to T$, with group multiplication defined pointwise. The group of rotations of $\bT$ is also denoted by $\bT$, which acts on itself by multiplication. We let a rotation $r \in \bT$ act on a loop $\gamma \in LT$ from the left by
\[
r\cdot \gamma(s) = \gamma(s r^{-1}).
\]
The \textit{extended loop group} of $T$ is the semidirect product 
\[
\bT \ltimes LT,
\]
with group operation given by
\[
(r,\gamma(s))(r',\gamma'(s)) = (rr', \gamma(sr') \gamma'(s)).
\]
It is easily verified that the inverse of $(r,\gamma(s))$ is $(r^{-1},\gamma(sr^{-1}))^{-1}$.
\end{definition}

\begin{remark}
The extended loop group acts on the single free loop space
\[
LX = \Map(\bT,X)
\]
by
\[
(r,\gamma(s)) \cdot \gamma'(s) = \gamma(r^{-1}s)\gamma'(r^{-1}s).
\]
\end{remark}

\begin{remark}[Maximal torus and normaliser]
The subgroup 
\[
\bT \times T \subset \bT \ltimes LT
\]
which is a product of the rotation group and the subgroup of constant loops $T \subset LT$ is a torus. In fact, $\bT \times T$ is a maximal torus in $\bT \ltimes LT$, because a nonconstant loop $\gamma \in LT$ does not commute with $\bT$. Let $N(\bT \times T)$ be the normaliser of $\bT \times T$ in $\bT \ltimes LT$. In the following proposition, we will consider the subgroup
\[
\check{T} \subset LT
\]
where we regard an element $m$ in $\check{T} = \Hom(\bT,T) \subset LT$ as a loop $s^m$ in $T$.
\end{remark}

\begin{proposition}
The subgroup $\check{T} \subset \bT \ltimes LT$ is contained in $N_{\bT\ltimes LT}(\bT \times T)$, and the composite map
\[
\check{T} \hookrightarrow N_{\bT\ltimes LT}(\bT \times T) \twoheadrightarrow W_{\bT\ltimes LT}(\bT \times T)
\]
is an isomorphism.
\end{proposition}

\begin{proof}
The first statement is clear, since for $m \in \check{T}$, we have 
\[
(1,s^m) (r,t) (1,s^{-m}) = (r,s^mr^mt)(1,s^{-m}) = (r,r^mt) \in \bT \times T.
\]
For the second statement, we will define an inverse to the composite map in the proposition. Let $g \in N_{\bT \ltimes LT}(\bT \times T)$, and let $[g]$ be its image in $W_{\bT\ltimes LT}(\bT \times T)$. Since we may translate and multiply $g$ by constant loops without changing $[g]$, we can find a loop $\gamma \in LT$ with $\gamma(1) = 1$ such that $[g] = [(1,\gamma)]$. For $(r,t) \in \bT \times T$, we have
\begin{equation}\label{aff}
(1,\gamma(s)) (r,t)(1,\gamma(s)^{-1}) = (1,r,\gamma(rs)t \gamma(s)^{-1}) 
\end{equation}
which must lie in $\bT \times T$. It follows that
\[
\gamma(rs) \gamma(s)^{-1} = \gamma(r)
\]
for all $r,s \in \bT$, and so $\gamma(r)\gamma(s) = \gamma(rs)$ for all $r,s\in \bT$. Therefore, $\gamma$ is an element of $\check{T} = \Hom(\bT,T)$. The map $[g] \mapsto \gamma$ is well defined, since if $g \in \bT \times T$ then we could choose $\gamma = 1$. It is clearly a homomorphism, and inverse to the composite map of the proposition. This completes the proof.
\end{proof}

\begin{remark}[Weyl action]\label{weyl}
Let $\gamma(s) = s^m$ be the loop corresponding to $m \in \check{T}$. It follows directly from equation \eqref{aff} that the action of $\check{T}$ on $\bT \times T$ is given by
\[
m \cdot (r,t) = (r,r^m t).
\]
The induced action on the complexification 
\[
\C^\times \times T_\C := \C^\times \times \check{T} \otimes \C^\times 
\]
of $\bT \times T$ is given by the same formula, in which case we write it as
\[
m \cdot (q,u) = (q,q^mu).
\]
\end{remark}

\begin{remark}
Note that $\D^\times \times T_\C$ is preserved by the Weyl action of $\check{T}$ on $\C^\times \times T_\C$. Since the action of the first factor of $\check{T}^2$ on $\D^\times \times T_\C$ is free and properly discontinuous, the quotient
\[
\psi_T: \D^\times \times T_\C \longrightarrow  \check{T} \backslash (\D^\times \times T_\C) =: C_T
\]
is a complex manifold, and $\psi_T$ is a map of complex manifolds over $\D^\times$. On the fiber over $q$, it restricts to the quotient map
\[
\psi_{T,q}: T_\C \longrightarrow q^{\check{T}} \backslash T_\C= C_{T,q}.
\]
Consider the complex analytic diagram 
\[
\begin{tikzcd}
\X^+ \times \t_\C \ar[r] \ar[d,two heads,"{\zeta_T}"] & \bH \times \t_\C \ar[r,"{\exp}"] \ar[d,two heads,"{\C^\times \backslash \zeta_T}"] & \D^\times \times T_\C \ar[d,two heads,"{\psi_T}"] \\
E_T \ar[r] & \C^\times \backslash E_T \ar[r] & C_T.
\end{tikzcd}
\]
We label the composite of the two upper arrows by $\phi_T$ and the two lower arrows by $\sigma_T$. The two upper horizontal maps are given by
\[
\begin{array}{rcccl}
(t_1,t_2,x)&\longmapsto & (t_1/t_2,x/t_2) &\longmapsto & (e^{2\pi i t_1/t_2},\exp(2\pi i (x/t_2))).
\end{array}
\]
It is easily verified that these are $\check{T}^2$-equivariant maps, where $\check{T}^2$ acts on $\bH \times \t_\C$ by $m\cdot (\tau,x) = (\tau,x +m_1\tau + m_2)$, and the second factor of $\check{T}^2$ acts trivially on $\D^\times \times T_\C$. The middle vertical arrow labeled $\C^\times \backslash \zeta_T$ is the map induced by $\zeta_T$ on $\C^\times$-orbits, and the two lower horizontal arrows are the maps induced on $\check{T}^2$-orbits by the two upper horizontal arrows. The restriction 
\[
\sigma_{T,t}: E_{T,t} \cong C_{T,q}
\]
of $\sigma_T$ to the fiber over $t \in \X^+$ is the isomorphism given by tensoring $\check{T}$ with the composite $E_{t} \cong E_{t_1/t_2} \cong C_{q}$ given by $z \mapsto z/t_2 \mapsto \exp(2\pi i z/t_2)$.
\end{remark}

\begin{definition}\label{saag}
Let $\U$ be an open cover of $T \times T$ adapted to $\S$. For $(q,u) \in \D^\times \times T_\C$, we may write $u = q^{x_1}\exp(2\pi ix_2)$ for $x_1$ unique in $\t$ and $x_2 \in \t$ unique modulo $\check{T}$. In other words, $(q,u)$ is the image of $(t,x_1,x_2)$ under the composite map
\[
\X^+ \times \t \times \t \xrightarrow{\xi_T} \X^+ \times \t_\C \xrightarrow{\phi_T} \D^\times \times T_\C.
\]
Consider the open subset
\[
V_{q,u} := (\phi_T \circ \xi_T) (\X^+ \times V_{x_1,x_2}) \subset \D^\times \times T_\C.
\]
We have $(q,u) \in V_{q,u}$, since $(x_1,x_2) \in V_{x_1,x_2}$. The set 
\[
\{ V_{q,u} \}_{(q,u) \in \D^\times \times T_\C}
\]
is an open cover of $\D^\times \times T_\C$ (with some redundant elements). The set
\[
\{\psi_T(V_{q,u})\}_{(q,u) \in \D^\times \times T_\C}
\]
is an open cover of $C_T$.
\end{definition}

\begin{definition}\label{sarg}
Let $\U$ be an open cover of $T \times T$ adapted to $\S$. Given $a  \in C_{T,q}$, define the open subset
\[
U_a := \psi_{T}(V_{q,u}) \cap C_{T,q}
\]
where $u \in T_\C$ is any element satisfying $a= \psi{T,q}(u)$, so that $a \in U_a$. The set of open subsets $\{U_a\}_{a\in C_{T,q}}$ is an open cover of $C_{T,q}$. Furthermore, we have
\[
\begin{array}{rcl}
U_{a} &=& (\psi_T \circ \phi_T \circ \xi_T) (\X^+ \times V_{x_1,x_2})  \cap C_{T,q} \\
&=& (\sigma_T \circ \zeta_T \circ \xi_T) (\X^+ \times V_{x_1,x_2})  \cap C_{T,q} \\
&=& \sigma_T ((\zeta_T \circ \xi_T) (\X^+ \times V_{x_1,x_2})  \cap E_{T,t}) \\
&=& \sigma_{T,t}(U_a)
\end{array}
\]
where, in the last lines, $t$ is any element in $\X^+$ satisfying $q = e^{2\pi i t_1/t_2 }$, and $U_a$ is an element of the open cover of $E_{T,t}$ defined in Definition \ref{sag}.
\end{definition}

\section{The equivariant K-theory of single loop spaces}

The single loop space $LX$ carries an action of $\bT \times T$, which is induced by the action of $\bT \ltimes LT$. We begin this section by defining the holomorphic sheaf $\K_{LT}(LX)$ as an inverse limit over all $\bT \times T$-equivariant K-theory rings of finite subcomplexes of $LX$, as in the definition of $\H_{L^2T}(L^2X)$. We then prove several technical lemmas, some of which rely on the corresponding results in Section \ref{low}, concerning the relationship between the fixed point subspaces $LX^{q,u} \subset LX$ and the open cover $V_{q,u}$ defined in the previous chapter. Finally, we use these technical lemmas to prove localisation results which are analogous to those obtained in Section \ref{low}.

\begin{remark}
Let $X$ be a finite $T$-CW complex. Then $LX$ is weakly $\bT \times T$-homotopy equivalent to a $\bT \times T$-CW complex (by Theorem 1.1 in \cite{LMS}). 
\end{remark}

\begin{definition}\label{varlim}
Define the sheaf of $\O_{\D^\times \times T_\C}$-algebras
\[
\K^*_{LT}(LX) := \varprojlim_{Y \subset LX \: \text{finite}} \K^*_{\bT \times T}(Y)_{\D^\times \times T_\C}
\]
where the inverse limit runs over all finite $\bT \times T$-CW subcomplexes $Y$ of $LX$.
\end{definition}

\begin{remark}
The inverse limit sheaf $\K^*_{LT}(LX)$ takes the value 
\[
\varprojlim_{Y \subset LX} K^*_{\bT \times T}(Y)_U := \varprojlim_{Y \subset LX \: \text{finite}} K^*_{\bT \times T}(Y) \otimes_{K_{\bT \times T}} \O_{\D^\times \times T_\C}(U)
\]
on an open set $U\subset \D^\times \times T_\C$. We have that $\K^*_{LT}(L*) = \O_{\D^\times \times T_\C}$, by construction.
\end{remark}

\begin{remark}
We have that $\K^*_{LT}(L*) = \O_{\D^\times \times T_\C}$, by construction.
\end{remark}

\begin{definition}
Let $(q,u) \in \D^\times \times T_\C$. Define the intersection
\[
T(q,u) = \bigcap_{(q,u) \in H_\C} H
\]
of closed subgroups $H \subset \bT \times T$. For a $\bT \times T$-space $Y$, denote by $Y^{q,u}$ the subspace of points fixed by $T(q,u)$.
\end{definition}

%In Nitu's proof of the following, I don't know why he is able to restrict attention only to those characters which are products of $e^\alpha - 1$. You should have to consider all characters. 

\begin{lemma}[Kitchloo, \cite{Kitch1}, Theorem 3.3]\label{pushh}
Let $(q,u) \in \D^\times \times T_\C$ and let $\tau \in \bH$ such that $q = e^{2\pi i \tau}$. Write $u = q^{x_1}\exp(2\pi ix_2)$ where $x_1 \in \t$ is unique and $x_2 \in \t$ is unique modulo $\check{T}$. We have
\[
T(q,u) = \langle \exp(2\pi i (r,x_1 r)),\exp(2\pi i x_2) \rangle_{r \in \R} \subset \bT \times T
\]
and
\[
T(q,u) \cap T = T(a).
\]
\end{lemma}

\begin{proof}
Recall that 
\[
\C^\times \times T_\C = \Hom(\hat{\bT} \times \hat{T},\C^\times).
\]
By definition of $T(q,u)$, the map $(q,u) : \hat{\bT} \times \hat{T} \to \C^\times $, which is given by evaluation at $(q,u)$, factors uniquely   
\begin{equation}\label{dia}
\begin{tikzcd}
\hat{\bT} \times \hat{T} \ar[r,"{(q,u)}"] \ar[d,"\rho"] & \C^\times \\
\hat{T}(q,u) \ar[ur,"{\phi}"]
\end{tikzcd}
\end{equation}
where $\rho$ is given by restriction of characters. Since $T(q,u)$ is the smallest subgroup to have this property, the map $\phi$ is an isomorphism of $\hat{T}(q,u)$ onto the image of $(q,u)$. Therefore, the kernel of $(q,u)$ is equal to the kernel of $\rho$, which says that an irreducible character of $\bT \times T$ is trivial on $(q,u)$ if and only if it is trivial on $T(q,u)$. Furthermore, $T(q,u)$ is completely determined by the irreducible characters of $\bT \times T$ which restrict trivially to it, by the Pontryagin duality theorem.

An irreducible character $\mu \in \hat{\bT} \times \hat{T}$ is trivial on $(q,u)$ if and only if $\mu(\tau, x_1\tau + x_2) = \tau \mu(1,x_1) + \mu(0,x_2) \in \Z$, which holds if and only if $\mu(1,x_1) = 0$ and $\mu(0,x_2) \in \Z$. This means that $\mu$ is trivial on the subgroup
\[
\langle \exp(2\pi i (r,x_1 r)),\exp(2\pi i x_2) \rangle_{r \in \R} \subset \bT \times T
\]
which is equal to $T(q,u)$, by the reasoning above. To prove the second statement, notice that we have 
\[
T(q,u) \cap T = \langle \exp(2\pi i x_1), \exp(2\pi i x_2) \rangle.
\]
since the intersection of $T(q,u)$ with $T$ is the subset of points where $r$ is an integer. We identify the additive circle $\R/\Z$ with the multiplicative circle $S^1 \subset \C^\times$ via the map induced by $\exp: 2\pi i \R \to S^1$. Therefore, $\exp(2\pi i x_1) = \pi(x_1)$ and $\exp(2\pi i x_2) = \pi(x_2)$, so that
\[
T(q,u) \cap T = T(a).
\]
\end{proof}

\begin{lemma}\label{surjj}
Let $\psi_T(q,u) = a$. There is a short exact sequence of compact abelian groups 
\[
1 \to T(a) \hookrightarrow T(q,u) \twoheadrightarrow \bT \to 1
\]
where $T(q,u) \twoheadrightarrow \bT$ is the restriction of the projection $\bT \times T \twoheadrightarrow \bT$.
\end{lemma}

\begin{proof}
By the description in Lemma \ref{pushh}, the projection of $T(q,u) \subset \bT \times T$ onto $\bT$ is surjective, and has kernel $T(q,u) \cap T = T(a)$.
\end{proof}

%\begin{remark}
%It follows from Lemma \ref{surjj} that $T(q,u) \to \bT$ is a $T(a)$-principal bundle, and a subbundle of the trivial $T$-bundle $\bT \times T$. If $T(a)$ is not connected, then $T(q,u) \to \bT$ may not be trivial. Similarly, the complexification $T(q,u)_\C$ is a $T(a)_\C$-principal bundle over $\bT_\C$, the fiber of which over $q$ contains $a$.
%\end{remark}
\begin{notation}
Write
\[
p_a: T \twoheadrightarrow T/T(a) \qquad \mathrm{and}  \qquad p_{q,u}: \bT \times T \twoheadrightarrow (\bT \times T)/T(q,u)
\]
for the quotient maps, and let 
\[
\iota_1: T \hookrightarrow \bT \times T
\]
denote the inclusion of groups. 
\end{notation}

\begin{remark}
It follows from Lemma \ref{surjj} that there is a commutative diagram 
\begin{equation}\label{rectt}
\begin{tikzcd}
T(a) \ar[r,hook] \ar[d,hook] & T \ar[r, two heads,"p_a"] \ar[d, "\iota_1",hook] & T/T(a) \ar[d, "\eta", dashed] \\
T(q,u) \ar[r,hook] \ar[d,two heads] & \bT \times T \ar[r, two heads,"{p_{q,u}}"] \ar[d,two heads] & (\bT \times T)/T(q,u) \ar[d] \\
\bT \ar[r,equal] & \bT \ar[r] & 1
\end{tikzcd}
\end{equation}
where $\eta$ is induced by $\iota_1$.
\end{remark}

\begin{lemma}\label{groups} 
The map $\eta$ of diagram \eqref{rectt} is an isomorphism.
\end{lemma}

\begin{proof}
By Lemma \ref{surjj}, the left hand column is a short exact sequence. It is obvious that the middle column is a short exact sequence, as are all rows. It follows by a straightforward diagram chase that $\eta$ is an isomorphism.
\end{proof}

\begin{remark}
For a finite $T$-CW complex $X$, there is an action of $\bT \times T$ on $LX$ which is induced by the action of $\bT \ltimes LT$ on $LX$. Let $LX^{q,u}$ denote the subspace of $LX$ fixed by $T(q,u)$. Let 
\[
\begin{tikzcd}
ev: LX \ar[r]& X
\end{tikzcd}
\]
denote the $T$-equivariant map given by evaluation at $1$.
\end{remark}

\begin{lemma}[Kitchloo, \cite{Kitch1}, Corollary 3.4]\label{pushyy}
Let $(q,u) \in \D^\times \times T_\C$ and write $u = q^{x_1}\exp(2\pi ix_2)$. We have
\[
LX^{q,u} = \{ \gamma \in LX \, | \,  \gamma(s) = s^{x_1}\cdot z, \, z \in X^a\}.
\]
\end{lemma}

\begin{proof}
By Lemma \ref{pushh}, a loop $\gamma \in LX$ is fixed by $T(q,u)$ if and only if
\[
\gamma(s) = (r,r^{x_1}) \cdot \gamma(s) = r^{x_1} \cdot \gamma(sr^{-1})
\]
and
\[
\gamma(s) = \exp(2\pi i x_2) \cdot \gamma(s) 
\]
for all $r,s \in \bT$. By setting $s = r$, one sees that this is true if and only if 
\[
\gamma(r) = r^{x_1} \cdot \gamma(1)
\]
and $\gamma(1) \in X^a$ for $a = \psi_T(q,u)$. Therefore, 
\[
LX^{q,u} = \{ \gamma \in LX \, | \,  \gamma(s) = s^{x_1}\cdot z, \, z \in X^a\}.
\]
\end{proof}

%\begin{remark}\label{push}
%Consider the map $s: \bT \to  T/T(\alpha)$ obtained as the composition $\bT \to \bT \times T \to  (\bT \times T)/\bT(q,a) \cong T/T(\alpha)$, where the first map is the identity section of the trivial $T$-bundle over $\bT$. It is straightforward to show that the quotient group $(\bT \times T)/\bT(q,a)$ is uniquely isomorphic to the pushout

%\centerline{
%\xymatrix{
%\bT \ar[r]^s \ar[d]^{=} & T/T(\alpha) \ar[d]^\cong \\
%\bT \ar[r] & \bT \times_s T/T(\alpha).
%}}

%The induced map $ds_1: \R \to Lie(T/T(u))$ on Lie algebras yields a map $\R \to \t$ by postcomposing with the splitting section of $\R \to \t \to Lie(T/T(a))$. This may be regarded as a constant connection on the trivial $T$-bundle over $\bT$, since a connection on $\bT \times T$ is equivalent to a map $S^1 \to \Hom(\R \oplus \t, \t)$, and a map $a:\R \to \t$ determines a map $a \oplus id: \R \oplus \t \to \t$.
%\end{remark}

%\begin{remark}
%Fix a map $X^u \to T/T(u)$ and consider the $T(u)$-principal bundle $P(u)$ on $T/T(u)$ with the canonical connection $\omega$. The subspace $LX^{\bT(q,u)} \subset LX$ consists of exactly those loops in $X^u$ such that the pullback of $P(u)$ and $\omega$ to $S^1$ 
%\end{remark}

\begin{lemma}\label{equall}
The map $ev$ induces a $T/T(a)$-equivariant homeomorphism
\[
\begin{tikzcd}
ev_{q,u}: LX^{q,u} \ar[r,"\cong"]& X^{a}.
\end{tikzcd}
\]
natural in $X$, and equivariant with respect to $\eta$.
\end{lemma}

\begin{proof}
This is evident by Lemma \ref{pushyy}. 
\end{proof}

\begin{remark} 
The subspace $LX^{q,u} \subset LX$ is a $\bT \times T$-equivariant CW subcomplex, consisting of those equivariant cells in $LX$ whose isotropy group contains $T(q,u)$. In fact, it follows easily from Proposition \ref{equall} that $LX^{q,u}$ is a finite $\bT \times T$-CW complex, since $X$ is finite.
\end{remark}

\begin{lemma}\label{hardd}
Let $X$ be a finite $T$-CW complex and let $\U$ be an open cover of $T \times T$ adapted to $\S(X)$. Let $(q,u),(q',v) \in \D^\times \times T_\C$, with $a = \psi_T(q,u)$ and $b =\psi_T(q',v)$. We have the following.
\begin{enumerate}
\item If $V_{q,u} \cap V_{q',v} \neq \emptyset$, then either $X^b \subset X^a$ or $X^a \subset X^b$. 
\item If $V_{q,u} \cap V_{q',v} \neq \emptyset$ and $X^b \subset X^a$, then $LX^{q',v} \subset LX^{q,u}$.
\end{enumerate}
\end{lemma}

\begin{proof}
Since $V_{q,u} \cap V_{q',v} \neq \emptyset$, by Definition \ref{saag} there exist $x_1,y_1$ unique in $\t$, and $x_2,y_2\in \t$ unique modulo $\check{T}$, such that
\[
V_{x_1,x_2} \cap V_{y_1,y_2} \neq \emptyset,
\]
where $u = q^{x_1}\exp(2\pi i x_2)$ and $v = (q')^{y_1}\exp(2\pi i y_2)$. This implies that 
\[
U_{a_1,a_2} \cap U_{b_1,b_2} \neq \emptyset,
\]
for $a_i = \pi_T(x_i)$ and $b_i = \pi_T(y_i)$. Note that $a_2,b_2$ only depend on $x_2,y_2$ modulo $\check{T} = \ker (\pi_T)$. So, by the second property of an adapted cover, we have either $(a_1,a_2) \leq (b_1,b_2)$ or $(b_1,b_2) \leq (a_1,a_2)$. This implies that either $X^b \subset X^a$ or $X^a \subset X^b$, which yields the first part of the result.

For the second part, assume that $X^b \subset X^a$, and let $\gamma \in LX^{q',v}$. By the description in Lemma \ref{pushyy}, we have
\[
\gamma(s) = s^{y_1}\cdot z
\]
for some $z \in X^b \subset X^a$. Let $H \in \S(X)$ be the isotropy group of $z$, so that $a_i,b_i \in H$ for $i = 1,2$. The condition 
\[
V_{x_1,x_2} \cap V_{y_1,y_2} \neq \emptyset
\]
implies by Lemma \ref{hard3} that 
\[
(x_1,x_2)- (y_1,y_2)\in \pi^{-1}(H\times H)^0 = \mathrm{Lie}(H) \times \mathrm{Lie}(H).
\]
In particular, 
\[
x_1 - y_1 \in \mathrm{Lie}(H).
\]
We can now write
\[
\begin{array}{rcl}
\gamma(r) &=& s^{y_1}\cdot z \\
&=& s^{y_1} \cdot (s^{x_1 - y_1}\cdot z) \\
&=& s^{x_1}\cdot z,
\end{array}
\]
which is a loop in $LX^{q,u}$ since $z \in X^a$. This yields the second part of the result.
\end{proof}

\begin{lemma}\label{hayy}
Let $X$ be a finite $T$-CW complex and let $\U$ be a cover adapted to $\S(X)$. If $(q',v) \in V_{q,u}$, then $LX^{q',v} \subset LX^{q,u}$.
\end{lemma}

\begin{proof}
Since $(q',v) \in V_{q,u}$, by Definition \ref{saag} there exist $x_1,y_1$ unique in $\t$, and $x_2,y_2$ unique in $\t$ modulo $\check{T}$, such that
\[
(y_1,y_2) \in V_{x_1,x_2} \subset \t \times \t
\] 
and $u = q^{x_1}\exp(2\pi i x_2)$ and $v = (q')^{y_1}\exp(2\pi i y_2)$. Therefore, 
\[
(b_1,b_2) \in U_{a_1,a_2} \subset T\times T
\]
where $a_i = \pi(x_i)$ and $b_i = \pi(y_i)$. Note that $a_2$ and $b_2$ only depend on $x_2$ and $y_2$ modulo $\check{T}$. It follows that $(a_1,a_2) \leq (b_1,b_2)$ by Lemma \ref{hard2}, so that $X^b \subset X^a$, since $\U$ is adapted to $\S(X)$. Lemma \ref{hardd} yields the result.
\end{proof}

\begin{proposition}\label{extendd}
Let $Y \subset LX$ be a finite $\bT \times T$-CW subcomplex. The inclusion $Y^{q,u} \subset Y$ induces an isomorphism of $\Z/2\Z$-graded $\O_{V_{q,u}}$-algebras
\[
\K^*_{\bT\times T}(Y)_{V_{q,u}} \longrightarrow  \K^*_{\bT\times T}(Y^{q,u})_{V_{q,u}}.
\]
\end{proposition}

\begin{proof}
Let $(q',v) \in V_{q,u}$. Then $LX^{q',v} \subset LX^{q,u}$ by Lemma \ref{hayy}, which implies that $Y^{q',v} \subset Y^{q,u}$. Consider the commutative diagram
\begin{equation}
\begin{tikzcd}
\K^*_{\bT\times T}(Y)_{V_{q,u}} \ar[dr] \ar[rr] && \K^*_{\bT\times T}(Y^{q,u})_{V_{q,u}} \ar[dl] \\
&\K^*_{\bT\times T}(Y^{q',v})_{V_{q,u}}, &
\end{tikzcd}
\end{equation}
induced by the evident inclusions. Taking stalks at $(q',v)$, Theorem \ref{localisation} implies that the two diagonal maps are isomorphisms, and so the horizontal map is also an isomorphism. Therefore, the horizontal map is an isomorphism of $\O_{V_{q,u}}$-algebras.
\end{proof}

\begin{corollary}\label{stalks}
The inclusion $LX^{q,u} \subset LX$ induces an isomorphism of $\O_{V_{q,u}}$-algebras
\[
\K^*_{LT}(LX)_{V_{q,u}} \cong \K^*_{\bT \times T}(LX^{q,u})_{V_{q,u}},
\]
natural in $X$. In particular, we have an isomorphism of stalks
\[
\K^*_{LT}(LX)_{q,u} \cong \K^*_{\bT \times T}(LX^{q,u})_{q,u}.
\]
\end{corollary}

\begin{proof}
It follows from Definition \ref{varlim} and Proposition \ref{extendd} that 
\[
\begin{array}{rcl}
\K^*_{LT}(LX)_{V_{q,u}} &=& \varprojlim_{Y \subset LX} \K^*_{\bT \times T}(Y)_{V_{q,u}} \\
&\cong& \varprojlim_{Y \subset LX} \K^*_{\bT \times T}(Y^{q,u})_{V_{q,u}} \\
&=& \K^*_{\bT \times T}(LX^{q,u})_{V_{q,u}}.
\end{array}
\]
The final equality holds by definition of the inverse limit, since each $Y^{q,u}$ is contained in the finite $\bT \times T$-CW subcomplex $LX^{q,u} \subset LX$. We can show naturality with respect to a $T$-equivariant map $f: X \to Y$ by refining the cover $\U$ so that it is adapted to $\S(f)$. The isomorphism is then natural in $X$ by the functoriality of K-theory.
\end{proof}

\begin{remark}\label{coherence} 
It follows from Corollary \ref{stalks} that $\K^*_{LT}(LX)$ is a coherent sheaf, since $LX^{q,u}$ is a finite $\bT \times T$-CW complex.
\end{remark}

\section{Construction of the sheaf $\F^*_T(X)$ over $C_T$}

We begin this section by showing that $\K^*_{LT}(LX)$ has a more computable description. We then use this description to study the action of the Weyl group $\check{T}$ on the sheaf $\K_{LT}(LX)$, which is induced both by the $\bT \ltimes LT$-action on $LX$ and the Weyl action on $\bT \times T$. Finally, we define the sheaf $\F^*_T(X)$ as the $\check{T}$-invariants of the pushforward of $\K_{LT}(LX)$ along $\psi_T: \D^\times \times T_\C \twoheadrightarrow C_T$.

\begin{definition}
Let $\mathcal{D}(X)$ denote the set of finite $\bT \times T$-CW complexes generated by the set 
\[
\{ LX^{q,u} \}_{(q,u) \in \D^\times \times T_\C}
\]
under finite unions and finite intersections. The set $\mathcal{D}(X)$ is partially ordered by inclusion. 
\end{definition}

\begin{lemma}\label{yes}
There is an isomorphism of $\Z/2\Z$-graded $\O_{\D^\times \times T_\C}$-algebras
\[
\K^*_{LT}(LX) \cong \varprojlim_{Y \in \mathcal{D}(X)} \K^*_{\bT \times T}(Y)_{\D^\times \times T_\C}
\]
natural in $X$. 
\end{lemma}

\begin{proof}
Consider the union
\[
S := \bigcup_{(q,u) \in \D^\times \times T_\C} LX^{q,u} \subset LX.
\]
Notice that $S$ is a $\bT \times T$-CW subcomplex of $LX$, and that $S$ containes

For each finite $\bT \times T$-CW subcomplex $Y \subset LX$, we have an inclusion $Y \cap S \hookrightarrow Y$. The induced map of $\Z/2\Z$-graded $\O_{\D^\times \times T_\C}$-algebras
\begin{equation}\label{maps}
\varprojlim_{Y \subset LX \: \text{finite}} \K^*_{\bT \times T}(Y)_{\D^\times \times T_\C} \to \varprojlim_{Y \subset LX \: \text{finite}} \K^*_{\bT \times T}(Y \cap S)_{\D^\times \times T_\C}
\end{equation}
is natural in $X$, by the functoriality of K-theory. Let $(q,u)$ be an arbitrary point. By Corollary \ref{stalks}, the map \eqref{maps} induces an isomorphism of stalks at $(q,u)$ because $S$ contains $LX^{q,u}$. Therefore, the map \eqref{maps} is an isomorphism of sheaves.

It is clear that the set $\{S \cap Y \, | \, Y \subset LX \, \text{finite}\}$ is equal to the set of all finite equivariant subcomplexes of $S$. Therefore, the target of map \eqref{maps} is equal to the inverse limit
\begin{equation}\label{mapsss}
\varprojlim_{Y \subset S \: \text{finite}} \K^*_{\bT \times T}(Y)_{\D^\times \times T_\C}
\end{equation}
over all finite $\bT \times T$-equivariant subcomplexes $Y \subset S$. Any such $Y$, since it is finite, is contained in the union of finitely many $LX^{q,u}$, which is also a finite $\bT \times T$-CW complex. Therefore, by definition of the inverse limit, \eqref{mapsss} is equal to 
\[
\varprojlim_{Y \in \mathcal{D}(X)} \K^*_{\bT \times T}(Y)_{\D^\times \times T_\C},
\]
which yields the first isomorphism. 
\end{proof}

\begin{remark}
Lemma \ref{yes} means that we can define $\K^*_{LT}(LX) $ without putting a $\bT \times T$-CW complex structure on $LX$.
\end{remark}

\begin{remark}\label{melbourne}
Recall that $\gamma \in LT$ acts on $\gamma' \in LX$ by the pointwise action $\gamma(s) \cdot \gamma'(s)$ induced by the action of $T$ on $X$. Therefore, the group $\check{T} = \Hom(\bT,T)$ acts on $LX$ via its inclusion into $LT$. Let $(q,u) \in \D^\times \times T_\C$ and write $u = q^{x_1}\exp(2\pi ix_2)$. If $m \in \check{T}$, then $q^mu =  q^{x_1+m}\exp(2\pi ix_2)$. The action of $m$ on $LX$ induces a homeomorphism 
\[
\begin{array}{rcl}
\kappa_m: LX^{q,u} &\cong& m\cdot LX^{q,u} = LX^{q,q^mu}  \\
s^{x_1}\cdot z & \mapsto &s^{x_1+m}\cdot z
\end{array}
\]
which is equivariant with respect to the Weyl action $w_m: \bT \times T \cong \bT \times T$ of $m$ (see Remark \ref{weyl}). Applying the change of groups map corresponding to $w_m$ followed by pullback along $\kappa_m$ yields an isomorphism of graded rings
\begin{equation}\label{melb}
m^*: K^*_{\bT\times T}(LX^{q,q^mu}) \xrightarrow{w_m^*} K^*_{\bT\times T}(LX^{q,q^mu}) \xrightarrow{\kappa_m^*} K^*_{\bT \times T}(LX^{q,u}).
\end{equation}
This sends, for example, the trivial bundle $\C_\mu \to LX^{q,q^mu}$ corresponding to $\mu \in \hat{T}(q,q^mu)$ to the trivial bundle $\C_{m^* \mu} \to LX^{q,u}$, where $m^* \mu$ denotes the pullback of $\mu$ along $w_m: T(q,u) \cong T(q,q^mu)$. Thus,
\[
m^*\C_\mu = \C_{m^*\mu}.
\]
Let $\omega_m: \D^\times \times T_\C \rightarrow \D^\times \times T_\C$ be the action map $(q,u) \mapsto (q,q^mu)$ of $m$. The map $m^*$ induces an isomorphism of sheaves
\[
\omega_m^* \K^*_{\bT \times T}(LX^{q,q^mu})_{\D^\times \times T_\C} \xrightarrow{\cong} \K^*_{\bT \times T}(LX^{q,u})_{\D^\times \times T_\C}.
\]
We obtain a similar map for any $Y \in \mathcal{D}(X)$ by replacing $LX^{q,u}$ by $Y$ and $LX^{q,q^mu}$ by $m\cdot Y$, which is also in $\mathcal{D}(X)$. Since they are induced by the action of $m$ on $LX$, the family of maps thus produced is compatible with inclusions in $\mathcal{D}(X)$. By Lemma \ref{yes}, we therefore have an isomorphism of inverse limits
\[
\omega_m^*\K^*_{LT}(LX) \to \K^*_{LT}(LX).
\]
For each $(q,u)$, this induces an isomorphism, also denoted $m^*$, 
\begin{equation}\label{hogart}
\begin{array}{rcl}
m^*: \K^*_{\bT \times T}(LX^{q,q^mu})_{(q,q^mu)} &\longrightarrow& \K^*_{\bT \times T}(LX^{q,u})_{(q,u)} \\
{[V]} \otimes g &\longmapsto & m^*[V] \otimes m^*g
\end{array}
\end{equation}
from the stalk at $(q,q^mu)$ to the stalk at $(q,u)$. We write $m^*g$ as shorthand for $\omega_m^*g$. Since it is induced by a group action, the collection of isomorphisms
\[
\{\omega_m^*\K^*_{LT}(LX) \xrightarrow{m} \K^*_{LT}(LX)\}_{m \in \check{T}}
\] 
defines an action of $\check{T}$ on $\K^*_{LT}(LX)$. We have therefore proved the following theorem.
\end{remark}

\begin{theorem}\label{actyon}
There is a $\check{T}$-equivariant structure on the sheaf $\K^*_{LT}(LX)$ of $\Z/2\Z$-graded $\O_{\D^\times \times T_\C}$-algebras, given on stalks by the isomorphism \eqref{hogart}.
\end{theorem}

\begin{proof}
See Remark \ref{melbourne}.
\end{proof}

We denote by $\iota_q$ the inclusion of the fiber $T_\C \hookrightarrow \D^\times \times T_\C$ over $q \in \D^\times$. We have a commutative diagram of complex manifolds
\[
\begin{tikzcd}
T_\C \ar[r,hook,"{\iota_q}"] \ar[d,two heads, "{\psi_{T,q}}"] & \D^\times \times T_\C \ar[d,two heads, "{\psi_T}"] \\
C_{T,q} \ar[r,hook] & C_T. 
\end{tikzcd}
\]
With Remark \ref{coherence} in mind, we make the following definition.

\begin{definition}
Define the sheaf of $\Z/2\Z$-graded $\O_{C_T}$-algebras 
\[
\F^*_T(X) := ((\psi_{T})_*\, \K^*_{LT}(LX))^{\check{T}}.
\]
Define the sheaf of $\Z/2\Z$-graded $\O_{C_{T,q}}$-algebras 
\[
\F^*_{T,q}(X) :=  ((\psi_{T,q})_*\, \iota_q^*\, \K_{LT}(LX))^{\check{T}}.
\]
\end{definition}

\begin{remark}
We have that 
\[
\begin{array}{rcl}
\F^*_{T}(\pt) &=& ((\psi_T)_*\O_{\D^\times \times T_\C})^{\check{T}}  \\
&=& \O_{C_T}
\end{array}
\]
and 
\[
\begin{array}{rcl}
\F^*_{T,q}(\pt) &=& ((\psi_{T,q})_*\O_{T_\C})^{\check{T}} \\
&=& \O_{C_{T,q}},
\end{array}
\]
by construction.
\end{remark}

\section{The suspension isomorphism}

We now show that $\F^*_T(X)$ inherits a suspension isomorphism from equivariant K-theory (see Definition \ref{susp}). To do this, we first need a Lemma to which tells us how the suspension functor interacts with the free loop space functor. 

\begin{remark}
We regard the free loop space $L(S^1 \wedge X_+)$ as a pointed $\bT \times T$-CW complex with basepoint given by the loop $\gamma_*: \bT \to \pt \subset S^1 \wedge X_+$. Since the basepoint of $S^1 \wedge X_+$ is fixed by $T$, the loop $\gamma_*$ is fixed by $\bT \times T$, and so it is contained in $L(S^1 \wedge X_+)^{q,u}$ for all $(q,u)$. Therefore, each $Y \in \mathcal{D}(S^1\wedge X_+)$ is a based subcomplex of $L(S^1 \wedge X_+)$.
\end{remark}

\begin{lemma}\label{wedgie}
For each $(q,u) \in \D^\times \times T_\C$, we have an equality
\[
L(S^1\wedge X_+)^{q,u} = S^1 \wedge LX^{q,u}_+ 
\]
of subsets of $L(S^1 \wedge X_+)$.
\end{lemma}

\begin{proof}
Suppose that $\gamma$ is a loop in $L(S^1\times X_+)^{q,u}$ sending $s \mapsto (\gamma_1(s),\gamma_2(s))$. Write $\gamma(1) = (z_1,z_2)$ and $u = q^{x_1}\exp(2\pi i x_2)$. The loop $\gamma$ is fixed by $T(q,u)$ if and only if 
\[
(\gamma_1(s),\gamma_2(s)) = r^{x_1}\cdot (\gamma_1(sr^{-1}),\gamma_2(sr^{-1})) =  (\gamma_1(sr^{-1}),r^{x_1}\cdot \gamma_2(sr^{-1}))
\]
for all $r,s \in \bT$, where $S^1$ is fixed by $T$. Setting $r = s$, one sees that this is true if and only if
\[
(\gamma_1(r),\gamma_2(r)) =  (\gamma_1(1),r^{x_1}\cdot \gamma_2(1)) \in S^1 \times LX^{q,u}.
\]
Therefore, we have an equality
\[
L(S^1 \times X_+)^{q,u} = S^1 \times LX^{q,u}_+.
\]
To prove the equality of the lemma, consider that the image of $\gamma$ is contained in $1 \times X_+ \cup S^1 \times \pt$ if and only if, for each $s$, we have either $\gamma_1(s) = 1$ or $\gamma_2(s) = \pt$. By the above calculation, this is true if and only if $\gamma_1(s) = 1$ for all $s$, or $\gamma_2(s) = \pt$ for all $s$, which is true if and only if  $\gamma \in 1 \times LX^{q,u} \cup S^1 \times \pt$. The equality of the lemma follows.
\end{proof}

\begin{definition}
For a pointed finite $T$-CW complex $X$, define
\[
\tilde{\F}_T(X) := \ker(\F^0_T(X) \to \F^0_T(\pt)) \\
\]
\end{definition}

\begin{proposition}\label{gerry}
Let $X$ be a finite $T$-CW complex. We have
\[
\F^{-1}_T(X) = \tilde{\F}_T(S^1 \wedge X_+).
\]
\end{proposition}

\begin{proof}
We will show that there are equalities
\[
\begin{array}{rcl}
\F^{-1}_T(X) &:=& ((\psi_T)_* \varprojlim_{Y \in \bD(X)} \K^{-1}_{\bT \times T}(Y)_{\D^\times \times T_\C})^{\check{T}}\\
&:=& ((\psi_T)_* \varprojlim_{Y \in \bD(X)} \tilde{\K}_{\bT \times T}(S^1 \wedge Y_+)_{\D^\times \times T_\C})^{\check{T} }\\
&=& ((\psi_T)_* \varprojlim_{Y \in \bD(S^1 \wedge X_+)} \tilde{\K}_{\bT \times T}(Y)_{\D^\times \times T_\C})^{\check{T}} \\
&=:& ((\psi_T)_* \varprojlim_{Y \in \bD(S^1 \wedge X_+)} \ker(\K^0_{\bT \times T}(Y) \to \K^0_{\bT \times T}(\pt))_{\D^\times \times T_\C})^{\check{T}} \\
&=&  \ker(\F^0_T(S^1\wedge X_+) \to \F^0_T(\pt)) \\
&=:&  \tilde{\F}_T(S^1 \wedge X_+).
\end{array}
\]
The third line holds because the equality $L(S^1\wedge X_+)^{q,u} = S^1 \wedge LX^{q,u}_+$ of Lemma \ref{wedgie} is an equality of subsets of $L(S^1 \wedge X_+)$. It therefore extends to an equality of sets partially ordered by inclusion
\[
\bD(S^1 \wedge X_+) = \{S^1 \wedge Y_+\}_{Y\in \mathcal{D}(X)}.
\]
The fifth equality holds since the inverse limit is a right adjoint functor, and therefore respects all limits, including kernels. The remaining equalities hold by definition.
\end{proof}

\begin{remark}
By an argument exactly analogous to that of Remark \ref{hewson}, the map sending $X \mapsto \widetilde{\Ell}_T^*(X)$ defines a reduced cohomology theory on finite $T$-CW complexes, once it is equipped with the suspension isomorphisms of Proposition \ref{gerry}.
\end{remark}

\section{A local description}

In this section, given a finite $T$-CW complex $X$, we obtain a local description of the sheaf $\F_{T,q}$ over a cover adapted to $X$. We then use the local description to show that $\F_{T,q}(X)$ is naturally isomorphic to Grojnowski's construction for the curve $E_\tau$, where $\tau \in \bH$ is any complex number satisfying $q = e^{2\pi i \tau}$. This isomorphism turns out to be compatible with suspension isomorphisms, so that it is an isomorphism of cohomology theories on finite $T$-CW complexes.

\begin{notation}
Let $t_{q,u}$ denote translation by $(q,u)$ in $\C^\times \times T_\C$, and $t_u$ translation by $u$ in $T_\C$. If $f$ is a map of compact Lie groups, we will abuse notation and also write $f$ for the induced map of complex Lie groups. 
\end{notation}

\begin{remark}
It will be essential to the proof of the following Lemma to show that the diagram
\begin{equation}\label{daam}
\begin{tikzcd}
T_\C \ar[r,"p_{a} \circ t_{u^{-1}}",two heads] \ar[d,hook,"\iota_q"] & (T/T(a))_\C \ar[d,"\eta","\cong"'] \\
\C^\times \times T_\C \ar[r,"p_{q,u}", two heads] & ((\bT \times T)/T(q,u))_\C
\end{tikzcd}
\end{equation}
commutes.
\end{remark}

%If $V$ is a $T/T(a)$-vector bundle on $X^a$, then the pullback of $V$ along $ev_{q,u}$ is equipped with a $(\bT\times T)/T(q,u)$-equivariant structure via $\eta$.

\begin{lemma}\label{equal11}
There is an isomorphism of sheaves of $\O_{T_\C}$-algebras
\[
 t_{u^{-1}}^*\, p_a^*\, \K_{T/T(a)}(X^{a}) \cong \iota_q^* \, p_{q,u}^*\, \K_{(\bT\times T)/T(q,u)}(LX^{q,u})
\]
natural in $X$, which sends an element of the form $[V] \otimes f$ to $[ev_{q,u}^*V] \otimes f$.
\end{lemma}

\begin{proof}
The map $ev_{q,u}$ of Proposition \ref{equall} induces a natural isomorphism of $K_{T/T(a)}$-algebras
\[
K_{T/T(a)}(X^{a}) \cong K_{(\bT\times T)/T(q,u)}(LX^{q,u}),
\]
sending $[V]$ to $[ev_{q,u}^*V]$. The ring $K_{T/T(a)}$ acts on the target via the isomorphism 
\[
K_{(\bT\times T)/T(q,u)} \cong K_{T/T(a)}
\]
induced by $\eta$. Thus, we have an isomorphism of $\O_{T_\C/T(a)_\C}$-algebras   
\[
\K_{T/T(a)}(X^{a}) \cong \eta^* \, \K_{(\bT\times T)/T(q,u)}(LX^{q,u}),
\]
and hence an isomorphism of $\O_{T_\C}$-algebras,
\[
t_{u^{-1}}^*\, p_a^*\, \K_{T/T(a)}(X^{a}) \cong t^*_{u^{-1}}\, p_{a}^*\,\eta^* \, \K_{(\bT\times T)/T(q,u)}(LX^{q,u})
\]
natural in $X$.

We now show that diagram \eqref{daam} commutes, from which the result follows immediately. By the commutativity of diagram \eqref{rectt}, we have
\[
\eta \circ p_a = p_{q,u} \circ \iota_1.
\]
By taking the complexification of the diagram
\[
\begin{tikzcd}
T(q,u) \ar[r] \ar[d] & \bT \times T \ar[d] \\
1 \ar[r] & (\bT \times T)/T(q,u) 
\end{tikzcd}
\]
of compact abelian groups, we see that $(q,u)$ lies in the kernel of 
\[
p_{q,u}: \C^\times \times T_\C \twoheadrightarrow ((\bT \times T)/T(q,u))_\C.
\]
Therefore,
\[
p_{q,u} = p_{q,u} \circ t_{q,u},
\]
and we have
\[
\begin{array}{rcl}
\eta \circ p_a &=& p_{q,u} \circ t_{q,u} \circ \iota_1 \\
&=& p_{q,u} \circ \iota_q \circ t_{u}. \\
\end{array}
\]
Composing on the right by $t_{u^{-1}}$ now yields the commutativity of diagram \eqref{daam}. 
\end{proof}

\begin{remark}\label{daniell}
Let $X$ be a finite $T$-CW complex and let $\U$ be a cover adapted to $\S(X)$. Let $V_u$ be the component of $\psi_{T,q}^{-1}(U_a) \subset T_\C$ containing $u \in \psi_{T,q}^{-1}(a)$, and choose open subsets $V \subset V_u$ and $U \subset U_a$ such that $V \cong U$ via $\psi_{T,q}$. Consider the sheaf $(t_{u^{-1}}^*\, \K_{T}(X^{a}))_{V_u}$, whose value on $V$ is given by
\[
K_T(X^a) \otimes_{K_T} \O_{T_\C}(Vu^{-1}).
\]
The map $\psi_{T,q}$ induces a complex analytic isomorphism $Vu^{-1} \cong Ua^{-1}$ around $1 \in T_\C$, by means of which holomorphic functions on $Vu^{-1}$ and $Ua^{-1}$ are identified. We therefore have a canonical isomorphism of $\O_{U_a}$-algebras
\[
(t_{u^{-1}}^*\, \K_{T}(X^{a}))_{V_u} \cong (t_{a^{-1}}^*\, \K_{T}(X^{a}))_{U_a}
\]
where the latter sheaf takes the value
\[
K_T(X^a)_{Ua^{-1}} := K_T(X^a) \otimes_{K_T} \O_{C_{T,q}}(Ua^{-1}) 
\]
on $U \subset U_a$. 
\end{remark}

\begin{theorem}\label{isoo}
Recall the setup in Remark \ref{daniell}. There is an isomorphism of $\O_{U_a}$-algebras
\[
\F_{T,q}(X)_{U_a} \cong (t_{a^{-1}}^*\, \K_{T}(X^{a}))_{U_a}
\]
natural in $X$.
\end{theorem}

\begin{proof}
We have shown that there exist isomorphisms of $\O_{U_a}$-algebras   
\[
\begin{array}{rcl}
((\psi_{T,q})_*\, \iota_q^*\,  \K_{LT}(LX))_{U_a} &\cong& \prod_{u\in \psi^{-1}_{T,q}(a)} (\iota_q^*\, \K_{\bT\times T}(LX^{q,u}))_{V_u} \\
&\cong& \prod_{u\in \psi^{-1}_{T,q}(a)} (\iota_q^* \, p_{q,u}^*\, \K_{(\bT\times T)/T(q,u)}(LX^{q,u}))_{V_u} \\
&\cong& \prod_{u\in \psi^{-1}_{T,q}(a)}(t_{u^{-1}}^*\, p_a^*\, \K_{T/T(a)}(X^{a}))_{V_u} \\
&\cong& \prod_{u\in \psi^{-1}_{T,q}(a)}(t_{u^{-1}}^*\, \K_{T}(X^{a}))_{V_u} \\
&=& \prod_{u\in \psi^{-1}_{T,q}(a)} (t_{a^{-1}}^*\, \K_{T}(X^{a}))_{U_a}.
\end{array}
\]
The first map is the isomorphism of Corollary \ref{stalks}. The second and fourth maps are the isomorphism of Proposition \ref{Segal}. The third map is the isomorphism of Lemma \ref{equal11}. The final equality follows from Remark \ref{daniell}. Each isomorphism has been shown to preserve the algebra structure, and to be natural in $X$, once we replace the covering $\U(X)$ with a refinement $\U(f)$ adapted to a $T$-equivariant map $f:X \to Y$. \par

We now show that, under this chain of isomorphisms, the action of $m\in \check{T}$ maps the factor of 
\[
\prod_{u\in \psi_{T,q}^{-1}(a)} (t_{a^{-1}}^*\, \K_{T}(X^{a}))_{U_a},
\]
corresponding to $(q,q^mu)$ identically onto the factor corresponding to $(q,u)$. In other words, we will show that the $\check{T}$-action merely permutes the factors, from which the result will follow immediately. \par

On the factor corresponding to $(q,u)$, the second, third and fourth maps become maps of $\O(V_u)$-modules
\[
\begin{array}{rcl}
K_{\bT \times T}(LX^{q,u}) \otimes_{K_{\bT\times T}} \O(V_{q,u})  & \longrightarrow & K_{(\bT \times T)/T(q,u)}(LX^{q,u}) \otimes_{K_{(\bT \times T)/T(q,u)}} \O(V_{u})\\
& \longrightarrow & \, K_{T/T(a)}(X^a) \otimes_{K_{T/T(a)}} \O(u^{-1}V_{u}) \\
& \longrightarrow & \, K_T(X^a) \otimes_{K_T} \O(u^{-1}V_{u}).
\end{array}
\]
We use $\Phi_{q,u}$ to denote this chain of maps. Note that for any two $u,u' \in \psi_{T,q}^{-1}(a)$, we have $u^{-1}U_{u} = (u')^{-1}U_{u'}$, since $u' = q^mu$ for some $m\in \check{T}$, and $U_{u'} = q^mU_{u}$. The result will follow if we can show that 
\begin{equation}\label{idd}
\Phi_{q,u} \circ m^* = \Phi_{q,q^mu}
\end{equation}
is the identity for all $m,u$. First we find an expression for $\Phi_{q,u}$. \par

Recall that a $\bT \times T$-equivariant vector bundle $V$ over $LX^{q,u}$ splits as a direct sum 
\[
V \cong \bigoplus_{\mu \in \hat{T}(q,u)} V_\mu
\]
where $T(q,u)$ acts on the fibers of $V_\mu$ via the character $\mu$. Consider the surjective map $\hat{\bT} \times \hat{T} \twoheadrightarrow \hat{T}(q,u)$ induced by the inclusion $T(q,u) \hookrightarrow \bT \times T$. Let $\mu \in \hat{T}(q,u)$ and choose a preimage $(n,\bar{\mu})\in \hat{\bT} \times \hat{T}$ of $\mu$. It is sufficient to calculate $\Phi_{q,u}$ on an element of the form $[V_\mu] \otimes f$. Recall the change of groups formula of Theorem \ref{Segal1}. The first two maps in the composite map $\Phi_{q,u}$ are 
\[
\begin{array}{rcl}
{[V_\mu]} \otimes f(z)&\mapsto &  [V_\mu \otimes \C_{-(n,\bar{\mu})}] \otimes q^n z^{\bar{\mu}} f(z) \\
&\mapsto & [(ev_{q,u}^{-1})^* (V_\mu \otimes \C_{-(n,\bar{\mu})})] \otimes q^n (uz')^{\bar{\mu}}f(uz') \\
\end{array}
\]
where $z = uz' \in V_u$, so that $z' \in u^{-1}V_u$. Applying the final map in the composite $\Phi_{q,u}$ yields the expression
\begin{equation}\label{phi}
\Phi_{q,u}([V_\mu] \otimes f) =  [(ev_{q,u}^{-1})^* (V_\mu \otimes \C_{-(n,\bar{\mu})}) \otimes \C_{\bar{\mu}}] \otimes q^n u^{\bar{\mu}}  t_u^*f.
\end{equation}
By Theorem \ref{Segal1}, this does not depend on the choice of $(n,\bar{\mu})$. \par

Now that we have an expression for $\Phi_{q,u}$, we will show that equation \eqref{idd} holds using the formulas in Remark \eqref{melbourne}. We have 
\begin{equation}\label{coo}
\begin{array}{rl}
(\Phi_{q,u}\circ m^*)&([V_\mu] \otimes f) = \Phi_{q,u}([m^*V_\mu] \otimes m^*f) \\
&= [(ev_{q,u}^{-1})^* (m^*V_\mu \otimes \C_{-(m)^*(n,\bar{\mu})}) \otimes \C_{\bar{\mu}}] \otimes q^{n}(q^m u)^{\bar{\mu}} t_{u}^*m^*f \\
&= [(ev_{q,u}^{-1})^* (m^*V_\mu \otimes m^*\C_{-(n,\bar{\mu})}) \otimes \C_{\bar{\mu}}] \otimes  q^{n}(q^m u)^{\bar{\mu}} t_{u}^*t_{q^m}^*f \\
&= [(\kappa_m\circ (ev_{q,u}^{-1}))^*(V_\mu \otimes \C_{-(n,\bar{\mu})}) \otimes \C_{\bar{\mu}}]\otimes q^{n}(q^m u)^{\bar{\mu}} t_{q^mu}^*f \\
&= [(ev_{q,q^mu}^{-1})^* (V_\mu \otimes \C_{-(n,\bar{\mu})}) \otimes \C_{\bar{\mu}}] \otimes q^{n} (q^m u)^{\bar{\mu}} t_{q^mu}^*f \\
&= \Phi_{q,q^mu}([V_\mu] \otimes f).
\end{array}
\end{equation}
The fifth equality holds since
\[
ev_{q,q^mu} =  \kappa_m \circ ev_{q,u}^{-1}.
\]
This completes the proof.
\end{proof}

\begin{remark}\label{gertrude}
Our aim is now to determine the gluing maps associated to the local description in Theorem \ref{isoo}. Let $X$ be a finite $T$-CW complex, let $\U$ be a cover adapted to $\S(X)$, and let $a,b \in C_{T,q}$ such that $U_{a} \cap U_{b} \neq \emptyset$.  Choose $u \in \psi_{T,q}^{-1}(a)$ and $v \in \psi_{T,q}^{-1}(b)$ such that $V_u \cap V_v \neq \emptyset$. Therefore $V_{q,u} \cap V_{q,v} \neq \emptyset$, and by Lemma \ref{hardd} we have either $X^b \subset X^a$ or $X^a \subset X^b$. We may assume that $X^b \subset X^a$. Let $U$ be an open subset of $ U_{a} \cap U_{b}$ and let $V \subset V_u \cap V_v$ be such that $V \cong U$ via $\psi_{T,q}$. Let $H$ be the subgroup of $T$ generated by $T(a)$ and $T(b)$. Consider the composite of isomorphisms
\[
\begin{array}{rcl}
K_T(X^a) \otimes_{K_T} \O_{T_\C}(Vu^{-1}) &\cong & K_T(X^b) \otimes_{K_T}  \O_{T_\C}(Vu^{-1}) \\
&\cong& K_{T/H}(X^b) \otimes_{K_T}  \O_{T_\C}(Vu^{-1}) \\ 
&\cong& K_{T/H}(X^b) \otimes_{K_T}  \O_{T_\C}(Vu^{-1}) \\
&\cong& K_T(X^b) \otimes_{K_T}  \O_{T_\C}(Vu^{-1})
\end{array}
\]
The first map is induced by the inclusion $i_{b,a}: X^b \hookrightarrow X^a$, the second and fourth maps are induced by the change of groups map of Proposition \ref{Segal1}, and the third map is  
\[
\id \otimes t^*_{vu^{-1}}: K_{T/H}(X^b) \otimes_{K_{T/H}} \O_{T_\C}(Vu^{-1}) \longrightarrow K_{T/H}(X^b) \otimes_{K_{T/H}} \O_{T_\C}(Vv^{-1})
\]
which is a well defined map of $\O_{T_\C}$-algebras, by Lemma \ref{gartt}. Thus, all maps preserve the $\O_{T_\C}$-algebra structure and it is straightforward to show that the composite sends an element $[V_\lambda] \otimes f$ to
\begin{equation}\label{gluhh}
[i_{b,a}^* V_\lambda] \otimes (vu^{-1})^{\bar{\lambda}}t_{vu^{-1}}^*f.
\end{equation}
We have a commutative diagram
\[
\begin{tikzcd}
Vu^{-1} \ar[r,"{t_{vu^{-1}}}"] \ar[d,"{\psi_{T,q}}"] & Vv^{-1} \ar[d,"{\psi{T,q}}"] \\
Ua^{-1} \ar[r,"{t_{ba^{-1}}}"] & Ub^{-1} 
\end{tikzcd}
\]
of complex analytic isomorphisms. Via this diagram, the composite above is canonically identified with
\[
\begin{array}{rcl}
K_T(X^a) \otimes_{K_T} \O_{C_{T,q}}(Ua^{-1}) &\cong & K_T(X^b) \otimes_{K_T} \O_{C_{T,q}}(Ua^{-1}) \\
&\cong& K_{T/H}(X^b) \otimes_{K_T} \O_{C_{T,q}}(Ua^{-1}) \\ 
&\cong& K_{T/H}(X^b) \otimes_{K_T} \O_{C_{T,q}}(Ub^{-1}) \\
&\cong& K_T(X^b) \otimes_{K_T} \O_{C_{T,q}}(Ub^{-1}),
\end{array}
\]
by analogy with Grojnowski's construction. We will show in Theorem \ref{duh} that this is the gluing map associated to Theorem \ref{isoo}.
\end{remark}

\begin{lemma}\label{gartt}
With the hypotheses of Remark \ref{gertrude}, the translation map
\[
t^*_{vu^{-1}}: \O_{T_\C}(Vu^{-1}) \longrightarrow \O_{T_\C}(V v^{-1})
\]
is $K_{T/H}$-linear.
\end{lemma} 

\begin{proof}
Let $u = \exp(2\pi i (x_1\tau + x_2))$ and $v = \exp(2\pi i (y_1\tau + y_2))$. Since $V_{q,u} \cap V_{q',v} \neq \emptyset$, by the argument in the proof of Lemma \ref{hardd} we have that 
\[
(y_1-x_1, y_2-x_2) \in \mathrm{Lie}(H) \times \mathrm{Lie}(H),
\]
since $T(a),T(b) \subset H$. Therefore
\begin{gather*}
vu^{-1} = \exp(2\pi i (y_1-x_1)\tau + 2\pi i (y_2-x_2)) \\ \in \exp(2\pi i \mathrm{Lie}(H)\tau + 2\pi i \mathrm{Lie}(H)) = \exp(2\pi i \mathrm{Lie}(H)_\C) \subset H_\C.
\end{gather*}
This implies the result.
\end{proof}

\begin{theorem}\label{duh}
With the hypotheses of Remark \ref{gertrude}, the gluing map associated to the local description in Theorem \ref{isoo} on an open subset $U \subset U_{a} \cap U_{b}$ is equal to the composite map in Remark \ref{gertrude}. 
\end{theorem}

\begin{proof}
Consider the commutative diagram
\begin{equation}\label{dia22}
\begin{tikzcd}
LX^{q,v} \ar[r,hook,"{i_{v,u}}"] \ar[d,"{\cong}"',"{ev_{q,v}}"] & LX^{q,u} \ar[d,"{\cong}"',"{ev_{q,u}}"]  \\
X^b \ar[r,hook,"{i_{b,a}}"] & X^a. \\
\end{tikzcd}
\end{equation}
It suffices to calculate the gluing map 
\[
\Phi_{q,v} \circ i_{v,u}^* \circ \Phi_{q,u}^{-1},
\]
on an element of the form 
\[
[V_\lambda] \otimes f \in K_T(X^a) \otimes_{K_T} \O_{T_\C}(Vu^{-1}).
\]
Choose $\bar{\lambda} \in \hat{T}$ restricting to $\lambda \in \hat{T}(a)$. It is straightforward to check, using the formula \eqref{phi} for $\Phi_{q,u}$, that 
\[
\Phi_{q,u}^{-1}([V_\lambda] \otimes f)([ev_{q,u}^* (V_\lambda \otimes \C_{-\bar{\lambda}}) \otimes  \C_{(0,\bar{\lambda})}] \otimes t_{u^{-1}}^*(u^{-\bar{\lambda}}f).
\]
We therefore have
\[
\begin{array}{l}
(\Phi_{q,v} \circ i_{v,u}^* \circ \Phi_{q,u}^{-1})([V_\lambda] \otimes f) =   (\Phi_{q,v} \circ i_{v,u}^*)([ev_{q,u}^* (V_\lambda \otimes \C_{-\bar{\lambda}}) \otimes  \C_{(0,\bar{\lambda})}] \otimes t_{u^{-1}}^*(u^{-\bar{\lambda}}f)) \\
= \Phi_{q,v}( [i_{v,u}^*(ev_{q,u}^* (V_\lambda \otimes \C_{-\bar{\lambda}}) \otimes  \C_{(0,\bar{\lambda})})] \otimes t_{u^{-1}}^*(u^{-\bar{\lambda}}f)) \\
= \Phi_{q,v}( [(ev_{q,u} \circ i_{v,u})^* (V_\lambda \otimes \C_{-\bar{\lambda}}) \otimes  \C_{(0,\bar{\lambda})}] \otimes t_{u^{-1}}^*(u^{-\bar{\lambda}}f)) \\
= [(ev_{q,v}^{-1})^* ((ev_{q,u} \circ i_{v,u})^* (V_\lambda \otimes \C_{-\bar{\lambda}}) \otimes  \C_{(0,\bar{\lambda})} \otimes \C_{-(n,\bar{\lambda})}) \otimes \C_{\bar{\lambda}}] \otimes q^n v^{\bar{\lambda}}  t_v^*t_{u^{-1}}^*(u^{-\bar{\lambda}}f) \\
= [(ev_{q,u} \circ i_{v,u} \circ ev_{q,v}^{-1})^* (V_\lambda \otimes \C_{-\bar{\lambda}}) \otimes  \C_{\bar{\lambda}} ] \otimes v^{\bar{\lambda}}  t_{vu^{-1}}^*(u^{-\bar{\lambda}}f) \\
= [i_{b,a}^* (V_\lambda \otimes \C_{-\bar{\lambda}}) \otimes  \C_{\bar{\lambda}} ] \otimes v^{\bar{\lambda}}  t_{vu^{-1}}^*(u^{-\bar{\lambda}}f) \\
= [i_{b,a}^* V_\lambda] \otimes (vu^{-1})^{\bar{\lambda}}t_{vu^{-1}}^*f
\end{array}
\]
which is equal to \eqref{gluhh}. In the fourth equality, to apply the map $\Phi_{q,v}$, we had to choose an extension of the character of the action of $T(q,v)$ on the fibers of 
\[
(ev_{q,u} \circ i_{v,u})^* (V_\lambda \otimes \C_{-\bar{\lambda}}) \otimes  \C_{(0,\bar{\lambda})}
\]
to $\bT \times T$. Since $T(q,v)$ acts trivially on 
\[
(ev_{q,u} \circ i_{v,u})^* (V_\lambda \otimes \C_{-\bar{\lambda}}),
\]
the extension had to agree with $\bar{\lambda}$ on $T$, and so it took the form $(n,\bar{\lambda})$ for some $n \in \Z$. The fifth equality holds since
\[
[\C_{(0,\bar{\lambda})} \otimes \C_{-(n,\bar{\lambda})}]  \otimes 1 = 1 \otimes q^{-n}.
\]
The final equality holds since $u^{-\bar{\lambda}}$ is a constant, and therefore commutes with $t_{vu^{-1}}^*$. This completes the proof.
\end{proof}

\begin{remark}
Theorems \ref{isoo} and \ref{duh} establish a local description of $\F^*_{T,q}(X)$ only in degree zero. However, the result extends by Remark \ref{harold} and Lemma \ref{gerry} to the entire $\Z/2\Z$-graded $\O_{C_{T,q}}$-algebra.
\end{remark}

\section{The character map}

Recall that $E_\tau := E_{(\tau,1)}$ where $(\tau,1) \in \X^+$, and we write $E_{T,\tau} = \check{T} \otimes E_\tau$. Similarly, we denote Grojnowski's construction of the corresponding elliptic cohomology theory by
\[
\G^*_{T,\tau}(X) := \G^*_{T,(\tau,1)}(X).
\]

\begin{theorem}\label{character}
Let $X$ be an equivariantly formal, finite $T$-CW complex. There is an isomorphism 
\[
\F^*_{T,q}(X) \cong (\sigma_{T,\tau})_* \G^*_{T,\tau}(X)
\]
of $\Z/2\Z$-graded $\O_{C_{T,q}}$-algebras, natural in $X$.
\end{theorem}

\begin{proof}
Choose a cover $\U$ adapted to $\S(X)$, and recall the local description of $\F^*_{T,q}(X)$ from Theorems \ref{isoo} and \ref{duh}. For each $a \in C_{T,q}$ we will define an isomorphism of $\Z/2\Z$-graded $\O_{U_a}$-algebras 
\begin{equation}\label{shirley}
\F^*_{T,q}(X)_{U_a} \xrightarrow{\cong} (\sigma_{T,\tau})_* \G^*_{T,\tau}(X))_{U_a}
\end{equation}
natural in $X$, such that, on each nonempty intersection $U_a \cap U_b$, the diagram 
\[
\begin{tikzcd}
\F^*_{T,q}(X)_{U_a}|_{U_a \cap U_b} \ar[r,"{\cong}"] \ar[d,"{\phi_{b,a}}"] & ((\sigma_{T,\tau})_* \G^*_{T,t}(X))_{U_a}|_{U_a \cap U_b} \ar[d,"{\phi_{b,a}}"]\\
\F^*_{T,q}(X)_{U_b}|_{U_a \cap U_b} \ar[r,"{\cong}"]& ((\sigma_{T,\tau})_* \G^*_{T,t}(X))_{U_b}|_{U_a \cap U_b}
\end{tikzcd}
\]
commutes, where $\phi_{b,a}$ are the respective gluing maps. \par

We must define the natural isomorphism first. Fix $a \in C_{T,q}$. It suffices to define the isomorphism on an arbitrary open subset $U \subset U_a$. By Theorem \ref{isoo}, we have an isomorphism
\begin{equation}\label{betty}
\begin{array}{c}
\F^*_{T,q}(X)_{U_a}(U) \cong K^*_{T}(X^{a}) \otimes_{K_T} \O_{C_{T,q}}(Ua^{-1}) 
\end{array}
\end{equation}
natural in $X$. By definition, we have 
\begin{equation}\label{sue}
\begin{array}{c}
((\sigma_{T,\tau})_* \G^*_{T,\tau}(X))_{U_a}(U) =  H^*_T(X^a) \otimes_{H_T} \O_{E_{T,\tau}}(U - a)
\end{array}
\end{equation}
where, by a severe abuse of notation, we have written $U$ again for $\sigma_{T,\tau}^{-1}(U)$ and $a$ for $\sigma_{T,\tau}^{-1}(a)$, to keep things consistent. We hope that the difference between the additive and multiplicative notation will prevent any confusion arising from this. Since $X$ is equivariantly formal and $U_a$ is small, we may apply the pushforward of the equivariant Chern character over $\psi_{T,q}|_V$, where $V$ is a small neighbourhood of $1$ such that $U_aa^{-1} \subset \psi_{T,q}(V)$. This yields an isomorphism
\begin{equation}\label{marcia}
ch_T: K^*_T(X^a) \otimes_{K_T} \O_{C_{T,q}}(Ua^{-1}) \cong H_T(X^a) \otimes_{H_T} \O_{E_{T,\tau}}(U-a).
\end{equation}
Composing \eqref{betty}, \eqref{sue} and \eqref{marcia} in the obvious way produces the isomorphism \eqref{shirley} over $U \subset U_a$. \par

It remains to show that the local isomorphisms thus defined are compatible with gluing maps. In other words, we must show that the diagrams
\[
\begin{tikzcd}
K_T(X^a)_{Ua^{-1}} \ar[r] \ar[d,"{ch_T}"] & K_T(X^b)_{Ua^{-1}} \ar[r] \ar[d,"{ch_T}"] & K_{T/H}(X^b)_{Ua^{-1}} \ar[d,"{ch_{T/H}}"] \\
H_T(X^a)_{U - a} \ar[r] & H_T(X^b)_{U-a} \ar[r] & H_{T/H}(X^b)_{U-a}.
\end{tikzcd}
\]
and
\[
\begin{tikzcd}
K_{T/H}(X^b)_{Ua^{-1}} \ar[r] \ar[d,"{ch_{T/H}}"] & K_{T/H}(X^b)_{Ub^{-1}} \ar[r] \ar[d,"{ch_{T/H}}"] & K_T(X^b)_{Ub^{-1}} \ar[d,"{ch_T}"] \\
H_{T/H}(X^b)_{U - a} \ar[r] & H_{T/H}(X^b)_{U - b} \ar[r]& H_T(X^b)_{U - b} \\
\end{tikzcd}
\]
commute, where $H = \langle T(a),T(b) \rangle$. The left hand square in the upper diagram commutes since $ch_T$ is natural in $X$. The left hand square in the lower diagram commutes since $t_{ba^{-1}}^*$ is $K_{T/H}$-linear and $t_{b-a}^*$ is $H_{T/H}$-linear. The other two squares commute because the Chern character respects the change of groups maps in Propositions \ref{changeh} and \ref{Segal1}. Indeed, recalling the maps $h$ and $j$ of the pullback diagram \eqref{pullback} and the definition of the Chern character, we have  
\[
\begin{array}{rcl}
h^*ch_{T/H}([V_\lambda \otimes \C_{-\bar{\lambda}}]) \otimes j^*ch_T(\C_{\bar{\lambda}}) &=& ch_T([V_\lambda \otimes \C_{-\bar{\lambda}}]) \otimes ch_T(\C_{\bar{\lambda}}) \\
&=& ch_T([V_\lambda]) \cup ch_T([\C_{-\bar{\lambda}}]) \otimes ch_T(\C_{\bar{\lambda}})  \\
&=& ch_T([V_\lambda]) \otimes e^{-\bar{\lambda}} \otimes e^{\bar{\lambda}}  \\
&=& ch_T ([V_\lambda]) \otimes 1
\end{array}
\]
where $\lambda \in \hat{H}$ and $\bar{\lambda}$ is an extension of $\lambda $ to a character of $T$. This completes the proof.

\end{proof}

\begin{corollary}\label{gluten}
There is an isomorphism of cohomology theories 
\[
\F^*_{T,q} \cong (\sigma_{T,\tau})_* \G^*_{T,\tau}
\]
defined on equivariantly formal, finite $T$-CW complexes and taking values in sheaves of $\Z/2\Z$-graded $\O_{C_{T,q}}$-algebras.
\end{corollary}

\begin{proof}
The natural isomorphism of Theorem \ref{character} is an isomorphism of cohomology theories since, by the proof of Proposition \ref{gerry}, the suspension isomorphisms of $\F^*_T(X)$ are inherited from equivariant K-theory, and these are compatible with the suspension maps of Grojnowski's theory via the Chern character, which is an isomorphism of cohomology theories. 
\end{proof}

\chapter{Kitchloo's elliptic cohomology theory}\label{twist}

In this chapter, we construct a holomorphic sheaf $^k\K_{\widetilde{LT}}(LX)$ in a way that is analogous to Kitchloo's construction in \cite{Kitch1}, only exchanging the simple, simply-connected compact Lie group $G$ for a torus $T$. This modification does not present any difficulties, and the construction goes through almost unchanged. One difference is that we construct a sheaf over $\D^\times \times T_\C$, while Kitchloo constructs his over $\bH \times \t_\C$, but this is not significant as the two spaces are related by the complex exponential map, and Kitchloo's $T$-equivariant construction over $\bH \times \t_\C$ is obtained from ours by pulling back over this map. \par

In the previous chapter, we considered the action of a loop group $LT$ on a space $LX$, and used equivariant K-theory to construct a sheaf $\K_{LT}(LX)$ out of this. However, it was not constructed in such a way that a cocycle is represented by a family of $LT$-representations over $LX$, as it turns out that there are no honest, nontrivial representations of $LT$. To bring any kind of representations of $LT$ into the picture, we really need to consider representations of a central extension $\widetilde{LT}$ of $LT$ by a circle, which are equivalently projective representations of $LT$. \par

It is helpful in this context to consider why there are no nontrivial honest representations of $LT$. Supposing that there were, then we should expect the character of such a representation to induce a $\check{T}$-invariant holomorphic function on $\D^\times \times T_\C$, since this is a subspace of the complexification of the maximal torus $\bT \times T$. However, this would be equivalent to a holomorphic function on the quotient $C_T$, which would then have to be constant since the fiber $C_{T,q} \subset C_T$ is a projective variety. Therefore, the character of the honest representation would have to be trivial. On the other hand, if we consider a level $k$ representation of $\widetilde{LT}$, then we find that the character induces a nontrivial section of a certain line bundle $\L^k$ over $C_T$, under certain conditions. \par

Evaluated on a point, the value of the $\check{T}$-invariants of the pushforward $^k\F_T(X)$ of $^k\K_{\widetilde{LT}}(LX)$ turns out to be the line bundle $\L^k$. For more general $X$, cocycles in $^k\K_{\widetilde{LT}}(LX)$ are represented by $\bT \times T$-equivariant families of Fredholm operators, which are parametrised over finite subcomplexes of $LX$ and which carry a natural action of the Weyl group $\check{T}$. We use the equivariant index theorem of Atiyah and Segal to relate $^k\K_{\widetilde{LT}}(LX)$ to the sheaf $\K_{LT}(LX)$ of the previous chapter, and it follows from this relationship that $^k\F_T(X)$ is isomorphic to $\F_T(X) \otimes \L^k$. Finally, we apply the character map of Theorem \ref{character} to show that $k\F_{T,q}(X)$ is isomorphic to the twisted theory $\G_{T,\tau}(X) \otimes \L^k$ of Grojnowski, for $X$ equivariantly formal and where $q = e^{2\pi i\tau}$.

\begin{notation}
In this chapter, as in the previous chapter, we identify the circle $\bT$ with $S^1 \subset \C$. 
\end{notation}

\section{Representation theory of the loop group $LT$}

In this Our reference for this section is Chapters 4 and 9 of \cite{PS}.

\begin{notation}
We fix a symmetric bilinear form 
\[
I: \check{T} \times \check{T} \longrightarrow \Z.
\]
Recall that, because $T$ is a torus, the perfect pairing between $\check{T}$ and $\hat{T}$ induces a canonical isomorphism $\hat{T} \cong \Hom(\check{T},\Z)$. The adjoint to $I$ therefore takes the form
\[
I: \check{T} \longrightarrow \hat{T}.
\]
Let $\phi(m) := \frac{1}{2}I(m,m)$ be the quadratic form associated to $I$.
\end{notation}

\begin{remark}
The bilinear form $I$ determines a cocycle $\omega$ on the Lie algebra $L\t := C^\infty(\bT;\t)$ of $LT$\footnote{The loop group $LT$ is actually an infinite dimensional Lie group. For details, we refer the reader to sections 3.1 and 3.2 in \cite{PS}}, given by
\[
\omega(\alpha,\beta) := \frac{1}{2\pi} \int_0^{2\pi} I(\alpha'(\theta),\beta(\theta)) d\theta.
\]
This, in turn, determines a central extension $\widetilde{L\t} = L\t \oplus \R$ with bracket given by
\[
[(\alpha,x),(\beta,y)] = ([\alpha,\beta],\omega(\alpha,\beta)).
\]
central extension  
\begin{equation}\label{ext3}
1\longrightarrow U(1) \longrightarrow \widetilde{LT} \longrightarrow LT \longrightarrow 1
\end{equation}
\end{remark}

\begin{remark}\label{convent}
The action of $\bT$ lifts to an action on $\widetilde{LT}$,\footnote{See Chapter 4 of \cite{PS}} and we are thus able to form the semidirect product
\[
\bT \ltimes \widetilde{LT}.
\]
All representations of $\bT \ltimes \tilde{LT}$ and $\widetilde{LT}$ that we consider in this chapter are unitary representations on a Hilbert space, although it may not be explicitly stated. Two representations $E$ and $E'$ are said to be \textit{essentially equivalent} if there is a continuous, equivariant linear map $E \to E'$ which is injective and has dense image. 
\end{remark}

\begin{definition}
A representation $V$ is called \textit{irreducible} if it has no closed invariant subspace. 
\end{definition}

\begin{remark}
The extension \eqref{ext3} splits over the constant loops, and therefore $T$ is a subgroup of $\widetilde{LT}$. If $T$ is a maximal torus in $T$, then 
\[
\bT \times T \times U(1) \subset \bT \ltimes \widetilde{LT}
\]
is a maximal torus. Up to essential equivalence, a representation $V$ of $\bT \ltimes \widetilde{LT}$ decomposes as a direct sum
\begin{equation}\label{decom}
\bigoplus_{(n,\lambda,k)} V_{(n,\lambda,k)},
\end{equation}
of subspaces, where $\bT \times T \times U(1)$ acts on $V_{(n,\lambda,k)}$ by the character $(n,\lambda,k) \in \hat{\bT} \times \hat{T} \times \hat{U(1)}$. A character occurring in the decomposition of $V$ is called a \textit{weight of $V$}.
\end{remark}

\begin{remark}
The parametrisation of $\bT$ induces an identification of $\hat{\bT}$ with $\Z$. For a representation $V$ with decomposition \eqref{decom}, if $V_{(n,\lambda,k)}\neq 0 $ for only finitely many values of $n < 0$, then $V$ is said to be a \textit{positive energy representation}. Henceforth, any representation of $\bT \ltimes \widetilde{LT}$ that we consider will be a positive energy representation, although it may not be explicitly stated.
\end{remark} 

\begin{theorem}[Theorem 9.3.1, \cite{PS}]\label{flo}
Any representation $V$ of $\bT \ltimes \widetilde{LT}$ contains a direct sum of irreducible representations as a dense subspace. Any representation of $\widetilde{LT}$ extends to a representation of $\bT \ltimes \widetilde{LT}$.
\end{theorem}

\begin{remark}
If $V$ is an irreducible representation of $\bT \ltimes \widetilde{LT}$, then, by Schur's lemma, the central circle $U(1)$ acts by scalars, and therefore only one value of $k$ can occur in the decomposition \eqref{decom} of $V$. This is called the \textit{level} of $V$.
\end{remark}

\begin{proposition}[Proposition 9.2.3, \cite{PS}]\label{pos}
The restriction of a positive energy, irreducible representation of $\bT \ltimes \widetilde{LT}$ to $\widetilde{LT}$ is also irreducible.
\end{proposition}

%\begin{proposition}[Proposition 9.3.4, \cite{PS}]
%Up to essential equivalence, a representation of $\bT \ltimes \widetilde{LT}$ has a decomposition
%\[
%V = \bigoplus_{n \in \hat{\bT}} V(n)
%\]
%according to the action of $\bT$. If $V$ is irreducible, then $V(n)$ is finite dimensional for all $n$.
%\end{proposition}

\begin{remark}
The affine Weyl group associated to $\bT \times T \times U(1) \subset \bT \ltimes \widetilde{LT}$ is defined as the quotient
\[
W_{aff} := N_{\bT \ltimes \widetilde{LT}}(\bT \times \tilde{T})/(\bT \times \tilde{T}).
\]
Since the extension by $U(1)$ is central, we have
\[
W_{aff} = \check{T}
\]
and the action of $\check{T}$ on $\bT \times \tilde{T}$ covers the action on $\bT \times T$. 
\end{remark}

\begin{proposition}[\cite{PS}, Proposition 4.9.4]
The Weyl action of $m \in \check{T}$ on $\bT \times T \times U(1)$ is given by
\begin{equation}\label{action1}
(q,u,z) \longmapsto (q,q^mu, q^{\phi(m)}u^{I(m)}z)
\end{equation}
\end{proposition}

\begin{proposition}\label{formula}
The Weyl action of $m \in \check{T}$ on a character $(n,\lambda,k)$ is given by
\begin{equation}\label{action}
m \cdot (n,\lambda,k) = (n + \lambda(m) + k \phi(m), \lambda + k I(m), k).
\end{equation}
\end{proposition}

\begin{proof}
An element $m \in \check{T}$ acts on a character $(n,\lambda,k)$ by pullback along the action of $m$, so that the diagram
\[
\begin{tikzcd}
\bT \times \tilde{T} \ar[rr,"{m\cdot (n,\lambda,k)}"] \ar[d,"m"]& &\C^\times \ar[d,"="] \\
\bT \times \tilde{T} \ar[rr, "{(n,\lambda,k)}"] & &\C^\times
\end{tikzcd}
\]
commutes. We therefore have
\[
\begin{array}{rcl}
(m\cdot (n,\lambda,k)) (q,u,z) &=& q^n(q^{m}u)^{\lambda}(q^{\phi(m)}u^{I(m)}z)^k \\
&=& q^{n+\lambda(m) + k\phi(m)}u^{\lambda + kI(m)}z^k.
\end{array}
\]
\end{proof}

\begin{theorem}[Theorem 9.3.5, \cite{PS}]\label{class}
If $k$ is a positive integer and $I$ is positive-definite, then there is a bijection between the set of irreducible level $k$ representations of $\widetilde{LT}$ up to essential equivalence, and the set 
\[
\hat{T}/k\check{T},
\]
where $\check{T}$ is regarded as a subgroup of $\hat{T}$ via the injection
\[
I: \check{T} \longrightarrow \hat{T}.
\]
\end{theorem}

\begin{remark}
Theorem \ref{class} enables us to classify all irreducible positive energy representations of $\bT \ltimes \widetilde{LT}$ of level $k$. By Proposition \ref{pos}, an irreducible representation of $\bT \ltimes \widetilde{LT}$ restricts to an irreducible representation of $\widetilde{LT}$, and any two representations which restrict to the same representation of $\widetilde{LT}$ can only differ by multiplication by a character of $\bT$. Furthermore, a positive energy representation of $\bT \ltimes \widetilde{LT}$ has only finitely many negative characters $n \in \hat{\bT}$ occurring in the decomposition \eqref{decom}, and the lowest such $n$ is called its \textit{lowest energy}. Also, by Theorem \ref{flo}, any irreducible representation of $\widetilde{LT}$ extends to an irreducible representation of $\bT \ltimes \widetilde{LT}$. Therefore, up to essential equivalence, irreducible level $k$ representations of $\bT \ltimes \widetilde{LT}$ are classified by elements $(n,[\lambda]) \in \Z \times \hat{T}/k\check{T}$.
\end{remark}

\section{Fredholm operators and the equivariant index map}

\begin{definition}
For an integer $k > 0$, define $\H_k$ to be the Hilbert space completion of the direct sum 
\[
\bigoplus_{\mathbb{N}} \bigoplus_{n,[\lambda]} \H_{n,[\lambda],k},
\]
where $n$ ranges over $\Z$, $[\lambda]$ ranges over $\hat{T}/k\check{T}$, and $\H_{n,[\lambda],k}$ denotes the irreducible level $k$ representation of $\bT \ltimes \widetilde{LT}$ corresponding to $[\lambda] \in \hat{T}/k\check{T}$ and with lowest energy $n$. Thus, any level $k$ representation of $\bT \ltimes \widetilde{LT}$ may be identified with a summand of $\H_k$, up to essential equivalence.
\end{definition}

\begin{definition}
A \textit{Fredholm operator} on a Hilbert space is a bounded linear operator with finite dimensional kernel and cokernel. We denote by $\F_k$ the space of all Fredholm operators on the Hilbert space $\H_k$ with the norm topology. The action of $\bT\ltimes \widetilde{LT}$ on $\H_k$ induces an action on $\F_k$ by conjugation. Note that the identity operator is a Fredholm operator.
\end{definition}

%The central circle acts trivially on $\F_k$, however the integer $k$ appears in the action of loops on $\F_k$. We will just write $\F$ for $\F_k$ when we wish to remember only the $\bT \times T$-action on $\F_k$. Similarly, we write $\H$ for the restriction of $\H_k$ to a $\bT \times T$-representation.

\begin{definition}\label{Fredholm}
Let $Y$ be a $\bT \times T$-CW complex. Define the group
\[
^kK_{\bT \times T \times U(1)}(Y) :=   \pi_0 \Map_{\bT \times T \times U(1)}(Y,\F_k)
\]
with identity element the map sending $Y$ to the identity operator, and group operation induced by composition of Fredholm operators.
\end{definition}

\begin{remark}\label{act}
The loop group $\bT \ltimes \widetilde{LT}$ acts on $\Map(LX,\F_k)$ as follows. An element $g \in \bT \ltimes \widetilde{LT}$ sends $F \in \Map(LX,\F_k)$ to the unique map $F^g$ such that 
\[
\begin{tikzcd}
LX \ar[d,"g"] \ar[r,"F^g"] & \F_k \ar[d,"g"] \\
LX \ar[r,"F"] & \F_k \\
\end{tikzcd}
\]
commutes. Thus, writing the map $F$ as $\gamma \mapsto F_\gamma$, we have the formula
\begin{equation}
F^g_{\gamma} = \theta_{g^{-1}}\circ F_{g\cdot \gamma} \circ \theta_g,
\end{equation}
where $\theta_g$ is the image of $g$ in $\Aut(\H_k)$. The subset 
\[
\Map_{\bT \times T \times U(1)}(LX,\F_k) \subset \Map(LX,\F_k)
\]
of $\bT \times T \times U(1)$-equivariant maps is the $\bT \times T \times U(1)$-fixed point set of this action, and therefore carries an action of the affine Weyl group $\check{T}$. More generally, let $Y$ be a $\bT \times T$-equivariant subspace of $LX$. The Weyl action of $m \in \check{T}$, since it is continuous, induces a group homomorphism 
\begin{equation}\label{girl}
m^*: \, ^kK^{*}_{\bT \times T \times U(1)}(m\cdot Y) \longrightarrow  \, ^kK^{*}_{\bT \times T \times U(1)}(Y)
\end{equation}
sending $[F]$ to $[F^m]$.
\end{remark}

\begin{definition}\label{gall}
Let $^kK_{\bT \times T \times U(1)}$ denote the free abelian group generated by the irreducible level $k$ characters of $\bT \times T \times U(1)$, which is evidently a module over the ring $K_{\bT \times T}$ of characters of $\bT \times T$. 
\end{definition}

\begin{remark}\label{galt}
The action of $m \in \check{T}$ induces a group homomorphism 
\[
m^*: \, ^kK_{\bT \times T \times U(1)} \longrightarrow \, ^kK_{\bT \times T \times U(1)}
\]
according to the formula
\begin{equation}
(n,\lambda,k) \mapsto (n + \lambda(m) + k \phi(m), \lambda + k I(m), k)
\end{equation}
of Proposition \ref{formula}.
\end{remark}

\begin{remark}\label{eqind}
We introduce the equivariant index map of Atiyah and Segal. Let $\H$ be a Hilbert space representation of $\bT \times T$ in which each irreducible representation of $\bT \times T$ occurs countably many times. Let $\F$ be the space of Fredholm operators on $\H$ with the norm topology, on which $\bT \times T$ acts by conjugation. Let $Y$ be a finite $\bT \times T$-CW complex and let $F: Y \to \F$, be a $\bT \times T$-equivariant, continuous map. By the results in Appendix 2 and 3 of \cite{Atiyah}, there exists a $\bT \times T$-equivariant subspace $\H' \subset \H$ of finite codimension such that, if $p$ denotes the orthogonal projection $\H \twoheadrightarrow \H'$, then $p \circ F_y$ is surjective for all $y \in Y$. Nonequivariantly, the index of a Fredholm operator $F$ is defined to be the difference between the dimensions of its kernel and cokernel, and it is a well known property that, for a continuous family $\{F_y\}_{y\in Y}$ of Fredholm operators, the index of $F_y$ remains constant as $y$ varies. Since $p \circ F_y$ is surjective for all $y$, the dimension of $\coker(p \circ F_y)$ remains constant (i.e. zero) for all $y$, which means the dimension of $\ker(p \circ F_y)$ remains constant. The bundle of vector spaces 
\[
\{\ker(p \circ F_y)\}_{y\in Y}
\]
is a finite dimensional $\bT \times T$-vector bundle on $Y$. By Corollary A3.2 in \cite{Atiyah}, the map 
\[
\begin{array}{rcl}
\pi_0 \Map_{\bT \times T}(Y,\F) &\longrightarrow& K_{\bT \times T}(Y) \\
{[F]} &\longmapsto & [\{\ker(p \circ F_y)\}_{y\in Y}] - [Y \times \ker(p)]
\end{array}
\]
is an isomorphism of groups, natural in $Y$, and does not depend on the choice of $p$. 
\end{remark}

We should think of the following Proposition as an application of the equivariant index map, combined with the change of groups isomorphism of Proposition \ref{Segal1}.

\begin{proposition}\label{index}
Let $Y$ be a finite $\bT \times T$-CW complex. There is an isomorphism of $K_{\bT \times T}$-modules
\[
^kK_{\bT \times T \times U(1)}(Y) \cong K_{\bT \times T}(Y) \otimes_{K_{\bT \times T}} \, ^kK_{\bT \times T \times U(1)}
\]
natural in $Y$. In particular, 
\[
^kK_{\bT \times T \times U(1)}(\pt) \cong \,  ^kK_{\bT \times T \times U(1)}.
\]
\end{proposition}

\begin{proof}
First, we show that the map
\begin{equation}\label{bell}
\Map_{\bT \times T \times U(1)}(Y,\F_k) \longrightarrow \Map_{\bT \times T}(Y,\F_k)
\end{equation}
given by forgetting the $U(1)$-action is an isomorphism. Indeed, an element $s \in U(1)$ acts on $\H_k$ by multiplication by $s^k$, and so the induced action of $U(1)$ on $\F_k$ by conjugation fixes $\F_k$, by the linearity of Fredholm operators. Therefore, any $\bT \times T$-equivariant map $Y \to \F_k$ is in fact $\bT \times T \times U(1)$-equivariant, since $U(1)$ acts trivially on both $Y$ and $\F_k$, and so \eqref{bell} has an inverse. Since the group operation induced by composition of Fredholm operators is clearly preserved, \eqref{bell} is an isomorphism of groups. 

The statement of the proposition now follows by applying the equivariant index map of Remark \ref{eqind}. By Theorem \ref{class}, $\H_k$ contains countably many copies of each irreducible representation of $\bT \times T$. Therefore,
\[
\pi_0 \Map_{\bT \times T}(Y,\F_k) \cong K_{\bT \times T}(Y).
\]
Combined with the map \eqref{bell}, this gives us an isomorphism of groups
\[
^kK_{\bT \times T \times U(1)}(Y) \cong K_{\bT \times T}(Y)
\]
which is natural in $Y$, since the equivariant index map is natural in $Y$.

Since "forgetting" the $U(1)$-action, given by $k \in \hat{U}(1)$, corresponds to tensoring with $\C_{-(0,0,k)}$, the isomorphism of the proposition is given by
\[
[F] \mapsto [\{\ker(p \circ F_y\}_{y \in Y} \otimes [\C_{(0,0,k)}] - [(Y \times \ker(p)) \otimes \C_{-(0,0,k)}] \otimes [\C_{(0,0,k)}].
\]
\end{proof}

\begin{definition}
Let $Y$ be a $\bT \times T$-CW complex. Let $^k\K^{*}_{\bT \times T \times U(1)}(Y)$ denote the sheaf of $\O_{\C^\times \times T_\C}$-modules which takes the value 
\[
^kK_{\bT \times T \times U(1)}(Y) \otimes_{K_{\bT \times T}} \O_{\C^\times \times T_\C}(U)
\]
on an open subset $U \subset \C^\times \times T_\C$. Define the sheaf of $\O_{\D^\times \times T_\C}$-modules
\[
^k\K_{\widetilde{LT}}(LX):= \varprojlim_{Y \subset LX \: \text{finite}} \, ^k\K_{\bT \times T \times U(1)}(Y)_{\D^\times \times T_\C},
\]
where the inverse limit runs over all finite $\bT \times T$-CW subcomplexes $Y \subset LX$.
\end{definition}

\begin{remark}
We define a $\check{T}$-action on $^k\K_{\widetilde{LT}}(LX)$ in a way that is analogous to Remark \ref{melbourne}. Recall that $\omega_m: \D^\times \times T_\C \rightarrow \D^\times \times T_\C$ denotes the action map $(q,u) \mapsto (q,q^mu)$ of $m$. The map $m^*$ of \eqref{girl} induces an isomorphism of sheaves
\[
\omega_m^* \, ^k\K_{\bT \times T \times U(1)}(m\cdot Y)_{\D^\times \times T_\C} \longrightarrow \, ^k\K_{\bT \times T\times U(1)}(Y)_{\D^\times \times T_\C}
\]
for $Y \subset LX$ a finite $\bT \times T$-CW complex. The family of maps thus produced is compatible with inclusions of subsets of $LX$, since it is induced by the action of $\bT \ltimes \widetilde{LT}$. We therefore have an isomorphism of inverse limits
\[
\omega_m^* \, ^k\K_{\widetilde{LT}}(LX) \longrightarrow \, ^k\K_{\widetilde{LT}}(LX).
\]
For each $(q,u)$, this induces an isomorphism, also denoted $m^*$, 
\begin{equation}
\begin{array}{rcl}
m^*: \, ^k\K_{\widetilde{LT}}(LX)_{(q,q^mu)} &\longrightarrow& \, ^k\K_{\widetilde{LT}}(LX)_{(q,u)} \\
{[F]} \otimes g &\longmapsto & [F^m] \otimes m^*g
\end{array}
\end{equation}
from the stalk at $(q,q^mu)$ to the stalk at $(q,u)$. Since it is induced by a group action, the collection of isomorphisms
\[
\{\omega_m^*\, ^k\K_{\widetilde{LT}}(LX) \longrightarrow \, ^k\K_{\widetilde{LT}}(LX)\}_{m \in \check{T}}
\] 
defines an action of $\check{T}$ on $^k\K_{\widetilde{LT}}(LX)$. 
\end{remark}
%\begin{proposition}
%The set $K_{\bT \times T}(Y)$ has a ring structure where, for Fredholm operators $A$ and $B$, addition $A + B$ is induced by
%\[
%A \oplus B
%\]
%and multiplication $A \cdot B$ is induced by
%\begin{equation}
%\left(  \begin{array}{cc}
%A \otimes I & -I \otimes B^* \\
%I \otimes B & A^* \otimes I 
%\end{array} \right)
%\end{equation}
%where $A^*$ is the adjoint of $A$ and we have used identifications $\H \oplus \H \cong \H$ and $\H \otimes \H \cong \H$.
%\end{proposition}

%\begin{proof}
%There is a proof of the non-equivariant case in \cite{Janich}. It remains to show that these operations are compatible with conjugation by $T \in \bT \ltimes \widetilde{LT}$, which is straightforward.
%\end{proof}

\begin{notation}
We use $^k\O_{\D^\times \times T_\C}$ to denote the sheaf $^k\K_{\widetilde{LT}}(L*)$. The value of $^k\O_{\D^\times \times T_\C}$ on open subset $U \subset \D^\times \times T_\C$ is
\[
^kK_{\bT \times T \times U(1)} \otimes_{K_{\bT \times T}} \O_{\D^\times \times T_\C}(U).
\]
\end{notation}

\begin{definition}
Define the $\O_{C_T}$-module 
\[
^k\F_T(X) := ((\psi_T)_*\, ^k\K_{\widetilde{LT}}(LX))^{\check{T}}.
\]
Define the \textit{Looijenga line bundle} to be
\[
\L^k := \, ^k\F_T(\pt).
\]
By definition, a section of $\L^k $ is a holomorphic function $s$ on $\D^\times \times T_\C$ which satisfies the transformation property
\[
s(q,u) = s(q,q^mu)u^{I(m)}q^{\phi(m)} 
\]
for all $m \in \check{T}$. Such functions are called \textit{theta functions of degree $k$}.
\end{definition}

\begin{remark}\label{calc}
We make some further remarks concerning the equivariant index map, and its relationship to the action of $\check{T}$. Let $p$ and $F_y$ be as in Remark \ref{eqind}, and let $p^m = \theta_{-m} \circ p \circ \theta_{m}$. Clearly, $p \circ F_y$ is surjective for all $y$ if and only if 
\[
p^m \circ F^m_y = \theta_{-m} \circ p \circ F_y \circ \theta_{m}
\]
is surjective for all $y$. Furthermore, $p^m$ is clearly $\bT \times T$-equivariant whenever $p$ is. Therefore, the map 
\[
\{\ker(p^m \circ F^m_y)\}_{y\in Y} 
\]
is a $\bT \times T$-vector bundle on $Y$, as is $Y \times \ker(p^m)$. In Remark \eqref{melbourne}, given a $\bT \times T$-vector bundle $V$ over $m\cdot Y$, there was only one possible way to define a $\bT \times T$-vector bundle $m^*[V]$ over $Y$. We must have
\[
m^* [\{\ker(p\circ F_{y}) \}_{y\in m\cdot Y}] = [\{\ker(p^m \circ F^m_{y}) \}_{y\in Y}].
\] 
and 
\[
m^* [(m\cdot Y) \times \ker(p)] = [(Y \times \ker(p^m))].
\] 
These two formulas will be essential to the proof of the following proposition.
\end{remark}

\begin{proposition}\label{obg}
There is a $\check{T}$-equivariant isomorphism of $\O_{\D^\times \times T_\C}$-modules
\[
^k\K_{\widetilde{LT}}(LX) \cong  \K_{LT}(LX) \otimes_{\O_{\D^\times \times T_\C}} \, ^k\O_{\D^\times \times T_\C}
\]
natural in $X$.
\end{proposition}

\begin{proof}
On an open set $U \subset \D^\times \times T_\C$, there is an isomorphism of $\O_{\D^\times \times T_\C}(U)$-modules
\[
\begin{array}{rcl}
^k\K_{\widetilde{LT}}(LX)(U) &:=& \varprojlim_{Y \subset LX} \, ^kK_{\bT \times T \times U(1)}(Y) \otimes_{K_{\bT \times T}} \O_{\D^\times \times T_\C}(U) \\
&\cong & \varprojlim_{Y \subset LX} \, K_{\bT \times T}(Y) \otimes_{K_{\bT \times T}} \, ^kK_{\bT \times T \times U(1)} \otimes_{K_{\bT \times T}} \O_{\D^\times \times T_\C}(U) \\
&=:& \varprojlim_{Y \subset LX} \, K_{\bT \times T}(Y) \otimes_{K_{\bT \times T}} \, ^k\O_{\D^\times \times T_\C}(U) \\
&= &  \K_{LT}(LX)(U) \otimes_{\O_{\D^\times \times T_\C}(U)} \, ^k\O_{\D^\times \times T_\C}(U)
\end{array}
\]
induced by the natural isomorphism of Proposition \ref{index}. To check $\check{T}$-equivariance, it suffices to show that the diagram
\[
\begin{tikzcd}
^kK_{\bT \times T \times U(1)}(m\cdot Y) \ar[d] \ar[r,"{\cong}"] & \, K_{\bT \times T}(m\cdot Y) \otimes_{K_{\bT \times T}} \, ^kK_{\bT \times T \times U(1)} \ar[d] \\
^kK_{\bT \times T \times U(1)}(Y) \ar[r,"{\cong}"] & \, K_{\bT \times T}(Y) \otimes_{K_{\bT \times T}} \, ^kK_{\bT \times T \times U(1)}
\end{tikzcd}
\]
commutes for a finite $\bT \times T$-CW subcomplex $Y \subset LX$. Here the horizontal arrows are the isomorphism of Proposition \ref{index}, the left vertical arrow is the isomorphism defined in \eqref{girl}, and the right vertical arrow is the group homomorphism determined by
\[
[V] \otimes \C_{(n,\lambda,k)} \longmapsto m^*[V] \otimes m^*\C_{(n,\lambda,k)},
\]
with $m^*[V]$ defined as in Remark \eqref{melbourne} and $m^*\C_{(n,\lambda,k)}$ defined as in Remark \ref{galt}. \par

Let $[F]$ be an element in $^kK_{\bT \times T \times U(1)}(m\cdot Y)$ represented by 
\[
F \in \Map_{\bT \times T \times U(1)}(m\cdot Y,\F_k).
\]
Via the upper horizontal arrow, $[F]$ is sent to
\[
[\{\ker(p \circ F_{y})\}_{ y\in m\cdot Y} \otimes \C_{-(0,0,k)}] \otimes [\C_{(0,0,k)}] - [ (m\cdot Y \times \ker(p))  \otimes \C_{-(0,0,k)}] \otimes [\C_{(0,0,k)}].
\]
The right vertical arrow sends this to
\begin{equation}\label{stove}
\begin{array}{c}
m^*[ \{\ker(p\circ F_{y})\}_{y\in m\cdot Y} \otimes \C_{-(0,0,k)}] \otimes m^*[\C_{(0,0,k)}] \\ - m^*[(m\cdot Y \times \ker(p))  \otimes \C_{-(0,0,k)}] \otimes m^*[\C_{(0,0,k)}].
\end{array}
\end{equation}
Using the formula in Remark \ref{calc}, we can rewrite \eqref{stove} as
\begin{gather*}
m^* [ \{\ker(p\circ F_{m\cdot y}\}_{y\in m\cdot Y} \otimes \C_{-(0,0,k)}] \otimes m^*[\C_{(0,0,k)}] \\ - [(Y \times \ker(p^m))  \otimes \C_{-(k \phi(m), k I(m),k)}] \otimes [\C_{(k \phi(m), k I(m),k)}]
\end{gather*}
which is equal to
\begin{equation}\label{grout}
m^* [ \{\ker(p\circ F_{m\cdot y}\}_{y\in m\cdot Y} \otimes \C_{-(0,0,k)}] \otimes m^*[\C_{(0,0,k)}] - [(Y \times \ker(p^m)) \otimes \C_{-(0,0,k)}] \otimes [\C_{(0,0,k)}].
\end{equation}

The lower horizontal map sends $[F^m]$ to
\[
[\{\ker(p^m \circ \, F^m_{y})\}_{y\in Y} \otimes \C_{-(0,0,k)}] \otimes [\C_{(0,0,k)}] - [(Y \times \ker(p^m))  \otimes \C_{-(0,0,k)}] \otimes [\C_{(0,0,k)}]
\]
which is equal to
\begin{equation}\label{sink}
\begin{array}{c}
[\{\ker(\theta_{-m} \circ p \circ F_{m\cdot y} \circ \theta_{m})\}_{y\in Y} \otimes \C_{-(k \phi(m), k I(m),k)}] \otimes [\C_{(k \phi(m), k I(m),k)}] \\ - [(Y \times \ker(p^m)) \otimes \C_{-(0,0,k)}] \otimes [\C_{(0,0,k)}].
\end{array}
\end{equation}
Again using the formula in Remark \ref{calc}, we can rewrite \eqref{sink} as
\[
m^* [\{\ker(p\circ \, F_{ y})\}_{y\in m \cdot Y} \otimes \C_{-(0,0,k)}] \otimes m^*[\C_{(0,0,k)}] - [(Y \times \ker(p^m))  \otimes \C_{-(0,0,k)}] \otimes [\C_{(0,0,k)}]
\]
which is equal to \eqref{grout}. This is what we wanted to show.
\end{proof}

\begin{theorem}\label{theo}
There is an isomorphism of $\Z/2\Z$-graded $\O_{C_T}$-modules
\[
^k\F^*_T(X) \cong \F^*_T(X) \otimes_{\O_{C_T}} \L^k
\]
natural in $X$.
\end{theorem}

\begin{proof}
There are natural isomorphisms of $\O_{C_T}$-modules
\[
\begin{array}{rcl}
((\psi_T)_* \, ^k\K_{\widetilde{LT}}(LX))^{\check{T}} &\cong& ((\psi_T)_* (\K_{LT}(LX) \otimes_{\O_{\D^\times \times T_\C}} \, ^k\O_{\D^\times \times T_\C}))^{\check{T}} \\
&\cong &  ((\psi_T)_* \K_{LT}(LX))^{\check{T}} \otimes_{\O_{C_T}}  ((\psi_T)_* \, ^k\O_{\D^\times \times T_\C})^{\check{T}},
\end{array}
\]
where first isomorphism follows directly from Proposition \ref{obg}, and the second isomorphism holds because the $\check{T}$-action on $\D^\times \times T_\C$ is free. This yields the isomorphism of the theorem in degree zero. By Remark \ref{harold}, we can formally extend to the $\Z/2\Z$-graded isomorphism.
\end{proof}

Combining Theorem \ref{theo} with the character map of Theorem \ref{character} gives us the following corollary. Note that here $\L^k_q$ denotes the restriction of $\L^k$ to $C_{T,q}$.

\begin{corollary}\label{wheat}
There is an isomorphism cohomology theories 
\[
^k\F^*_{T,q} \cong  (\sigma_{T,\tau})_* (\G^*_{T,\tau}) \otimes_{\O_{C_{T,q}}} \L^k_q
\]
defined on the full subcategory of equivariantly formal, finite $T$-CW complexes, and taking values in $\Z/2\Z$-graded $\O_{C_{T,q}}$-modules. 
\end{corollary}

%\subfile{Comp} Comparison of the two local descriptions. Equivariant Chern character. Transgression from the loop group of a torus to a categorical torus. Why can't we go higher than two loops? Because there are no discrete subgroups of $\C$ of rank strictly greater than 2. This is evident in our proof (cf. Etingof's note on Frenkel, ch. 3 of http://www.math.stonybrook.edu/frenkel60/Frenkel/Poster2/igorwork.pdf). Another way to think of this - the two models are obtained by successively 'unravelling' the two loops of a topological elliptic curve. Unravelling the first loop, we remember the rank one subgroup that arises from the unravelling. We think of the rank one subgroup as the image of the dual of $\bT$. The effect of the unravelling is to introduce an extra loop corresponding to $\bT$.

%Another note - the localisation theorems show that we only care about 'ghost loops' (even better, just those ghost loops which are orbits of one-dimensional subgroups of $T$). So we can restrict attention to those loops which are compatible with whatever structure is on $X$ (e.g. smooth loops), since the $T$-action on $X$ will be compatible with that structure.

%\begin{appendix}
%\subfile{EEC}

%\end{appendix}

\end{document}